\documentclass[a4paper, 12pt]{article}
\pdfoutput=1

\usepackage[margin=25mm]{geometry}

\usepackage[utf8]{inputenc}
\usepackage[T1]{fontenc}
\usepackage[british]{babel}
\usepackage[style=british]{csquotes}
\usepackage[en-GB]{datetime2}
\DTMlangsetup[en-GB]{ord=omit}
\usepackage{xcolor}

\usepackage{amsthm, thmtools}
\usepackage{mathtools}
\usepackage{titlesec} 
\usepackage{hyperref}
\hypersetup{
    bookmarksnumbered,
    colorlinks,
    linkcolor={cyan!50!blue},
    citecolor={cyan!50!blue},
    urlcolor={cyan!50!blue}
}
\usepackage[capitalise,nameinlink]{cleveref}

\usepackage{tikz}
\usepackage{tikz-cd}

\usepackage[lining,tabular,scale=.9]{FiraSans}
\usepackage[tt=false,semibold]{libertinus}
\usepackage[libertine,timesmathacc,smallerops]{newtxmath}
\usepackage[lining,scaled=.8]{FiraMono}
\usepackage[cal=boondox,frak=euler]{mathalpha}

\DeclareFontEncoding{MDA}{}{}
\DeclareSymbolFont{mathdesignA}{MDA}{mdput}{m}{n}
\SetSymbolFont{mathdesignA}{bold}{MDA}{mdput}{b}{n}
\DeclareSymbolFontAlphabet{\mathbb}{mathdesignA}

\tikzcdset{
    arrow style=tikz,
    arrows={/tikz/line width=.6pt},
    diagrams={>={Straight Barb[scale=0.67]}}
}

\DeclareFontFamily{OMX}{MnSymbolE}{}
\DeclareSymbolFont{MnLargeSymbols}{OMX}{MnSymbolE}{m}{n}
\SetSymbolFont{MnLargeSymbols}{bold}{OMX}{MnSymbolE}{b}{n}
\DeclareFontShape{OMX}{MnSymbolE}{m}{n}{
    <-6>  MnSymbolE5
   <6-7>  MnSymbolE6
   <7-8>  MnSymbolE7
   <8-9>  MnSymbolE8
   <9-10> MnSymbolE9
  <10-12> MnSymbolE10
  <12->   MnSymbolE12
}{}
\DeclareFontShape{OMX}{MnSymbolE}{b}{n}{
    <-6>  MnSymbolE-Bold5
   <6-7>  MnSymbolE-Bold6
   <7-8>  MnSymbolE-Bold7
   <8-9>  MnSymbolE-Bold8
   <9-10> MnSymbolE-Bold9
  <10-12> MnSymbolE-Bold10
  <12->   MnSymbolE-Bold12
}{}
\DeclareMathDelimiter{[}{\mathopen}{MnLargeSymbols}{'000}{MnLargeSymbols}{'000}
\DeclareMathDelimiter{]}{\mathclose}{MnLargeSymbols}{'005}{MnLargeSymbols}{'005}
\DeclareMathDelimiter{\llbr}{\mathopen}{MnLargeSymbols}{'102}{MnLargeSymbols}{'102}
\DeclareMathDelimiter{\rrbr}{\mathclose}{MnLargeSymbols}{'107}{MnLargeSymbols}{'107}

\usepackage{bm}
\usepackage{cases}
\usepackage{accents}
\usepackage{adjustbox}

\usepackage{setspace}
\usepackage{subdepth} 

\newcommand{\initlengths}{%
    \setlength{\abovedisplayshortskip}{3pt plus 9pt minus 3pt}%
    \setlength{\belowdisplayshortskip}{9pt plus 9pt minus 9pt}%
    \setlength{\abovedisplayskip}{9pt plus 9pt minus 9pt}%
    \setlength{\belowdisplayskip}{9pt plus 9pt minus 9pt}%
    \hfuzz 1pt%
    \tolerance 400
}

\usepackage{titling}
\pretitle{\vspace{0pt}\setstretch{1.05}\begin{center}\LARGE\LibertinusDisplay\fontdimen2\font=0.3em} 
\posttitle{\end{center}\vspace{6pt}}

\numberwithin{paragraph}{subsection}
\newcommand{\parafont}{\bfseries}
\newcommand{\parasep}{9pt plus 3pt minus 3pt}

\titleformat{\section}{\Large\LibertinusDisplay}{\thesection}{1em}{}
\titleformat{\subsection}{\large\firamedium\boldmath}{\thesubsection}{1em}{}
\titleformat{\paragraph}[runin]{\parafont}{\theparagraph.}{.33em}{\normalfont\bfseries\boldmath}
\titlespacing{\paragraph}{0pt}{\parasep}{.5em}

\usepackage{tocloft}

\renewenvironment{abstract}{%
    \centering\begin{minipage}{.85\textwidth}%
    \setlength{\parindent}{1.5em}%
    \centerline{\large\firamedium\abstractname}%
    \par\vspace{12pt}%
}{\end{minipage}\par\vspace{3pt}}

\makeatletter
\newcommand*{\@parabookmark}{%
  \pdfbookmark[3]{%
    \theparagraph
    \ifx\@currentlabelname\@empty
    \else
      .\space\@currentlabelname%
    \fi
  }{\theparagraph}
}
\newcommand*{\@thmbookmark}{%
  \pdfbookmark[3]{%
    \theparagraph.\space\thmt@thmname
    \ifx\@currentlabelname\@empty
    \else
      .\space\@currentlabelname%
    \fi
  }{\theparagraph}
}
\newcommand*{\parabookmark}{\@parabookmark}
\newcommand*{\thmbookmark}{\@thmbookmark}
\newcommand*{\suppressparabookmarks}{%
  \renewcommand*{\parabookmark}{}%
  \renewcommand*{\thmbookmark}{}%
}
\newcommand*{\resumeparabookmarks}{%
  \renewcommand*{\parabookmark}{\@parabookmark}%
  \renewcommand*{\thmbookmark}{\@thmbookmark}%
}

\declaretheoremstyle[
    spaceabove=\parasep, spacebelow=\parasep,
    postheadspace=.5em,
    postheadhook=\thmbookmark,
    headfont=\normalfont\bfseries,
    headpunct={},
    headformat={\NUMBER.\@\ \NAME.\@\NOTE},
    notefont=\normalfont\bfseries\boldmath,
    notebraces={}{.},
    bodyfont=\itshape,
]{theorem}
\declaretheoremstyle[
    spaceabove=\parasep, spacebelow=\parasep,
    postheadspace=.5em,
    headfont=\normalfont\bfseries,
    headpunct={},
    headformat={\NAME.\@\NOTE},
    notefont=\normalfont\bfseries\boldmath,
    notebraces={}{.},
    bodyfont=\itshape,
]{theorem*}
\declaretheoremstyle[
    spaceabove=\parasep, spacebelow=\parasep,
    postheadspace=.5em,
    postheadhook=\thmbookmark,
    headfont=\normalfont\bfseries,
    headpunct={},
    headformat={\NUMBER.\@\ \NAME.\@\NOTE},
    notefont=\normalfont\bfseries\boldmath,
    notebraces={}{.},
]{definition}
\declaretheoremstyle[
    spaceabove=\parasep, spacebelow=\parasep,
    postheadspace=.5em,
    postheadhook=\parabookmark,
    headfont=\normalfont\bfseries,
    headpunct={},
    headformat={\NUMBER.\@\NOTE},
    notefont=\normalfont\bfseries\boldmath,
    notebraces={}{.},
]{para}

\renewenvironment{proof}[1][\proofname]{\par
    \pushQED{\qed}%
    \normalfont\trivlist
    \item[\hskip\labelsep\bfseries #1\@addpunct{.}]\ignorespaces
}{%
    \popQED\endtrivlist\@endpefalse
}
\makeatother

\declaretheorem[sibling=paragraph, style=para, refname={\S,\S\S}]{para}
\declaretheorem[sibling=paragraph, style=theorem, name=Theorem]{theorem}
\declaretheorem[sibling=paragraph, style=theorem, name=Lemma]{lemma}
\declaretheorem[sibling=paragraph, style=theorem, name=Corollary]{corollary}

\declaretheorem[numbered=no, style=theorem*, name=Theorem]{theorem*}
\declaretheorem[numbered=no, style=theorem*, name=Lemma]{lemma*}

\declaretheorem[sibling=paragraph, style=definition, name=Example]{example}

\numberwithin{equation}{paragraph}

\crefdefaultlabelformat{#2{\upshape#1}#3}

\crefformat{equation}{#2{\upshape(#1)}#3}
\crefrangeformat{equation}{{\upshape#3(#1)#4--#5(#2)#6}}
\crefmultiformat{equation}{#2{\upshape(#1)}#3}{ and #2{\upshape(#1)}#3}{, #2{\upshape(#1)}#3}{, and~#2{\upshape(#1)}#3}

\crefformat{enumi}{#2{\upshape#1}#3}
\crefrangeformat{enumi}{{\upshape#3#1#4--#5#2#6}}
\crefmultiformat{enumi}{#2{\upshape#1}#3}{ and #2{\upshape#1}#3}{, #2{\upshape#1}#3}{, and~#2{\upshape#1}#3}

\crefformat{section}{#2{\upshape\S#1}#3}
\crefrangeformat{section}{{\upshape#3\S\S#1#4--#5#2#6}}
\crefmultiformat{section}{{\upshape#2\S\S#1#3}}{ and {\upshape#2#1#3}}{, {\upshape#2#1#3}}{, and~{\upshape#2#1#3}}
\crefformat{subsection}{#2{\upshape\S#1}#3}
\crefrangeformat{subsection}{{\upshape#3\S\S#1#4--#5#2#6}}
\crefmultiformat{subsection}{{\upshape#2\S\S#1#3}}{ and {\upshape#2#1#3}}{, {\upshape#2#1#3}}{, and~{\upshape#2#1#3}}
\crefformat{paragraph}{#2{\upshape\S#1}#3}
\crefrangeformat{paragraph}{{\upshape#3\S\S#1#4--#5#2#6}}
\crefmultiformat{paragraph}{{\upshape#2\S\S#1#3}}{ and {\upshape#2#1#3}}{, {\upshape#2#1#3}}{, and~{\upshape#2#1#3}}
\crefformat{para}{#2{\upshape\S#1}#3}
\crefrangeformat{para}{{\upshape#3\S\S#1#4--#5#2#6}}
\crefmultiformat{para}{{\upshape#2\S\S#1#3}}{ and {\upshape#2#1#3}}{, {\upshape#2#1#3}}{, and~{\upshape#2#1#3}}

\crefname{figure}{Figure}{Figures}

\usepackage{enumitem}
\setlist{noitemsep}
\setlist[enumerate]{label=\textnormal{(\roman*)}}

\usepackage[labelsep=period, labelfont={bf}]{caption}

\usepackage[
    backend=biber,
    style=oxnum,
    giveninits,
    maxbibnames=5,
    maxcitenames=5,
    sorting=nyt
]{biblatex}

\DeclareFieldFormat[article]{version}{\IfDecimal{#1}{version~}{}#1}
\DefineBibliographyStrings{english}{
  withafterword = {with an appendix by},
}

\DeclareFieldFormat{pages}{%
  \iffieldundef{bookpagination}%
    {#1}%
    {\mkpageprefix[bookpagination]{#1}}%
}

\newcommand{\calHom}{{\mathcal{H}\mspace{-5mu}\mathit{om}}}

\newcommand{\Gm}{\mathbb{G}_\mathrm{m}}
\newcommand{\heart}{\mathbin{\heartsuit}}
\newcommand{\subXsd}[1][\mathcal{X}]{{\scriptstyle{#1}^{\raisebox{-.2ex}{$\scriptscriptstyle\smash{\mathrm{sd}}$}}}}
\newcommand{\vdim}{\operatorname{vdim}}

\newcommand{\longsimto}{\mathrel{\overset{\smash{\raisebox{-.8ex}{$\sim$}}\mspace{3mu}}{\longrightarrow}}}

\newcommand{\simto}{\mathrel{\overset{\smash{\raisebox{-.8ex}{$\sim$}}\mspace{3mu}}{\to}}}
\newcommand{\simTo}{\mathrel{\overset{\smash{\raisebox{-.8ex}{$\sim$}}\mspace{3mu}}{\Rightarrow}}}

\newcommand{\leftsubstack}[2][6em]{\substack{\makebox[#1][l]{\scriptsize$\begin{aligned}#2\end{aligned}$}}}

\renewcommand{\geq}{\geqslant}
\renewcommand{\leq}{\leqslant}

\makeatletter
\def\big#1{{\hbox{$\left#1\vbox to10\p@{}\right.\n@space$}}}
\makeatother

\title{Orthosymplectic Donaldson--Thomas theory}
\author{Chenjing Bu}
\date{}

\begin{document}

\initlengths

\maketitle

\begin{abstract}
    We construct and study Donaldson--Thomas invariants
counting orthogonal and symplectic objects
in linear categories,
which are a generalization of the usual Donaldson--Thomas invariants
from the structure groups $\mathrm{GL} (n)$
to the groups $\mathrm{O} (n)$ and $\mathrm{Sp} (2n)$,
and a special case of the intrinsic Donaldson--Thomas theory
developed by the author, Halpern-Leistner, Ibáñez Núñez, and Kinjo
\mbox{\cite{epsilon-i,epsilon-ii,epsilon-iii}}.
Our invariants are defined using the motivic Hall algebra
and its orthosymplectic analogue,
the motivic Hall module.
We prove wall-crossing formulae for our invariants,
which relate the invariants with respect to different stability conditions.

As examples, we define Donaldson--Thomas invariants counting
orthogonal and symplectic perfect complexes on a Calabi--Yau threefold,
and Donaldson--Thomas invariants counting self-dual representations of
a self-dual quiver with potential.
In the case of quivers, we compute the invariants explicitly in some cases.
We also define a motivic version of Vafa--Witten invariants
counting orthogonal and symplectic Higgs complexes
on a class of algebraic surfaces.

\end{abstract}

\clearpage
\vspace*{12pt}
{
    \hypersetup{linkcolor=black}
    \tableofcontents
}

\clearpage
\suppressparabookmarks
\section{Introduction}

\addtocounter{subsection}{1}

\begin{para}[Overview]
    \label{para-intro-overview}
    The theory of
    \emph{Donaldson--Thomas}~(DT) \emph{invariants}
    has been a central topic in enumerative geometry,
    initiated in the works of
    \textcite{donaldson-thomas-1998} and
    \textcite{thomas-2000-dt},
    and further developed by
    \textcite{joyce-song-2012},
    \textcite{kontsevich-soibelman-motivic-dt},
    and many others.

    The usual DT~invariants are, roughly speaking, virtual counts of
    semistable objects in $3$-Calabi--Yau linear categories,
    such as the category of coherent sheaves on a Calabi--Yau threefold,
    or the category of representations of a quiver with potential.
    These were first constructed by
    \textcite{joyce-song-2012},
    based on a formalism of
    \textcite{
        joyce-2006-configurations-i,
        joyce-2007-configurations-ii,
        joyce-2007-configurations-iii,
        joyce-2008-configurations-iv,
        joyce-2007-stack-functions}
    for constructing motivic enumerative invariants
    for abelian categories,
    and a similar approach was taken by
    \textcite{kontsevich-soibelman-motivic-dt}.

    However, until recently,
    it was not known how to extend this theory
    beyond the linear case.
    A step in this direction was taken by
    \textcite{bu-self-dual-i},
    who developed a theory
    of motivic enumerative invariants
    for moduli stacks of orthosymplectic objects,
    parallel to Joyce's formalism,
    and partially inspired by the works of
    \textcite{young-2016-hall-module,young-2015-self-dual,young-2020-quiver}
    on self-dual quivers.

    This theory was later generalized to \emph{intrinsic DT theory}
    by the author, Halpern-Leistner, Ibáñez Núñez, and Kinjo
    \cite{epsilon-i,epsilon-ii,epsilon-iii},
    which is a new framework for enumerative geometry
    that is intrinsic to the moduli stack,
    and applies to all algebraic stacks satisfying mild assumptions.
    The new framework also makes it possible to significantly simplify
    the original work \cite{bu-self-dual-i}.

    The present paper supersedes the author's work \cite{bu-self-dual-i}
    by updating and simplifying the theory
    using the new framework of \cite{epsilon-i,epsilon-ii,epsilon-iii}.

    The main goal of this paper is to construct DT invariants
    counting \emph{self-dual objects},
    or orthosymplectic objects,
    in self-dual $3$-Calabi--Yau linear categories.
    Self-dual objects typically do not form linear categories,
    so the usual linear DT theory does not apply.

    As applications, we construct
    DT invariants counting
    orthosymplectic representations of self-dual quivers with potential,
    DT invariants counting
    orthosymplectic perfect complexes on
    curves and Calabi--Yau threefolds,
    and also a motivic version of Vafa--Witten invariants
    counting orthosymplectic Higgs complexes
    on a class of algebraic surfaces.

    Orthosymplectic DT invariants
    are related to counting D-branes
    in string theories on Calabi--Yau \emph{$3$-orientifolds},
    discussed in, for example,
    \textcite[\S5.2]{witten-1998-d-branes},
    \textcite{diaconescu-2007-orientifolds},
    and \textcite{hori-walcher-2008-orientifolds}.
\end{para}

\begin{para}[The setting]
    \label{para-intro-setting}
    The basic setting of our construction is as follows.
    We start with a \emph{self-dual linear category}~$\mathcal{A}$,
    that is a linear category equipped with a contravariant involution
    \begin{equation*}
        (-)^\vee \colon \mathcal{A} \longsimto \mathcal{A}^\mathrm{op} \ .
    \end{equation*}
    For example, $\mathcal{A}$ could be the category of
    vector bundles on a smooth projective curve,
    where the involution is given by taking the dual bundle.
    See \cref{subsec-sd-cat} for more details and examples.

    We then consider the fixed locus of the involution,
    which is a groupoid~$\mathcal{A}^\mathrm{sd}$,
    called the groupoid of \emph{self-dual objects}.
    Explicitly, an object of~$\mathcal{A}^\mathrm{sd}$ is a pair $(x, \phi)$,
    where~$x$ is an object of~$\mathcal{A}$,
    and $\phi \colon x \simto x^\vee$ is an isomorphism
    satisfying $\phi = \phi^\vee$.
    For example, in the case of vector bundles,
    $\mathcal{A}^\mathrm{sd}$ consists of
    either orthogonal or symplectic vector bundles,
    depending on the choice of a sign $\varepsilon = \pm 1$
    when identifying a vector bundle with its double dual,
    which is a part of the data of the involution.

    We also assume that~$\mathcal{A}$ is equipped with
    a moduli stack~$\mathcal{X}$ of objects,
    so the involution~$(-)^\vee$ defines an involution on~$\mathcal{X}$.
    Its fixed locus~$\mathcal{X}^\mathrm{sd}$
    is the moduli stack of objects in~$\mathcal{A}^\mathrm{sd}$.

    In fact, our theory does not essentially use
    the categories~$\mathcal{A}$ and~$\mathcal{A}^\mathrm{sd}$,
    and only depends on the stack~$\mathcal{X}$,
    equipped with its linear structure and involution.
\end{para}

\begin{para}[Harder--Narasimhan filtrations]
    Recall that given a \emph{stability condition}~$\tau$
    on an abelian category~$\mathcal{A}$, which, for simplicity,
    we assume is given by a slope $\tau (x) \in \mathbb{R}$
    for each non-zero object $x \in \mathcal{A}$,
    satisfying certain conditions,
    there is the notion of \emph{$\tau$-semistable objects},
    and every object~$x$ has a unique
    \emph{Harder--Narasimhan}~(HN) \emph{filtration}
    \begin{equation*}
        \begin{tikzcd}[sep=small]
            \mathllap{0 = {}} x_0 \ar[r, hook] &
            x_1 \ar[r, hook] \ar[d, two heads] &
            x_2 \ar[r, hook] \ar[d, two heads] &
            \cdots \ar[r, hook] &
            x_k \mathrlap{{} = x} \ar[d, two heads] \\
            & y_1 & y_2 & & y_k \rlap{ ,}
        \end{tikzcd}
    \end{equation*}
    with each quotient $y_i = x_i / x_{i - 1}$
    non-zero and semistable, such that
    $\tau (y_1) > \cdots > \tau (y_k)$.

    Now, suppose that~$\mathcal{A}$ is self-dual,
    and that~$\tau$ is compatible with the self-dual structure,
    in that $\tau (x^\vee) = - \tau (x)$ for all non-zero objects~$x$.
    Then, for a non-zero self-dual object $(x, \phi)$,
    we necessarily have $\tau (x) = 0$,
    and the self-dual structure~$\phi$ induces isomorphisms
    $y_i \simto y_{\smash{k+1-i}}^\vee$ of its HN~factors.
    In particular, if~$k$ is odd,
    the middle piece~$y_{(k + 1) / 2}$
    admits an induced self-dual structure.
    For convenience, when~$k$ is even,
    we sometimes think of it as having the zero self-dual object as the middle piece.

    Therefore, heuristically speaking,
    if we think of objects of~$\mathcal{A}$
    as composed of semistable objects via HN~filtrations,
    we should think of an object of~$\mathcal{A}^\mathrm{sd}$
    as composed of a series of semistable objects of~$\mathcal{A}$,
    which are those in the left half of the HN filtration,
    together with a single semistable self-dual object
    in~$\mathcal{A}^\mathrm{sd}$ in the middle;
    the factors in the right half are dual to those on the left,
    and do not contain new information.
\end{para}

\begin{para}[Orthosymplectic modules]
    \label{para-intro-osp-modules}
    The above phenomenon suggests that the collection of
    orthosymplectic objects can be viewed as
    a module for the collection of linear objects in some sense.
    This can be made precise in several different ways.

    For example, if we consider the symmetric monoidal groupoid
    $(\mathcal{A}^\simeq, \oplus, 0)$,
    where $\mathcal{A}^\simeq$ is the underlying groupoid of~$\mathcal{A}$,
    then the operation
    $\oplus^\mathrm{sd} \colon \mathcal{A}^\simeq \times \mathcal{A}^\mathrm{sd}
    \to \mathcal{A}^\mathrm{sd}$
    sending $(x, y)$ to $x \oplus y \oplus x^\vee$,
    equipped with the obvious self-dual structure,
    establishes~$\mathcal{A}^\mathrm{sd}$ as a module for~$\mathcal{A}^\simeq$.
    Similarly, the corresponding operation
    $\oplus^\mathrm{sd} \colon \mathcal{X} \times \mathcal{X}^\mathrm{sd}
    \to \mathcal{X}^\mathrm{sd}$
    on moduli stacks establishes~$\mathcal{X}^\mathrm{sd}$
    as a module for the commutative monoid stack~$(\mathcal{X}, \oplus, 0)$.

    As another example,
    consider the \emph{ring of motives}
    $\mathbb{M} (\mathcal{X})$ over~$\mathcal{X}$,
    which we define in \cref{subsec-motives}.
    It has the structure of the \emph{motivic Hall algebra},
    introduced by \textcite{joyce-2007-configurations-ii},
    whose multiplication~$*$
    is roughly given by parametrizing all possible extensions
    of given objects.
    Note that this is a different multiplication
    from the one given by the ring structure on the ring of motives.
    In the orthosymplectic case, the ring of motives
    $\mathbb{M} (\mathcal{X}^\mathrm{sd})$
    is a module for the motivic Hall algebra
    $\mathbb{M} (\mathcal{X})$,
    which we call the \emph{motivic Hall module}.
    The module structure~$\diamond$,
    which we define in \cref{subsec-hall},
    is roughly given by parametrizing
    three-step self-dual filtrations,
    with given graded pieces $x, y, x^\vee$,
    with~$x$ an ordinary object and~$y$ a self-dual object.

    Other similar constructions include
    \emph{cohomological Hall modules}
    considered by \textcite{young-2020-quiver},
    and twisted modules for Joyce vertex algebras
    introduced by \textcite{bu-self-dual-ii}.
\end{para}

\begin{para}[Epsilon motives]
    \label{para-intro-epsilon}
    The first main construction of this paper
    is that of \emph{epsilon motives}
    for the moduli stack~$\mathcal{X}^\mathrm{sd}$,
    which we present in \cref{subsec-epsilon},
    parallel to the construction of
    \textcite{
        joyce-2006-configurations-i,
        joyce-2007-configurations-ii,
        joyce-2007-configurations-iii,
        joyce-2008-configurations-iv,
        joyce-2007-stack-functions}
    in the linear case.

    Given a stability condition~$\tau$,
    for each non-zero connected component
    $\mathcal{X}_\alpha \subset \mathcal{X}$,
    corresponding to an element
    $\alpha \in \uppi_0 (\mathcal{X}) \setminus \{ 0 \}$,
    Joyce defined the \emph{epsilon motive}
    \begin{equation*}
        \epsilon_\alpha (\tau) \in \mathbb{M} (\mathcal{X}) \ ,
    \end{equation*}
    living in the ring of motives over~$\mathcal{X}$,
    and supported on the semistable locus
    $\mathcal{X}^\mathrm{ss}_\alpha (\tau) \subset \mathcal{X}_\alpha$.
    It can be thought of as an interpolation
    between the stable locus and the semistable locus.

    The epsilon motive~$\epsilon_\alpha (\tau)$
    satisfies an important property called the \emph{no-pole theorem},
    originally proved by \textcite[Theorem~8.7]{joyce-2007-configurations-iii},
    which ensures that it has a well-defined Euler characteristic,
    which can be used to define DT invariants.
    In general, Euler characteristics of stacks are ill-behaved,
    as for example, any reasonable definition would give
    $\chi (* / \mathbb{G}_\mathrm{m}) = 1/0 = \infty$.
    The no-pole theorem ensures that for epsilon motives,
    this type of divergence does not occur
    after multiplying by the motive
    $[\mathbb{G}_\mathrm{m}] = \mathbb{L} - 1$,
    which roughly corresponds to removing the copy of~$\mathbb{G}_\mathrm{m}$
    from the stabilizer groups of all non-zero points of~$\mathcal{X}$,
    corresponding to scalar automorphisms of objects in~$\mathcal{A}$.

    In the orthosymplectic case,
    we present a parallel construction.
    Given a self-dual stability condition~$\tau$,
    for each connected component
    $\theta \in \uppi_0 (\mathcal{X}^\mathrm{sd})$,
    we define the \emph{epsilon motive}
    \begin{equation*}
        \epsilon^\mathrm{sd}_\theta (\tau) \in \mathbb{M} (\mathcal{X}^\mathrm{sd}) \ ,
    \end{equation*}
    supported on the semistable locus
    $\mathcal{X}^{\smash{\mathrm{sd, ss}}}_\theta (\tau)
    \subset \mathcal{X}^\mathrm{sd}_\theta$.
    It satisfies an analogous no-pole theorem,
    and will be used to define orthosymplectic DT invariants.

    The construction of the epsilon motives
    $\epsilon^\mathrm{sd}_\theta (\tau)$
    is a special case of a general construction in
    \emph{intrinsic DT theory} by \textcite[\S5.2]{epsilon-ii},
    and we build upon the general framework
    by specifying explicit combinatorial data,
    called \emph{stability measures} there,
    used to define the epsilon motives.
    We also use the motivic Hall module structure,
    described in \cref{para-intro-osp-modules},
    to make the construction more explicit in our setting.
\end{para}

\begin{para}[DT invariants]
    We now consider the case when~$\mathcal{X}$
    is either a smooth stack,
    or a \emph{$(-1)$-shifted symplectic stack}
    in the sense of \textcite{pantev-toen-vaquie-vezzosi-2013}.
    The latter case often occurs when~$\mathcal{A}$
    is a $3$-Calabi--Yau category,
    such as the category of coherent sheaves on a Calabi--Yau threefold,
    or more precisely, a modification of this category that is self-dual,
    which we describe in \cref{subsec-coh}.

    In this case,
    following \textcite{joyce-song-2012},
    we define the \emph{DT invariant}
    $\mathrm{DT}_\alpha (\tau) \in \mathbb{Q}$
    for~$\mathcal{X}$
    and a class $\alpha \in \uppi_0 (\mathcal{X}) \setminus \{ 0 \}$
    as a weighted Euler characteristic
    \begin{equation*}
        \mathrm{DT}_\alpha (\tau) =
        \int_{\mathcal{X}} {}
        (1 - \mathbb{L}) \cdot
        \epsilon_\alpha (\tau) \cdot \nu_{\mathcal{X}} \, d \chi \ ,
    \end{equation*}
    where~$\nu_{\mathcal{X}}$ is the \emph{Behrend function} of~$\mathcal{X}$,
    a constructible function on~$\mathcal{X}$
    introduced by \textcite{behrend-2009-dt}
    to capture the virtual geometry of~$\mathcal{X}$.
    Multiplying by the factor $(1 - \mathbb{L})$ corresponds to
    removing a copy of~$\mathbb{G}_\mathrm{m}$
    from all stabilizer groups in~$\mathcal{X}$,
    as we mentioned in~\cref{para-intro-epsilon}.

    Using our self-dual epsilon motives,
    in \cref{subsec-dt},
    we define \emph{self-dual DT invariants}
    $\mathrm{DT}^\mathrm{sd}_\theta (\tau) \in \mathbb{Q}$
    for~$\mathcal{X}^\mathrm{sd}$ and a class
    $\theta \in \uppi_0 (\mathcal{X}^\mathrm{sd})$
    by a similar weighted Euler characteristic
    \begin{equation*}
        \mathrm{DT}^\mathrm{sd}_\theta (\tau) =
        \int_{\mathcal{X}^\mathrm{sd}} {}
        \epsilon^\mathrm{sd}_\theta (\tau) \cdot \nu_{\mathcal{X}^\mathrm{sd}} \, d \chi \ ,
    \end{equation*}
    where we no longer need the factor $(1 - \mathbb{L})$,
    since a general point in~$\mathcal{X}^\mathrm{sd}$
    does not have $\mathbb{G}_\mathrm{m}$ in its stabilizer,
    as scaling a self-dual object does not preserve its self-dual structure.

    We also define motivic enhancements of these invariants,
    called \emph{motivic DT invariants},
    and denoted by $\mathrm{DT}^\mathrm{mot}_\alpha (\tau)$
    and $\mathrm{DT}^{\smash{\mathrm{mot,sd}}}_\theta (\tau)$.
    They live in a ring of monodromic motives over the base field,
    and in the $(-1)$-shifted symplectic case,
    they also depend on choices of orientations
    of the stacks~$\mathcal{X}$ and~$\mathcal{X}^\mathrm{sd}$.
    The linear case was due to
    \textcite{kontsevich-soibelman-motivic-dt}.
\end{para}

\begin{para}[Wall-crossing formulae]
    In \cref{sec-wcf},
    we prove \emph{wall-crossing formulae} for our DT invariants,
    which relate the DT invariants of~$\mathcal{X}^\mathrm{sd}$
    for different stability conditions~$\tau$,
    generalizing the wall-crossing formulae
    of \textcite{joyce-song-2012}.
    They are of the form
    \begin{align}
        \label{eq-intro-wcf}
        \mathrm{DT}^\mathrm{sd}_\theta (\tau_-)
        =
        \sum_{ \leftsubstack[6em]{
            & n \geq 0; \
            \alpha_1, \dotsc, \alpha_n \in \uppi_0 (\mathcal{X}) \setminus \{ 0 \}; \,
            \rho \in \uppi_0 (\mathcal{X}^\mathrm{sd}) \colon
            \\[-1ex]
            & \theta =
            \alpha_1 + \alpha_1^\vee + \cdots +
            \alpha_n + \alpha_n^\vee + \rho
        } } {}
        C (\alpha_1, \dotsc, \alpha_n, \rho; \tau_+, \tau_-) \cdot
        \mathrm{DT}_{\alpha_1} (\tau_+) \cdots
        \mathrm{DT}_{\alpha_n} (\tau_+) \cdot
        \mathrm{DT}^\mathrm{sd}_{\rho} (\tau_+) \ ,
        \raisetag{3ex}
    \end{align}
    where $\tau_\pm$ are self-dual stability conditions
    satisfying certain conditions,
    $C (\alpha_1, \dotsc, \alpha_n, \rho; \tau_+, \tau_-)$
    is a rational combinatorial coefficient,
    and the sum has only finitely many non-zero terms.
    A similar formula holds for the motivic DT invariants.

    A key ingredient in proving the wall-crossing formula
    is the \emph{motivic integral identity} for Behrend functions.
    In the linear case, the identity was proved by
    \textcite[Theorem~5.11]{joyce-song-2012}
    in a numerical form, and conjectured by
    \textcite[Conjecture~4]{kontsevich-soibelman-motivic-dt}
    in a stronger motivic form,
    which was proved by \textcite{le-2015-integral}.
    In the general case, the identity was formulated
    and proved by \textcite[Theorem~4.2.2]{bu-integral}.

    Wall-crossing formulae govern the structure of the DT invariants,
    and can be used to compute them explicitly in some cases.
    The same wall-crossing structure is also satisfied by
    other types of enumerative invariants in the orthosymplectic case,
    such as Joyce's~\cite{joyce-homological}
    homological enumerative invariants,
    extended by \textcite{bu-self-dual-ii}
    to the case of orthosymplectic quivers,
    where the wall-crossing formula~\cref{eq-intro-wcf}
    plays a central role.
\end{para}

\begin{para}[Application to quivers]
    In \cref{subsec-quiver-dt},
    we define DT invariants for
    \emph{self-dual quivers with potential},
    which are an orthosymplectic analogue of
    the usual DT theory for quivers with potential,
    studied in \textcite[Ch.~7]{joyce-song-2012}
    and \textcite[\S8]{kontsevich-soibelman-motivic-dt}.

    Self-dual quivers were first introduced by
    \textcite{derksen-weyman-2002},
    and studied by
    \textcite{young-2015-self-dual,young-2016-hall-module,young-2020-quiver}
    in the context of DT theory.
    These works were a main early source of inspiration
    for our work.

    We also provide an algorithm for computing DT invariants
    for self-dual quivers where the potential is zero,
    and present some numerical results.
    We mention a relation between self-dual quivers
    and orthosymplectic coherent sheaves
    in \cref{eg-dt-p1}.
\end{para}

\begin{para}[Application to sheaves]
    In \cref{subsec-threefolds},
    we define DT invariants counting
    orthosymplectic perfect complexes
    on Calabi--Yau threefolds.

    More precisely,
    we consider Bridgeland stability conditions on the threefold
    that are compatible with the self-dual structure,
    so the abelian category of semistable objects
    of slope~$0$ is self-dual.
    We then apply our theory to this category
    to define DT invariants.
    We also prove wall-crossing formulae
    relating the invariants for different
    Bridgeland stability conditions.

    As mentioned in \cref{para-intro-overview},
    we expect that these invariants are related to
    counting D-branes in orientifold string theories.

    In \cref{subsec-curves,subsec-vw},
    we discuss two variants of this theory,
    DT invariants for curves
    and motivic Vafa--Witten type invariants for
    del~Pezzo, K3, and abelian surfaces.
    These Vafa--Witten invariants are a motivic version and an orthosymplectic analogue
    of the theory of usual Vafa--Witten invariants developed by
    \textcite{tanaka-thomas-2020-vw-i,tanaka-thomas-2018-vw-ii}.
\end{para}

\begin{para}[Acknowledgements]
    The author thanks Dominic Joyce
    for his continuous support throughout this project,
    and for his many valuable comments and suggestions on the paper.
    The author also thanks
    Andrés Ibáñez Núñez
    and Tasuki Kinjo
    for helpful discussions.

    The author was supported by
    the Mathematical Institute, University of Oxford.
\end{para}

\begin{para}[Conventions]
    \label{para-intro-conventions}
    Throughout this paper, we use the following conventions:

    \begin{itemize}
        \item
            We work over a base field~$K$.

        \item
            All \emph{schemes}, \emph{algebraic spaces},
            and \emph{algebraic stacks} over~$K$
            are assumed to be locally of finite type over~$K$
            and have affine diagonal.

        \item
            A \emph{derived algebraic stack} over~$K$
            is a derived stack over~$K$
            that has an open cover by
            \emph{geometric stacks} in the sense of
            \textcite[\S2.2.3]{toen-vezzosi-2008},
            and is assumed locally almost of finite presentation.
    \end{itemize}
\end{para}

\resumeparabookmarks
\section{Categories and stacks}

\subsection{Self-dual categories}
\label{subsec-sd-cat}

\begin{para}
    We begin by introducing a notion of \emph{self-dual linear categories},
    as described in \cref{para-intro-setting}.
    This notion will not be essentially used in our main constructions,
    since we primarily work with moduli stacks of objects in such categories,
    which we will discuss in \cref{subsec-linear-stacks}.
    However, they will be useful in providing motivations,
    as well as in studying examples and applications.
\end{para}

\begin{para}[Self-dual linear categories]
    \label{para-sd-cat}
    A \emph{self-dual $K$-linear category} consists of the following data:
    \begin{itemize}
        \item
            A $K$-linear category $\mathcal{A}$.
        \item
            An equivalence of $K$-linear categories
            $(-)^\vee \colon \mathcal{A} \simto \mathcal{A}^\mathrm{op}$,
            called the \emph{dual functor}.
        \item
            A natural isomorphism
            $\eta \colon (-)^{\vee \vee} \simTo \mathrm{id}_{\mathcal{A}}$,
            such that for any object $x \in \mathcal{A}$, we have
            $\eta_{x^\vee} = (\eta_x^\vee)^{-1} \colon x^{\vee \vee \vee} \simto x^\vee$.
    \end{itemize}
    Given such a category~$\mathcal{A}$,
    a \emph{self-dual object} in~$\mathcal{A}$ is a pair~$(x, \phi)$,
    where $x \in \mathcal{A}$ and $\phi \colon x \simto x^\vee$ is an isomorphism,
    such that $\phi^\vee = \phi \circ \eta_x$:
    \vspace{-3pt}
    \begin{equation*}
        \begin{tikzcd}[row sep={1.8em,between origins}, column sep=2.5em]
            x \ar[dr, "\phi" {pos=.45},
                "\textstyle \sim" {pos=.52, anchor=center, rotate=-15, shift={(0, -1ex)}}] \\
            & x^\vee \rlap{ .} \\
            x^{\vee\vee} \ar[uu, "\eta_x" {pos=.45},
                "\textstyle \sim" {pos=.45, anchor=center, rotate=90, shift={(0, -1ex)}}]
            \ar[ur, "\phi^\vee"' {pos=.35, inner sep=.05em},
                "\textstyle \sim" {pos=.42, anchor=center, rotate=15, shift={(0, .6ex)}}]
        \end{tikzcd}
        \vspace{-3pt}
    \end{equation*}
    We denote by~$\mathcal{A}^\mathrm{sd}$
    the groupoid of self-dual objects in~$\mathcal{A}$,
    whose morphisms are isomorphisms compatible with the self-dual structures.

    More conceptually, ignoring size issues,
    a self-dual $K$-linear category is a fixed point
    of the $\mathbb{Z}_2$-action on the $2$-category of $K$-linear categories,
    given by taking the opposite category.
    A self-dual object in~$\mathcal{A}$ is a fixed point
    of the $\mathbb{Z}_2$-action on the underlying groupoid of~$\mathcal{A}$,
    given by the dual functor~$(-)^\vee$.

    From this viewpoint, we can analogously define
    self-dual categories and objects in the context of higher categories,
    although more coherence data is needed
    if we were to write down the axioms explicitly.
\end{para}

\begin{example}[Vector bundles]
    \label{eg-sd-cat-vb}
    Let $X$ be a $K$-scheme,
    and let $\mathcal{A} = \mathsf{Vect} (X)$
    be the $K$-linear exact category
    of vector bundles on $X$ of finite rank.

    For each choice of a sign $\varepsilon \in \{ \pm 1 \}$,
    there is a self-dual structure
    $(-)^\vee \colon \mathcal{A} \simto \mathcal{A}^\mathrm{op}$
    sending a vector bundle to its dual vector bundle,
    with the natural isomorphism
    $\eta \colon (-)^{\vee\vee} \simTo \mathrm{id}_{\mathcal{A}}$
    given by~$\varepsilon$ times the usual identification.

    A self-dual object in $\mathcal{A}$ is a pair $(E, \phi)$,
    where $E$ is a vector bundle on $X$,
    and $\phi \colon E \simto E^\vee$ is an isomorphism,
    satisfying $\phi^\vee = \phi \circ \eta_E$.
    Equivalently, $\phi$ is a non-degenerate
    symmetric (or~antisymmetric) bilinear form on~$E$
    when $\varepsilon = +1$ (or~$-1$).
    In particular, if $K$ is algebraically closed
    of characteristic~$\neq 2$,
    then self-dual objects of~$\mathcal{A}$
    can be identified with principal
    $\mathrm{O} (n)$-bundles (or~$\mathrm{Sp} (n)$-bundles) on~$X$.
\end{example}

\begin{example}[Self-dual quivers]
    Let $Q$ be a \emph{self-dual quiver},
    that is, a quiver with an involution $\sigma \colon Q \simto Q^\mathrm{op}$,
    where $Q^\mathrm{op}$ is the opposite quiver of $Q$.
    See \cref{subsec-quiver-dt} for details.

    Let $\mathcal{A} = \mathsf{Mod} (K Q)$
    be the $K$-linear abelian category of
    finite-dimensional representations of~$Q$ over~$K$.
    There is a self-dual structure
    $(-)^\vee \colon \mathcal{A} \simto \mathcal{A}^\mathrm{op}$
    sending a representation
    to the representation with the dual vector spaces and dual linear maps.
    This also involves choosing signs when defining
    $\eta \colon (-)^{\vee\vee} \simTo \mathrm{id}_{\mathcal{A}}$, as in the previous example.
    Again, see \cref{subsec-quiver-dt} for details.

    Self-dual objects in $\mathcal{A}$ are called
    \emph{self-dual representations} of $Q$,
    which we think of as analogues of
    orthogonal or symplectic bundles
    in the quiver setting.
\end{example}

\begin{para}[Non-example.\ Coherent sheaves]
    \label{para-non-eg-coh}
    Let~$X$ be a connected, smooth, projective $K$-variety
    of positive dimension,
    and let $\mathcal{A} = \mathsf{Coh} (X)$
    be the abelian category of coherent sheaves on~$X$.

    Then~$\mathcal{A}$ does not admit a self-dual structure.
    This is because~$\mathcal{A}$ is \emph{noetherian},
    meaning that every ascending chain of subobjects
    of a given object stabilizes,
    while it is not \emph{artinian},
    in that there exists an infinite descending chain of subobjects
    $\mathcal{O}_X \supset \mathcal{O}_X (-1) \supset \mathcal{O}_X (-2) \supset \cdots$.
    Since taking the opposite category
    exchanges the properties of
    being noetherian and artinian,
    the category~$\mathcal{A}$
    is not equivalent to~$\mathcal{A}^\mathrm{op}$.

    This problem can be fixed, however,
    by considering the derived category
    $\mathcal{D} = \mathsf{D}^\mathrm{b} \mathsf{Coh} (X)$,
    and taking an alternative heart
    $\mathcal{A}' \subset \mathcal{D}$
    that is compatible with derived duality,
    which can be constructed from Bridgeland stability conditions.
    See \cref{subsec-threefolds} for details.
\end{para}

\begin{para}[Self-dual filtrations]
    \label{para-sd-filt}
    We now discuss a useful construction
    in self-dual linear categories.

    Let~$\mathcal{A}$ be a \emph{self-dual exact category} over~$K$,
    meaning a $K$-linear exact category with a self-dual structure,
    such that the dual functor~$(-)^\vee$
    sends short exact sequences $y \hookrightarrow x \twoheadrightarrow z$
    to short exact sequences $z^\vee \hookrightarrow x^\vee \twoheadrightarrow y^\vee$.

    For an integer~$n \geq 0$,
    define the $K$-linear category~$\mathcal{A}^{(n)}$
    of \emph{$n$-step filtrations} in~$\mathcal{A}$
    where objects are diagrams
    \begin{equation}
        \label{eq-n-step-filt}
        \begin{tikzcd}[sep=small]
            \mathllap{0 = {}} x_0 \ar[r, hook] &
            x_1 \ar[r, hook] \ar[d, two heads] &
            x_2 \ar[r, hook] \ar[d, two heads] &
            \cdots \ar[r, hook] &
            x_n \mathrlap{{} = x} \ar[d, two heads] \\
            & y_1 & y_2 & & y_n \rlap{ ,}
        \end{tikzcd}
    \end{equation}
    with each sequence $x_{i-1} \hookrightarrow x_i \twoheadrightarrow y_i$
    short exact in~$\mathcal{A}$,
    and morphisms are usual morphisms of diagrams.
    Define the \emph{dual filtration} of~\cref{eq-n-step-filt}
    to be the $n$-step filtration
    \begin{equation}
        \label{eq-n-step-filt-dual}
        \begin{tikzcd}[sep=small]
            \mathllap{0 = {}} (x / x_n)^\vee \ar[r, hook] &
            (x / x_{n-1})^\vee \ar[r, hook] \ar[d, two heads] &
            (x / x_{n-2})^\vee \ar[r, hook] \ar[d, two heads] &
            \cdots \ar[r, hook] &
            (x / x_0)^\vee \mathrlap{{} = x^\vee} \ar[d, two heads] \\
            & y_n^\vee & y_{n-1}^\vee & & y_1^\vee \rlap{ ,}
        \end{tikzcd}
    \end{equation}
    where~$x / x_i$ denotes the cokernel of the inclusion $x_i \hookrightarrow x$,
    which exists by the axioms of an exact category.
    We have the short exact sequence
    $y_i \hookrightarrow x / x_{i-1} \twoheadrightarrow x / x_i$
    by the third isomorphism theorem,
    which holds in any exact category.

    This defines a self-dual structure on~$\mathcal{A}^{(n)}$.
    Its self-dual objects are called
    \emph{$n$-step self-dual filtrations} in~$\mathcal{A}$,
    and will be an important idea in our subsequent constructions.
\end{para}

\subsection{Moduli stacks}
\label{subsec-linear-stacks}

\begin{para}
    We describe a set of axioms for algebraic stacks
    that behave like moduli stacks of objects
    in linear categories and self-dual linear categories,
    based on the notion of \emph{linear moduli stacks}
    introduced by \textcite[\S7.1]{epsilon-i},
    which we call \emph{linear stacks} here.

    For the main constructions of this paper,
    it is enough to work only with the moduli stacks,
    without needing to refer to the original categories.
    This is also a benefit of the intrinsic framework for enumerative geometry
    developed in \cite{epsilon-i,epsilon-ii,epsilon-iii}.
\end{para}

\begin{para}[Graded and filtered points]
    \label{para-grad-filt}
    Let $\mathcal{X}$ be an algebraic stack over~$K$.
    Following \textcite{halpern-leistner-instability},
    define the \emph{stack of graded points}
    and the \emph{stack of filtered points} of~$\mathcal{X}$
    as the mapping stacks
    \begin{align*}
        \mathrm{Grad} (\mathcal{X})
        & = \mathrm{Map} ( {*} / \mathbb{G}_\mathrm{m}, \mathcal{X} ) \ ,
        \\
        \mathrm{Filt} (\mathcal{X})
        & = \mathrm{Map} ( \mathbb{A}^1 / \mathbb{G}_\mathrm{m}, \mathcal{X} ) \ ,
    \end{align*}
    where we use the scaling action of~$\mathbb{G}_\mathrm{m}$ on~$\mathbb{A}^1$.
    These are again algebraic stacks over~$K$.

    Consider the morphisms
    \begin{equation*}
        \begin{tikzcd}
            {*} / \mathbb{G}_\mathrm{m}
            \ar[shift left=0.5ex, r, "0"]
            &
            \mathbb{A}^1 / \mathbb{G}_\mathrm{m}
            \ar[shift left=0.5ex, l, "\mathrm{pr}"]
            &
            {*} \vphantom{^n} \ ,
            \ar[shift left=0.5ex, l, "1"]
            \ar[shift right=0.5ex, l, "0"']
            \ar[ll, bend right, start anchor=north west, end anchor=north east, looseness=.8]
        \end{tikzcd}
    \end{equation*}
    where $\mathrm{pr}$ is induced by the projection $\mathbb{A}^1 \to *$.
    These induce morphisms of stacks
    \begin{equation*}
        \begin{tikzcd}
            \mathrm{Grad} (\mathcal{X})
            \ar[rr, bend left, start anchor=north east, end anchor=north west, looseness=.8, "\smash{\mathrm{tot}}"]
            \ar[r, shift right=0.5ex, "\mathrm{sf}"']
            &
            \mathrm{Filt} (\mathcal{X})
            \ar[l, shift right=0.5ex, "\mathrm{gr}"']
            \ar[r, shift left=0.5ex, "\mathrm{ev}_0"]
            \ar[r, shift right=0.5ex, "\mathrm{ev}_1"']
            &
            \mathcal{X} \rlap{ ,}
        \end{tikzcd}
    \end{equation*}
    where the notations `$\mathrm{gr}$', `$\mathrm{sf}$', and `$\mathrm{tot}$' stand for
    the \emph{associated graded point},
    the \emph{split filtration},
    and the \emph{total point}, respectively.
    The morphism $\mathrm{gr}$ is an $\mathbb{A}^1$-deformation retract,
    and the morphisms $\mathrm{tot}$ and $\mathrm{ev}_1$ are representable,
    under our assumptions in \cref{para-intro-conventions}.
\end{para}

\begin{para}[Linear stacks]
    \label{para-linear-stacks}
    Following \textcite[\S7.1]{epsilon-i},
    define a \emph{linear stack} over~$K$
    to be the following data:

    \begin{itemize}
        \item
            An algebraic stack~$\mathcal{X}$ over~$K$.
        \item
            A commutative monoid structure
            $\oplus \colon \mathcal{X} \times \mathcal{X} \to \mathcal{X}$,
            with unit $0 \in \mathcal{X} (K)$.
        \item
            A $* / \Gm$-action
            $\odot \colon {*} / \Gm \times \mathcal{X} \to \mathcal{X}$
            respecting the monoid structure.
    \end{itemize}
    Note that these structures come with extra coherence data.

    In this case, the set~$\uppi_0 (\mathcal{X})$
    of connected components of~$\mathcal{X}$
    carries the structure of a commutative monoid.
    We denote its operation by~$+$, and its unit by~$0$.

    We require the following additional property:

    \begin{itemize}
        \item
            \label{item-linear-moduli-stack-grad}
            There is an isomorphism
            \begin{equation}
                \label{eq-linear-moduli-stack-grad}
                \coprod_{\gamma \colon \mathbb{Z} \to \uppi_0 (\mathcal{X})} {}
                \prod_{n \in \mathrm{supp} (\gamma)}
                \mathcal{X}_{\gamma (n)}
                \longsimto \mathrm{Grad} (\mathcal{X}) \ ,
            \end{equation}
            where~$\gamma$ runs through maps of sets
            $\mathbb{Z} \to \uppi_0 (\mathcal{X})$ such that
            $\mathrm{supp} (\gamma) = \mathbb{Z} \setminus \gamma^{-1} (0)$
            is finite, and the morphism is defined by the composition
            \begin{equation*}
                * / \Gm \times
                \prod_{n \in \mathrm{supp} (\gamma)} \mathcal{X}_{\gamma (n)}
                \overset{(-)^n}{\longrightarrow}
                \prod_{n \in \mathrm{supp} (\gamma)} {}
                (* / \Gm \times \mathcal{X}_{\gamma (n)})
                \overset{\odot}{\longrightarrow}
                \prod_{n \in \mathrm{supp} (\gamma)} \mathcal{X}_{\gamma (n)}
                \overset{\oplus}{\longrightarrow}
                \mathcal{X}
            \end{equation*}
            on the component corresponding to~$\gamma$,
            where the first morphism is given by the
            $n$-th power map $(-)^n \colon {*} / \Gm \to {*} / \Gm$
            on the factor corresponding to~$\mathcal{X}_{\gamma (n)}$.
    \end{itemize}
    We can think of~\cref{eq-linear-moduli-stack-grad}
    roughly as an isomorphism $\mathrm{Grad} (\mathcal{X}) \simeq \mathcal{X}^\mathbb{Z}$,
    where we only consider components of~$\mathcal{X}^\mathbb{Z}$
    involving finitely many non-zero classes in~$\uppi_0 (\mathcal{X})$.

    Most examples of
    moduli stacks of objects in abelian categories
    are linear stacks.
    See \cite[\S7.1.3]{epsilon-i} for details.
\end{para}

\begin{para}[Stacks of filtrations]
    \label{para-linear-filt}
    For a linear stack~$\mathcal{X}$,
    we have canonical isomorphisms
    \begin{equation*}
        \uppi_0 (\mathrm{Filt} (\mathcal{X})) \simeq
        \uppi_0 (\mathrm{Grad} (\mathcal{X})) \simeq
        \{
            \gamma \colon \mathbb{Z} \to \uppi_0 (\mathcal{X}) \mid
            \mathrm{supp} (\gamma) \text{ finite}
        \} \ ,
    \end{equation*}
    where the first isomorphism is induced by the morphism $\mathrm{gr}$,
    and the second is given by \cref{eq-linear-moduli-stack-grad}.

    For classes $\alpha_1, \dotsc, \alpha_n \in \uppi_0 (\mathcal{X})$,
    there is a \emph{stack of filtrations}
    \begin{equation*}
        \mathcal{X}_{\alpha_1, \dotsc, \alpha_n}^+
        \subset \mathrm{Filt} (\mathcal{X}) \ ,
    \end{equation*}
    defined as a component
    corresponding to a map~$\gamma$ as above
    whose non-zero values agree with
    the non-zero elements in $\alpha_n, \dotsc, \alpha_1$, preserving order.
    We think of this as the stack parametrizing $n$-step filtrations
    with stepwise quotients of classes
    $\alpha_1, \dotsc, \alpha_n$.
    The isomorphism type of this stack
    does not depend on the choice of~$\gamma$,
    as in \cite[\S7.1]{epsilon-i}.

    The morphisms defined in \cref{para-grad-filt}
    restrict to canonical morphisms
    $\mathrm{gr} \colon \mathcal{X}_{\alpha_1, \dotsc, \alpha_n}^+
    \to \mathcal{X}_{\alpha_1} \times \cdots \times \mathcal{X}_{\alpha_n}$
    and
    $\mathrm{ev}_1 \colon \mathcal{X}_{\alpha_1, \dotsc, \alpha_n}^+
    \to \mathcal{X}_{\alpha_1 + \dotsc + \alpha_n}$,
    sending a filtration to its
    associated graded object and total object, respectively.
    These do not depend on the choice of~$\gamma$.

    We say that~$\mathcal{X}$ has \emph{quasi-compact filtrations},
    if for any
    $\alpha_1, \dotsc, \alpha_n \in \uppi_0 (\mathcal{X})$,
    the morphism
    $\mathrm{ev}_1 \colon \mathcal{X}_{\alpha_1, \dotsc, \alpha_n}^+
    \to \mathcal{X}_{\alpha_1 + \dotsc + \alpha_n}$
    is quasi-compact.
    See also \textcite[Definition~3.8.1]{halpern-leistner-instability}.
    This is a very mild condition,
    and is satisfied by all examples of our interest.
\end{para}

\begin{para}[Self-dual linear stacks]
    \label{para-sd-linear-moduli-stacks}
    We now introduce a notion of \emph{self-dual linear stacks},
    which describe moduli stacks of objects
    in self-dual linear categories.

    Let~$\mathcal{X}$ be a linear stack over~$K$.
    A \emph{self-dual structure} on~$\mathcal{X}$
    is a $\mathbb{Z}_2$-action on~$\mathcal{X}$,
    given by an involution
    \begin{equation*}
        (-)^\vee \colon \mathcal{X} \longsimto \mathcal{X} \ ,
    \end{equation*}
    together with a $2$-isomorphism
    $\eta \colon (-)^{\vee \vee} \simTo \mathrm{id}_{\mathcal{X}}$
    with $\eta_{(-)^\vee} = (\eta_{\smash{(-)}}^\vee)^{-1}$
    similarly to \cref{para-sd-cat},
    such that the involution respects the monoid structure~$\oplus$ on~$\mathcal{X}$,
    and inverts the $* / \Gm$-action~$\odot$,
    meaning that it is equivariant with respect to the involution
    $(-)^{-1} \colon {*} / \Gm \to {*} / \Gm$.
    Note that extra coherence data is needed
    for these compatibility conditions as well.

    In this case, we call~$\mathcal{X}$
    a \emph{self-dual linear stack}.
    Define the \emph{stack of self-dual points} of~$\mathcal{X}$
    as the fixed locus
    \begin{equation*}
        \mathcal{X}^\mathrm{sd} = \mathcal{X}^{\mathbb{Z}_2} \ .
    \end{equation*}
    It has affine diagonal by \cref{lemma-fixed-affine} below.
    Note that this is different from the fixed locus
    of the automorphism~$(-)^\vee$ of~$\mathcal{X}$,
    which would give the fixed locus of the corresponding
    $\mathbb{Z}$-action on~$\mathcal{X}$,
    rather than that of the $\mathbb{Z}_2$-action.

    There is a monoid action
    \begin{equation*}
        \oplus^\mathrm{sd} \colon
        \mathcal{X} \times \mathcal{X}^\mathrm{sd}
        \longrightarrow \mathcal{X}^\mathrm{sd} \ ,
    \end{equation*}
    given by $(x, y) \mapsto x \oplus y \oplus x^\vee$.
    This induces a monoid action
    $\uppi_0 (\mathcal{X}) \times \uppi_0 (\mathcal{X}^\mathrm{sd})
    \to \uppi_0 (\mathcal{X}^\mathrm{sd})$,
    which we often denote by
    $(\alpha, \theta) \mapsto \alpha + \theta + \alpha^\vee$,
    where $\alpha + \alpha^\vee$ can also be seen as a class
    in~$\uppi_0 (\mathcal{X}^\mathrm{sd})$,
    corresponding to the case when $\theta = 0$.
\end{para}

\begin{example}
    Let $\mathcal{X} = \coprod_{n \in \mathbb{N}} {*} / \mathrm{GL} (n)$
    be the moduli stack of vector spaces over~$K$,
    which is a linear moduli stack.
    Consider the involution
    $(-)^\vee \colon {*} / \mathrm{GL} (n) \to {*} / \mathrm{GL} (n)$
    sending a vector space to its dual,
    or equivalently, sending a matrix to its inverse transpose.
    Then~$\mathcal{X}^\mathrm{sd}$ is the classifying stack
    of non-degenerate symmetric (or anti-symmetric) bilinear forms,
    depending on the choice of the $2$-morphism~$\eta$,
    similarly to \cref{eg-sd-cat-vb}.
    In particular, if~$K$ is algebraically closed of characteristic~$\neq 2$,
    then~$\mathcal{X}^\mathrm{sd}$ is either
    $\coprod_{n \in \mathbb{N}} {*} / \mathrm{O} (n)$
    or $\coprod_{n \in \mathbb{N}} {*} / \mathrm{Sp} (2n)$.
\end{example}

\begin{para}[Self-dual graded points]
    \label{para-sd-grad-filt}
    The involution on~$\mathcal{X}$ induces an involution
    on $\mathrm{Grad} (\mathcal{X})$, and we may identify
    $\mathrm{Grad} (\mathcal{X}^\mathrm{sd}) \simeq \mathrm{Grad} (\mathcal{X})^{\mathbb{Z}_2}$.
    This gives an isomorphism
    \begin{equation}
        \label{eq-linear-moduli-stack-sd-grad}
        \mathrm{Grad} (\mathcal{X}^\mathrm{sd}) \simeq
        \coprod_{\substack{
            \gamma \colon \mathbb{Z} \setminus \{ 0 \} \to \uppi_0 (\mathcal{X}) \\
            \textnormal{involutive,} \\
            \gamma (0) \in \uppi_0 (\mathcal{X}^{\smash{\mathrm{sd}}})
        }} {}
        \biggl(
            \mathcal{X}^\mathrm{sd}_{\gamma (0)} \times
            \prod_{n > 0 \colon \gamma (n) \neq 0}
            \mathcal{X}_{\gamma (n)}
        \biggr) \ ,
    \end{equation}
    where~$\gamma$ runs through finitely supported maps
    that are \emph{involutive}, meaning that
    $\gamma (-n) = \gamma (n)^\vee$ for all
    $n > 0$,
    and~$\gamma (0)$ is a convenient notation
    which is independent of the map~$\gamma$,
    and $\mathcal{X}^\mathrm{sd}_{\gamma (0)} \subset \mathcal{X}^\mathrm{sd}$
    denotes the component corresponding to~$\gamma (0)$.
\end{para}

\begin{para}[Self-dual filtrations]
    \label{para-sd-stack-filt}
    For classes $\alpha_1, \dotsc, \alpha_n \in \uppi_0 (\mathcal{X})$
    and $\theta \in \uppi_0 (\mathcal{X}^\mathrm{sd})$,
    define the \emph{stack of self-dual filtrations}
    \begin{equation*}
        \mathcal{X}_{\alpha_1, \dotsc, \alpha_n, \theta}^{\smash{\mathrm{sd}, +}}
        \subset \mathrm{Filt} (\mathcal{X}^\mathrm{sd})
    \end{equation*}
    as a component
    such that under the isomorphism
    $\uppi_0 (\mathrm{Filt} (\mathcal{X}^\mathrm{sd})) \simeq
    \uppi_0 (\mathrm{Grad} (\mathcal{X}^\mathrm{sd}))$,
    its corresponding map~$\gamma$ as above
    has $\gamma (0) = \theta$,
    and its non-zero values at positive integers agree with
    the non-zero elements in $\alpha_n, \dotsc, \alpha_1$, preserving order.
    This does not depend on the choice of~$\gamma$
    by the constancy theorem of
    \textcite[Theorem~6.1.2]{epsilon-i}.

    The stack $\mathcal{X}_{\alpha_1, \dotsc, \alpha_n, \theta}^{\smash{\mathrm{sd}, +}}$
    can be thought of as parametrizing
    \emph{self-dual filtrations} in the sense of
    \cref{para-sd-filt},
    with stepwise quotients of classes
    $\alpha_1, \dotsc, \alpha_n, \theta, \alpha_n^\vee, \dotsc, \alpha_1^\vee$.

    The morphisms defined in \cref{para-grad-filt}
    restrict to morphisms
    $\mathrm{gr} \colon \mathcal{X}_{\alpha_1, \dotsc, \alpha_n, \theta}^{\smash{\mathrm{sd}, +}}
    \to \mathcal{X}_{\alpha_1} \times \cdots \times \mathcal{X}_{\alpha_n} \times \mathcal{X}_{\theta}^\mathrm{sd}$
    and
    $\mathrm{ev}_1 \colon \mathcal{X}_{\alpha_1, \dotsc, \alpha_n, \theta}^{\smash{\mathrm{sd}, +}}
    \to \mathcal{X}_{\alpha_1 + \cdots + \alpha_n + \theta + \smash{\alpha_n^\vee + \cdots + \alpha_1^\vee}}^\mathrm{sd}$.
    If~$\mathcal{X}$ has quasi-compact filtrations
    as in \cref{para-linear-filt},
    then the morphism~$\mathrm{ev}_1$ described above
    is always quasi-compact.
\end{para}

\begin{lemma}
    \label{lemma-fixed-affine}
    Let~$\mathcal{X}$ be an algebraic stack
    with affine diagonal, acted on by~$\mathbb{Z}_2$.
    Then the forgetful morphism
    $\mathcal{X}^{\mathbb{Z}_2} \to \mathcal{X}$
    is affine.
\end{lemma}

\begin{proof}
    Let
    $\mathcal{J} = \mathcal{X} \times_{j_0, \, \mathcal{X} \times \mathcal{X}, \, j_1} \mathcal{X}$,
    where~$j_0$ is the diagonal morphism,
    and $j_1 = (\mathrm{id}, i)$,
    where~$i$ is the involution.
    Let $\pi \colon \mathcal{J} \to \mathcal{X}$
    be the projection to the first factor,
    which is affine as~$\mathcal{X}$ has affine diagonal.
    Let $\mathbb{Z}_2$ act on~$\mathcal{J}$
    by the involution on the second factor,
    so we may identify
    $\mathcal{J}^{\mathbb{Z}_2} \simeq \mathcal{X}^{\mathbb{Z}_2}$.
    Then~$\pi$ is equivariant with respect to
    the trivial $\mathbb{Z}_2$-action on~$\mathcal{X}$,
    so the forgetful morphism
    $\mathcal{J}^{\mathbb{Z}_2} \to \mathcal{J}$
    is a closed immersion,
    which can be seen by base changing along morphisms
    from affine schemes to~$\mathcal{X}$.
    The composition
    $\mathcal{X}^{\mathbb{Z}_2} \simeq \mathcal{J}^{\mathbb{Z}_2}
    \to \mathcal{J} \to \mathcal{X}$
    is thus affine.
\end{proof}

Note that in this lemma,
we may allow the base to be any algebraic stack,
where we assume that the relative diagonal of~$\mathcal{X}$ is affine,
and that the action respects the structure morphism.

\subsection{Stability conditions}
\label{subsec-stack-stability}

\begin{para}
    We define a notion of \emph{stability conditions}
    on linear stacks,
    based on ideas in the works of
    \textcite{rudakov-1997-stability},
    \textcite{joyce-2007-configurations-iii},
    \textcite{bridgeland-2007-stability},
    as well as
    \textcite{halpern-leistner-instability}.
    Such a stability condition will determine a
    \emph{semistable locus} in the stack,
    and enumerative invariants will count the points
    in the semistable locus.
\end{para}

\begin{para}[\texorpdfstring{$\Theta$}{Θ}-stratifications]
    \label{para-theta-strat}
    We first define \emph{$\Theta$-stratifications}
    following \textcite{halpern-leistner-instability}.
    This is a geometric formulation of the existence and uniqueness
    of HN~filtrations.
    We slightly weaken the original definition
    by discarding the ordering on the set of strata,
    which we will not need.

    A \emph{$\Theta$-stratification} of a stack~$\mathcal{X}$
    is the following data:

    \begin{itemize}
        \item
            Open substacks
            $\mathcal{S} \subset \mathrm{Filt} (\mathcal{X})$
            and $\mathcal{Z} \subset \mathrm{Grad} (\mathcal{X})$,
            with
            $\mathcal{S} = \mathrm{gr}^{-1} (\mathcal{Z})$,
    \end{itemize}
    such that for each
    $\lambda \in \uppi_0 (\mathrm{Grad} (\mathcal{X}))
    \simeq \uppi_0 (\mathrm{Filt} (\mathcal{X}))$,
    if we write
    $\mathcal{S}_\lambda \subset \mathcal{S}$
    and $\mathcal{Z}_\lambda \subset \mathcal{Z}$
    for the parts lying in the components
    $\mathcal{X}_\lambda^+ \subset \mathrm{Filt} (\mathcal{X})$
    and
    $\mathcal{X}_\lambda \subset \mathrm{Grad} (\mathcal{X})$,
    respectively, then:

    \begin{itemize}
        \item
            For each $\lambda$, the morphism
            $\mathrm{ev}_1 \colon \mathcal{S}_\lambda \to \mathcal{X}$
            is a locally closed immersion,
            and the family
            $(\mathcal{S}_\lambda)_\lambda$
            defines a locally closed stratification of~$\mathcal{X}$.
    \end{itemize}
    In this case,
    each~$\mathcal{S}_\lambda$ is called a \emph{stratum},
    and each~$\mathcal{Z}_\lambda$ is called the \emph{centre}
    of the stratum~$\mathcal{S}_\lambda$.
\end{para}

\begin{para}[Stability conditions]
    \label{para-linear-stack-stability}
    Let~$\mathcal{X}$ be a linear stack.
    A \emph{stability condition} on~$\mathcal{X}$ is a map
    \begin{equation*}
        \tau \colon \uppi_0 (\mathcal{X}) \setminus \{ 0 \} \longrightarrow T
    \end{equation*}
    to a totally ordered set~$T$,
    satisfying the following conditions:

    \begin{enumerate}
        \item
            If $\alpha_1, \alpha_2 \in \uppi_0 (\mathcal{X}) \setminus \{ 0 \}$
            and $\tau (\alpha_1) \leq \tau (\alpha_2)$,
            then $\tau (\alpha_1) \leq \tau (\alpha_1 + \alpha_2) \leq \tau (\alpha_2)$.

        \item
            For any class $\alpha \in \uppi_0 (\mathcal{X})$,
            the \emph{semistable locus}
            \begin{equation}
                \label{eq-def-x-ss}
                \mathcal{X}_\alpha^\mathrm{ss} (\tau) =
                \mathcal{X}_\alpha \Bigm\backslash
                \bigcup_{\substack{\alpha = \alpha_1 + \alpha_2 \\ \tau (\alpha_1) > \tau (\alpha_2)}}
                \mathrm{ev_1} (\mathcal{X}_{\alpha_1, \alpha_2}^+)
            \end{equation}
            is open in~$\mathcal{X}_\alpha$,
            where $\alpha_1, \alpha_2$ are assumed non-zero.
            Moreover, for any $t \in T$,
            the union of all $\mathcal{X}_\alpha^\mathrm{ss} (\tau)$
            with either $\alpha = 0$ or $\tau (\alpha) = t$
            is an open linear substack of~$\mathcal{X}$.

        \item
            The open substacks
            \begin{alignat*}{2}
                \mathcal{Z}_{\alpha_1, \dotsc, \alpha_n} (\tau)
                & =
                \mathcal{X}_{\alpha_1}^\mathrm{ss} (\tau) \times \cdots \times \mathcal{X}_{\alpha_n}^\mathrm{ss} (\tau)
                && \subset
                \mathcal{X}_{\alpha_1} \times \cdots \times \mathcal{X}_{\alpha_n} \ ,
                \\
                \mathcal{S}_{\alpha_1, \dotsc, \alpha_n} (\tau)
                & =
                \mathrm{gr}^{-1} (\mathcal{Z}_{\alpha_1, \dotsc, \alpha_n} (\tau))
                && \subset
                \mathcal{X}_{\alpha_1, \dotsc, \alpha_n}^+
            \end{alignat*}
            for all $n \geq 0$ and classes
            $\alpha_1, \dotsc, \alpha_n \in \uppi_0 (\mathcal{X}) \setminus \{ 0 \}$
            with $\tau (\alpha_1) > \cdots > \tau (\alpha_n)$
            define a $\Theta$-stratification of~$\mathcal{X}$
            in the sense of~\cref{para-theta-strat}.
    \end{enumerate}
    More precisely, the last condition means that for each choice of
    $\alpha_1, \dotsc, \alpha_n$ as above,
    we choose an element~$\lambda$ as in
    \cref{para-theta-strat}
    such that
    $\mathcal{X}_\lambda \simeq \mathcal{X}_{\alpha_1} \times \cdots \times \mathcal{X}_{\alpha_n}$
    and
    $\mathcal{X}_\lambda^+ \simeq \mathcal{X}_{\alpha_1, \dotsc, \alpha_n}^+$,
    and we set $\mathcal{Z}_\lambda$ and $\mathcal{S}_\lambda$ as above;
    for all other~$\lambda$, we set them to be empty.
\end{para}

\begin{para}[Examples]
    \label[example]{eg-stack-stability}
    Here are some examples of stability conditions
    on linear stacks.

    \begin{enumerate}
        \item
            \label{item-stack-stability-trivial}
            Let~$\mathcal{X}$ be any linear stack.
            The constant map $\tau \colon \uppi_0 (\mathcal{X}) \setminus \{ 0 \} \to \{ 0 \}$
            is called the \emph{trivial stability condition},
            where $\mathcal{X}_\alpha^\mathrm{ss} (\tau) = \mathcal{X}_\alpha$
            for all~$\alpha$.
        \item
            \label{item-stack-stability-quiver}
            Let~$\mathcal{X}$ be the moduli stack of representations of a quiver~$Q$.
            Then any \emph{slope function}
            $\mu \colon Q_0 \to \mathbb{Q}$
            induces a stability condition on~$\mathcal{X}$ given by
            \begin{equation*}
                \tau (d) = \frac{\sum_{i \in Q_0} d_i \cdot \mu (i)}{\sum_{i \in Q_0} d_i}
            \end{equation*}
            for non-zero dimension vectors $d \in \uppi_0 (\mathcal{X}) \setminus \{ 0 \}$,
            where the $\Theta$-stratification can be constructed from
            \textcite[Theorem~2.6.3]{ibanez-nunez-filtrations}.
            See \cref{subsec-quiver-dt}
            for more details.
        \item
            \label{item-stack-stability-gieseker}
            Let~$\mathcal{X}$ be the moduli stack of coherent sheaves
            on a projective scheme~$Y$
            over a field~$K$ of characteristic zero.
            Then \emph{Gieseker stability}
            is a stability condition on~$\mathcal{X}$,
            where the choice of~$\tau$ is described in
            \textcite[Example~4.16]{joyce-2007-configurations-iii},
            and the $\Theta$-stratification exists by
            \textcite[Example~7.28]{alper-halpern-leistner-heinloth-2023}.
    \end{enumerate}
\end{para}

\begin{para}[For self-dual linear stacks]
    \label{para-sd-linear-stack-stability}
    Let~$\mathcal{X}$ be a self-dual linear stack over~$K$,
    and let~$\tau$ be a stability condition on~$\mathcal{X}$.
    We say that~$\tau$ is \emph{self-dual},
    if the following condition holds:

    \begin{itemize}
        \item
            For any $\alpha, \beta \in \uppi_0 (\mathcal{X}) \setminus \{ 0 \}$,
            we have $\tau (\alpha) \leq \tau (\beta)$
            if and only if
            $\tau (\alpha^\vee) \geq \tau (\beta^\vee)$.
    \end{itemize}
    In this case,
    for each $\theta \in \uppi_0 (\mathcal{X}^\mathrm{sd})$,
    writing $\alpha = j (\theta)$ for the corresponding class
    in $\uppi_0 (\mathcal{X})$, we have the \emph{semistable locus}
    \begin{equation}
        \label{eq-def-x-sd-ss}
        \mathcal{X}^{\smash{\mathrm{sd}, \mathrm{ss}}}_\theta (\tau) =
        \mathcal{X}^\mathrm{ss}_\alpha (\tau)^{\mathbb{Z}_2} \cap \mathcal{X}^\mathrm{sd}_\theta
        \subset \mathcal{X}^\mathrm{sd}_\theta \ ,
    \end{equation}
    where $\mathcal{X}^\mathrm{sd}_\theta \subset (\mathcal{X}_\alpha)^{\mathbb{Z}_2}$
    as an open and closed substack.

    We have an induced
    $\Theta$-stratification of~$\mathcal{X}^\mathrm{sd}$
    given by the open substacks
    \begin{alignat*}{2}
        \mathcal{Z}^\mathrm{sd}_{\alpha_1, \dotsc, \alpha_n, \theta} (\tau)
        & =
        \mathcal{X}^\mathrm{ss}_{\alpha_1} (\tau) \times \cdots \times
        \mathcal{X}^\mathrm{ss}_{\alpha_n} (\tau) \times
        \mathcal{X}^{\smash{\mathrm{sd}, \mathrm{ss}}}_\theta (\tau)
        && \subset
        \mathcal{X}_{\alpha_1} \times \cdots \times \mathcal{X}_{\alpha_n} \times \mathcal{X}^\mathrm{sd}_\theta \ ,
        \\
        \mathcal{S}^\mathrm{sd}_{\alpha_1, \dotsc, \alpha_n, \theta} (\tau)
        & =
        \mathrm{gr}^{-1} (\mathcal{Z}^\mathrm{sd}_{\alpha_1, \dotsc, \alpha_n, \theta} (\tau))
        && \subset
        \mathcal{X}_{\alpha_1, \dotsc, \alpha_n, \theta}^{\smash{\mathrm{sd}, +}} \ ,
    \end{alignat*}
    where $\alpha_1, \dotsc, \alpha_n \in \uppi_0 (\mathcal{X}) \setminus \{ 0 \}$
    and $\theta \in \uppi_0 (\mathcal{X}^\mathrm{sd})$
    are classes such that
    $\tau (\alpha_1) > \cdots > \tau (\alpha_n) > 0$.
    These strata and centres can also be realized as
    $\mathbb{Z}_2$-fixed loci in the strata and centres of
    the $\Theta$-stratification of~$\mathcal{X}$ given by~$\tau$.
\end{para}

\begin{para}[Permissibility]
    \label{para-permissible-stability}
    Let~$\mathcal{X}$ be a linear stack over~$K$,
    and let~$\tau$ be a stability condition on~$\mathcal{X}$.
    We say that~$\tau$ is \emph{permissible},
    if the following condition holds:

    \begin{itemize}
        \item
            For any $\alpha \in \uppi_0 (\mathcal{X})$,
            the semistable locus
            $\mathcal{X}^\mathrm{ss}_\alpha (\tau) \subset \mathcal{X}_\alpha$
            is quasi-compact.
    \end{itemize}
    This is similar to the notion of
    permissible weak stability conditions in
    \textcite[Definition~4.7]{joyce-2007-configurations-iii} and
    \textcite[Definition~3.7]{joyce-song-2012}.

    We also explain in \cref{para-permissibility-comparison} below
    that this notion of permissibility
    implies the corresponding \emph{stability measures}
    being permissible as in \cite[\S4.1.4]{epsilon-ii}.
    The following lemma can be seen as
    a shadow of this result, which we prove directly.
\end{para}

\begin{lemma}
    \label{lemma-finite-decomposition}
    Let~$\mathcal{X}$ be a linear stack over~$K$,
    with quasi-compact filtrations as in \cref{para-linear-filt},
    and let~$\tau$ be a permissible stability condition on~$\mathcal{X}$.

    Then for any $\alpha \in \uppi_0 (\mathcal{X}) \setminus \{ 0 \}$,
    there are only finitely many decompositions
    $\alpha = \alpha_1 + \cdots + \alpha_n$
    into classes $\alpha_i \in \uppi_0 (\mathcal{X}) \setminus \{ 0 \}$,
    such that $\tau (\alpha_i) = \tau (\alpha)$
    and $\mathcal{X}^\mathrm{ss}_{\alpha_i} (\tau) \neq \varnothing$
    for all~$i$.
\end{lemma}

\begin{proof}
    Let $t = \tau (\alpha)$.
    Then the open substack
    \begin{equation*}
        \mathcal{X} (\tau; t) =
        \{ 0 \} \cup
        \coprod_{\substack{
            \beta \in \uppi_0 (\mathcal{X}) \setminus \{ 0 \} \mathrlap{:} \\
            \tau (\beta) = t
        }}
        \mathcal{X}^\mathrm{ss}_{\smash{\beta}} (\tau)
        \subset \mathcal{X}
    \end{equation*}
    is again a linear stack.
    Replacing~$\mathcal{X}$ by~$\mathcal{X} (\tau; t)$,
    we may assume that~$\mathcal{X}$
    has quasi-compact connected components,
    and that~$\tau$ is trivial.

    By the finiteness theorem of
    \textcite[Theorem~6.2.3]{epsilon-i},
    each connected component of~$\mathcal{X}$
    has finitely many \emph{special faces}.
    In this case, this is the statement that for any
    $\alpha \in \uppi_0 (\mathcal{X}) \setminus \{ 0 \}$,
    there are finitely many decompositions
    $\alpha = \alpha_1 + \cdots + \alpha_n$,
    such that all other decompositions can be obtained
    from combining terms in these decompositions,
    and hence the total number of decompositions is finite.
\end{proof}

\section{Rings of motives}

\subsection{Definition}
\label{subsec-motives}

\begin{para}
    We provide background material on
    \emph{rings of motives} over an algebraic stack,
    following \textcite[\S3]{epsilon-ii},
    based on the theory of \textcite{joyce-2007-stack-functions}.
    These rings will be used to construct DT invariants later.

    Recall from \cref{para-intro-conventions}
    our running assumptions on algebraic stacks.
\end{para}

\begin{para}[The ring of motives]
    \label{para-ring-of-motives}
    Let~$\mathcal{X}$ be an algebraic stack over~$K$,
    and let~$A$ be a commutative ring.
    The \emph{ring of motives} over~$\mathcal{X}$ with coefficients in~$A$
    is the $A$-module
    \begin{equation*}
        \mathbb{M} (\mathcal{X}; A) =
        \mathop{\hat{\bigoplus}}_{\mathcal{Z} \to \mathcal{X}} \ 
        A \cdot [\mathcal{Z}] \, \Big/ {\sim} \ ,
    \end{equation*}
    where we run through isomorphism classes of representable morphisms
    $\mathcal{Z} \to \mathcal{X}$ of finite type,
    with~$\mathcal{Z}$ quasi-compact,
    and $\hat{\oplus}$ indicates that we take the set of \emph{locally finite sums},
    that is, possibly infinite sums
    $\sum_{\mathcal{Z} \to \mathcal{X}} a_{\mathcal{Z}} \cdot [\mathcal{Z}]$,
    such that for each quasi-compact open substack $\mathcal{U} \subset \mathcal{X}$,
    there are only finitely many $\mathcal{Z}$ such that
    $a_{\mathcal{Z}} \neq 0$ and $\mathcal{Z} \times_{\mathcal{X}} \mathcal{U} \neq \varnothing$.
    The relation $\sim$ is generated by locally finite sums of elements of the form
    \begin{equation*}
        a \cdot ([\mathcal{Z}] - [\mathcal{Z}'] - [\mathcal{Z} \setminus \mathcal{Z}']) \ ,
    \end{equation*}
    where $a \in A$, $\mathcal{Z}$ is as above,
    and $\mathcal{Z}' \subset \mathcal{Z}$ is a closed substack.
    The class $[\mathcal{Z}] \in \mathbb{M} (\mathcal{X}; A)$
    is called the \emph{motive} of~$\mathcal{Z}$.

    For a representable morphism $\mathcal{Z} \to \mathcal{X}$ of finite type,
    where~$\mathcal{Z}$ is not necessarily quasi-compact,
    we can still define its motive $[\mathcal{Z}] \in \mathbb{M} (\mathcal{X}; A)$,
    by stratifying~$\mathcal{Z}$ into quasi-compact locally closed substacks,
    $\mathcal{Z} = \bigcup_{i \in I} \mathcal{Z}_i$,
    and defining $[\mathcal{Z}] = \sum_{i \in I} {} [\mathcal{Z}_i]$ as a locally finite sum.

    The ring structure on $\mathbb{M} (\mathcal{X}; A)$
    is given by $[\mathcal{Z}] \cdot [\mathcal{Z}'] =
    [\mathcal{Z} \times_{\mathcal{X}} \mathcal{Z}']$
    on generators,
    with unit element $[\mathcal{X}]$.

    We also write~$\mathbb{M} (\mathcal{X})$ for $\mathbb{M} (\mathcal{X}; \mathbb{Z})$,
    and~$\mathbb{M} (K; A)$ for $\mathbb{M} (\mathrm{Spec} (K); A)$.
\end{para}

\begin{para}[Properties]
    \label{para-motive-properties}
    We list some basic properties of rings of motives;
    see \cite[\S 3]{epsilon-ii} for details.

    \begin{enumerate}
        \item
            For a morphism $f \colon \mathcal{Y} \to \mathcal{X}$,
            there is a \emph{pullback map}
            \begin{equation*}
                f^* \colon \mathbb{M} (\mathcal{X}; A)
                \longrightarrow \mathbb{M} (\mathcal{Y}; A) \ ,
            \end{equation*}
            given by $[\mathcal{Z}] \mapsto [\mathcal{Z} \times_{\mathcal{X}} \mathcal{Y}]$
            on generators,
            which is a ring homomorphism.

        \item
            For a representable quasi-compact morphism
            $f \colon \mathcal{Y} \to \mathcal{X}$,
            there is a \emph{pushforward map}
            \begin{equation*}
                f_! \colon \mathbb{M} (\mathcal{Y}; A)
                \longrightarrow \mathbb{M} (\mathcal{X}; A) \ ,
            \end{equation*}
            given by $[\mathcal{Z}] \mapsto [\mathcal{Z}]$ on generators.
            This is not a ring homomorphism in general.

        \item
            For stacks $\mathcal{X}$ and $\mathcal{Y}$,
            there is an \emph{external product}
            \begin{equation*}
                \boxtimes \colon
                \mathbb{M} (\mathcal{X}; A) \otimes \mathbb{M} (\mathcal{Y}; A)
                \longrightarrow \mathbb{M} (\mathcal{X} \times \mathcal{Y}; A) \ ,
            \end{equation*}
            given by $[\mathcal{Z}] \otimes [\mathcal{Z}'] \mapsto
            [\mathcal{Z} \times \mathcal{Z}']$ on generators.
            The multiplication on~$\mathbb{M} (\mathcal{X}; A)$
            can be realized as the external product for
            $\mathcal{Y} = \mathcal{X}$
            followed by pulling back along the diagonal.

        \item
            For a representable quasi-compact morphism
            $f \colon \mathcal{Y} \to \mathcal{X}$,
            we have the \emph{projection formula}
            \begin{equation}
                \label{eq-motive-projection}
                f_! (a \cdot f^* (b)) = f_! (a) \cdot b
            \end{equation}
            for $a \in \mathbb{M} (\mathcal{Y}; A)$ and $b \in \mathbb{M} (\mathcal{X}; A)$,
            which can be verified on generators.

        \item
            For a pullback diagram
            \begin{equation*}
                \begin{tikzcd}[row sep=1.5em]
                    \mathcal{Y}' \ar[r, "{\smash[t]{g'}}"] \ar[d, "f'"']
                    \ar[dr, phantom, "\ulcorner", pos=.15]
                    & \mathcal{Y} \ar[d, "f"]
                    \\ \mathcal{X}' \ar[r, "g"]
                    & \mathcal{X} \rlap{ ,}
                \end{tikzcd}
            \end{equation*}
            where~$f$ is representable and quasi-compact,
            we have the \emph{base change formula}
            \begin{equation}
                \label{eq-motive-base-change}
                g^* \circ f_! = f'_! \circ (g')^* \colon
                \mathbb{M} (\mathcal{Y}; A) \longrightarrow \mathbb{M} (\mathcal{X}'; A) \ .
            \end{equation}

        \item
            For a generator $[\mathcal{Z}] \in \mathbb{M} (\mathcal{X}; A)$
            and a vector bundle
            $\mathcal{E} \to \mathcal{Z}$ of rank~$n$, we have
            \begin{equation}
                \label{eq-motive-vb}
                [\mathcal{E}] = \mathbb{L}^n \cdot [\mathcal{Z}] \ .
            \end{equation}
    \end{enumerate}
\end{para}

\begin{para}[Motivic integration]
    \label{para-motivic-integration}
    We also consider the localization
    \begin{equation*}
        \hat{\mathbb{M}} (\mathcal{X}; A) =
        \mathbb{M} (\mathcal{X}; A)
        \underset{A [\mathbb{L}]}{\mathbin{\hat{\otimes}}}
        A [\mathbb{L}^{\pm 1}, (\mathbb{L}^k - 1)^{-1} : k > 0] \ ,
    \end{equation*}
    where $\mathbb{L} = [\mathbb{A}^1]$
    is the motive of the affine line,
    and $\hat{\otimes}$ denotes the completed tensor product
    with respect to locally finite sums.
    We call this the \emph{completed ring of motives} over~$\mathcal{X}$.

    For a quasi-compact algebraic stack~$\mathcal{X}$ over~$K$,
    there is a \emph{motivic integration map}
    \begin{equation*}
        \int_{\mathcal{X}} {} (-) \colon
        \mathbb{M} (\mathcal{X}; A) \longrightarrow \hat{\mathbb{M}} (K; A) \ ,
    \end{equation*}
    sending a generator of the form $[Z / \mathrm{GL} (n)]$
    to the element $[Z] \cdot [\mathrm{GL} (n)]^{-1}$,
    where~$Z$ is a quasi-compact algebraic space over~$K$.
    Such elements $[Z / \mathrm{GL} (n)]$ generate~$\mathbb{M} (\mathcal{X}; A)$.
    See \cite[\S3]{epsilon-ii} for details.
\end{para}

\begin{para}[Euler characteristics]
    \label{para-euler-characteristic}

    Let
    \begin{equation*}
        \hat{\mathbb{M}}^\mathrm{reg} (K; A) =
        \mathbb{M} (K; A)
        \underset{A [\mathbb{L}]}{\otimes}
        A [\mathbb{L}^{\pm 1}, (1 + \mathbb{L} + \cdots + \mathbb{L}^{k - 1})^{-1} : k > 0]
        \Big/ (\mathbb{L} - 1) \textnormal{-torsion}
    \end{equation*}
    be the subring of~$\hat{\mathbb{M}} (K; A)$
    consisting of motives `with no poles at $\mathbb{L} = 1$',
    and suppose that~$A$ is a $\mathbb{Q}$-algebra.
    Then there is an Euler characteristic map
    $\chi \colon \hat{\mathbb{M}}^\mathrm{reg} (K; A) \to A$,
    sending~$\mathbb{L}$ to~$1$,
    and $(1 + \mathbb{L} + \cdots + \mathbb{L}^{k - 1})^{-1}$ to~$1 / k$.
    See \cite[\S3]{epsilon-ii} for details.
\end{para}

\begin{para}[The virtual rank decomposition]
    \label{para-virtual-rank}
    Let~$\mathcal{X}$ be a stack over~$K$,
    and let~$A$ be a commutative $\mathbb{Q}$-algebra.
    As in \textcite[\S5]{joyce-2007-stack-functions}
    and \textcite[\S5.1]{epsilon-ii},
    there is a \emph{virtual rank decomposition}
    \begin{equation*}
        \mathbb{M} (\mathcal{X}; A) =
        \mathop{\hat{\bigoplus}}_{k \geq 0} {}
        \mathbb{M}^{(k)} (\mathcal{X}; A) \ ,
    \end{equation*}
    where $\hat{\oplus}$ means taking locally finite sums
    as in \cref{para-ring-of-motives},
    and each $\mathbb{M}^{(k)} (\mathcal{X}; A) \subset \mathbb{M} (\mathcal{X}; A)$
    is the submodule of motives of \emph{pure virtual rank~$k$}.
    Roughly speaking, having virtual rank~$k$
    means having a pole of order at most~$k$ at~$\mathbb{L} = 1$
    after motivic integration.
    We omit the precise definition here,
    which is complicated.

    When~$\mathcal{X}$ is quasi-compact,
    the image of the map
    $\int_{\mathcal{X}} {} (-) \colon
    \mathbb{M}^{(\leq k)} (\mathcal{X}; A) \to \hat{\mathbb{M}} (K; A)$
    lies in the subspace
    $(\mathbb{L} - 1)^{-k} \cdot \hat{\mathbb{M}}^\mathrm{reg} (K; A)
    \subset \hat{\mathbb{M}} (K; A)$,
    where $\mathbb{M}^{(\leq k)} =
    \mathbb{M}^{(0)} \oplus \cdots \oplus \mathbb{M}^{(k)}$.
    In particular, there is an Euler characteristic integration map
    \begin{equation*}
        \int_{\mathcal{X}} {} (\mathbb{L} - 1)^k \cdot (-) \, d \chi
        = \chi \circ \int_{\mathcal{X}} {} (\mathbb{L} - 1)^k \cdot (-)
        \colon \
        \mathbb{M}^{(\leq k)} (\mathcal{X}; A) \longrightarrow A \ .
    \end{equation*}
\end{para}

\subsection{Motivic Hall algebras and modules}
\label{subsec-hall}

\begin{para}
    We introduce the \emph{motivic Hall algebra}
    for a linear stack,
    originally defined by \textcite{joyce-2007-configurations-ii},
    which is an associative algebra structure
    on the ring of motives~$\mathbb{M} (\mathcal{X})$.

    For self-dual linear stacks,
    we show that the ring of motives~$\mathbb{M} (\mathcal{X}^\mathrm{sd})$
    is a module for the motivic Hall algebra~$\mathbb{M} (\mathcal{X})$,
    which we call the \emph{motivic Hall module}.

    Hall modules have been constructed and studied
    for other flavours of Hall algebras, such as by
    \textcite{young-2015-self-dual,young-2016-hall-module,young-2020-quiver}
    in the context of Ringel's \cite{ringel-1990-hall-1,ringel-1990-hall-2}
    Hall algebras and that of cohomological Hall algebras.
    A similar construction in the context of Joyce's
    \cite{joyce-hall,joyce-homological}
    vertex algebras is obtained by
    \textcite{bu-self-dual-ii}.
    Another closely related work is
    \textcite{dehority-latyntsev},
    who studied the relation between the cohomological version
    and the vertex algebra version.
\end{para}

\begin{para}[The motivic Hall algebra]
    \label{para-hall-alg}
    Let~$\mathcal{X}$ be a linear stack over~$K$,
    with quasi-compact filtrations as in \cref{para-linear-filt}.
    Define an operation
    \begin{equation*}
        {*} \colon
        \mathbb{M} (\mathcal{X}) \otimes \mathbb{M} (\mathcal{X})
        \longrightarrow \mathbb{M} (\mathcal{X})
    \end{equation*}
    by the composition
    \begin{equation*}
        \mathbb{M} (\mathcal{X}) \otimes \mathbb{M} (\mathcal{X})
        \overset{\boxtimes}{\longrightarrow}
        \mathbb{M} (\mathcal{X} \times \mathcal{X})
        \overset{\mathrm{gr}^*}{\longrightarrow}
        \mathbb{M} (\mathcal{X}^+)
        \overset{(\mathrm{ev}_1)_!}{\longrightarrow}
        \mathbb{M} (\mathcal{X}) \ ,
    \end{equation*}
    where $\mathcal{X}^+$ denotes the disjoint union of
    the stacks of filtrations $\mathcal{X}_{\alpha_1, \alpha_2}^+$
    for all $\alpha_1, \alpha_2 \in \uppi_0 (\mathcal{X})$.

    Roughly speaking, for motives $a, b \in \mathbb{M} (\mathcal{X})$,
    the product $a * b \in \mathbb{M} (\mathcal{X})$
    parametrizes all possible extensions
    of objects parametrized by~$a$ and~$b$, respectively.

    We will see in \cref{thm-hall-assoc}
    that the product~$*$ is associative,
    and that it has a unit element~$[\{ 0 \}] \in \mathbb{M} (\mathcal{X})$,
    which is the motive of the component $\{ 0 \} \subset \mathcal{X}$.
    This defines an associative algebra structure on~$\mathbb{M} (\mathcal{X})$,
    called the \emph{motivic Hall algebra} of~$\mathcal{X}$.
\end{para}

\begin{para}[The motivic Hall module]
    \label{para-hall-mod}
    Now, let~$\mathcal{X}$ be a self-dual linear stack over~$K$,
    with quasi-compact filtrations as in \cref{para-linear-filt}.
    Define an operation
    \begin{equation*}
        \diamond \colon
        \mathbb{M} (\mathcal{X}) \otimes \mathbb{M} (\mathcal{X}^\mathrm{sd})
        \longrightarrow \mathbb{M} (\mathcal{X}^\mathrm{sd})
    \end{equation*}
    by the composition
    \begin{equation*}
        \mathbb{M} (\mathcal{X}) \otimes \mathbb{M} (\mathcal{X}^\mathrm{sd})
        \overset{\boxtimes}{\longrightarrow}
        \mathbb{M} (\mathcal{X} \times \mathcal{X}^\mathrm{sd})
        \overset{\mathrm{gr}^*}{\longrightarrow}
        \mathbb{M} (\mathcal{X}^{\mathrm{sd}, +})
        \overset{(\mathrm{ev}_1)_!}{\longrightarrow}
        \mathbb{M} (\mathcal{X}^\mathrm{sd}) \ ,
    \end{equation*}
    where $\mathcal{X}^{\mathrm{sd}, +}$ denotes the disjoint union of
    the stacks of filtrations $\mathcal{X}_{\alpha, \theta}^+$
    for all $\alpha \in \uppi_0 (\mathcal{X})$
    and $\theta \in \uppi_0 (\mathcal{X}^\mathrm{sd})$.

    Again, roughly speaking, for motives $a \in \mathbb{M} (\mathcal{X})$
    and $b \in \mathbb{M} (\mathcal{X}^\mathrm{sd})$,
    the product $a \diamond b \in \mathbb{M} (\mathcal{X}^\mathrm{sd})$
    parametrizes the total objects of all possible three-step self-dual filtrations,
    as in \cref{para-sd-filt},
    whose graded pieces are parametrized by~$a$, $b$, and~$a^\vee$, respectively,

    We will prove in \cref{thm-hall-assoc}
    that the product~$\diamond$ establishes~$\mathbb{M} (\mathcal{X}^\mathrm{sd})$
    as a left module for the motivic Hall algebra~$\mathbb{M} (\mathcal{X})$.
    This is called the \emph{motivic Hall module} of~$\mathcal{X}$.
\end{para}

\begin{theorem}
    \label{thm-hall-assoc}
    Let~$\mathcal{X}$ be a linear stack over~$K$,
    with quasi-compact filtrations.

    \begin{enumerate}
        \item
            \label{item-hall-assoc-alg}
            Recall the operation~$*$ defined in \cref{para-hall-alg}.
            Then for any $a, b, c \in \mathbb{M} (\mathcal{X})$, we have
            \begin{gather}
                \label{eq-hall-alg-unit}
                [\{ 0 \}] * a = a = a * [\{ 0 \}] \ ,
                \\
                \label{eq-hall-alg-assoc}
                (a * b) * c = a * (b * c) \ ,
            \end{gather}
            where $[\{ 0 \}] \in \mathbb{M} (\mathcal{X})$
            is the motive of the component $\{ 0 \} \subset \mathcal{X}$.

        \item
            \label{item-hall-assoc-mod}
            Suppose that~$\mathcal{X}$ is equipped with a self-dual structure.
            Consider the involution~$(-)^\vee$ on~$\mathbb{M} (\mathcal{X})$
            induced by the involution of~$\mathcal{X}$,
            and the operation~$\diamond$ defined in \cref{para-hall-mod}.
            Then for any $a, b \in \mathbb{M} (\mathcal{X})$
            and $c \in \mathbb{M} (\mathcal{X}^\mathrm{sd})$,
            we have
            \begin{align}
                \label{eq-hall-alg-invol}
                a^\vee * b^\vee & = (b * a)^\vee \ ,
                \\
                \label{eq-hall-mod-unit}
                [\{ 0 \}] \diamond c & = c \ ,
                \\
                \label{eq-hall-mod-assoc}
                a \diamond (b \diamond c) & = (a * b) \diamond c \ .
            \end{align}
    \end{enumerate}
\end{theorem}

\begin{proof}
    For~\cref{eq-hall-alg-unit},
    it is enough to show that for any
    $\alpha \in \uppi_0 (\mathcal{X})$,
    the morphisms
    $\mathcal{X}_{0, \alpha}^+ \to \mathcal{X}_{\alpha}$
    and $\mathcal{X}_{\alpha, 0}^+ \to \mathcal{X}_{\alpha}$
    are isomorphisms,
    which follows from the descriptions in \cref{para-linear-filt}.

    For~\cref{eq-hall-alg-assoc},
    we may assume that $a \in \mathbb{M} (\mathcal{X}_{\alpha_1})$,
    $b \in \mathbb{M} (\mathcal{X}_{\alpha_2})$,
    and $c \in \mathbb{M} (\mathcal{X}_{\alpha_3})$,
    for some $\alpha_1, \alpha_2, \alpha_3 \in \uppi_0 (\mathcal{X})$.
    Applying the base change formula~\cref{eq-motive-base-change}
    to the pullback squares in the diagrams
    \begin{equation}
        \label{eq-cd-hall-assoc}
        \hspace{-1em}
        \begin{tikzcd}[column sep={3em,between origins}, row sep={3em,between origins}]
            &&
            \smash[t]{\mathcal{X}_{\alpha_1, \alpha_2, \alpha_3}^+}
            \ar[dl] \ar[dr]
            \ar[dd, phantom, "\ulcorner" {pos=.15, rotate=-45, shift={(-.08em,0)}}]
            &&
            \\
            & \mathcal{X}_{\alpha_1, \alpha_2}^+ \times \mathcal{X}_{\alpha_3}
            \ar[dl] \ar[dr]
            && \mathcal{X}_{\alpha_1 + \alpha_2, \alpha_3}^+
            \ar[dl] \ar[dr]
            \\
            \mathcal{X}_{\alpha_1} {\times} \mathcal{X}_{\alpha_2} {\times} \mathcal{X}_{\alpha_3}
            && \mathcal{X}_{\alpha_1 + \alpha_2} {\times} \mathcal{X}_{\alpha_3}
            && \mathcal{X}_{\alpha_1 + \alpha_2 + \alpha_3}
            \rlap{ ,}
        \end{tikzcd}
        \quad
        \begin{tikzcd}[column sep={3em,between origins}, row sep={3em,between origins}]
            &&
            \smash[t]{\mathcal{X}_{\alpha_1, \alpha_2, \alpha_3}^+}
            \ar[dl] \ar[dr]
            \ar[dd, phantom, "\ulcorner" {pos=.15, rotate=-45, shift={(-.08em,0)}}]
            &&
            \\
            & \mathcal{X}_{\alpha_1} \times \mathcal{X}_{\alpha_2, \alpha_3}^+
            \ar[dl] \ar[dr]
            && \mathcal{X}_{\alpha_1, \alpha_2 + \alpha_3}^+
            \ar[dl] \ar[dr]
            \\
            \mathcal{X}_{\alpha_1} {\times} \mathcal{X}_{\alpha_2} {\times} \mathcal{X}_{\alpha_3}
            && \mathcal{X}_{\alpha_1} {\times} \mathcal{X}_{\alpha_2 + \alpha_3}
            && \mathcal{X}_{\alpha_1 + \alpha_2 + \alpha_3}
            \rlap{ ,}
        \end{tikzcd}
        \hspace{-3em}
    \end{equation}
    we see that both sides of~\cref{eq-hall-alg-assoc}
    are equal to
    $(\mathrm{ev}_1)_! \circ \mathrm{gr}^* (a \boxtimes b \boxtimes c)$,
    where $\mathrm{gr}$ and $\mathrm{ev}_1$
    are the outer compositions in both diagrams in~\cref{eq-cd-hall-assoc}.
    These diagrams are special cases of the
    \emph{associativity theorem}
    of \textcite[\S6.3]{epsilon-i},
    as explained in \cite[\S7.1.7]{epsilon-i}.

    The relation~\cref{eq-hall-alg-invol}
    follows from the commutativity of the diagram
    \begin{equation}
        \begin{tikzcd}[column sep={7em,between origins}, row sep={3.5em,between origins}]
            \mathcal{X}_{\alpha_1} \times \mathcal{X}_{\alpha_2}
            \ar[d, "(-)^\vee"', "\simeq"]
            & \mathcal{X}_{\alpha_1, \alpha_2}^+
            \ar[d, "(-)^\vee"', "\simeq"] \ar[l, "\mathrm{gr}"'] \ar[r, "\mathrm{ev}_1"]
            & \mathcal{X}_{\alpha_1 + \alpha_2}
            \ar[d, "(-)^\vee"', "\simeq"]
            \\ \mathcal{X}_{\alpha_2^\vee} \times \mathcal{X}_{\alpha_1^\vee}
            & \mathcal{X}_{\alpha_2^\vee, \alpha_1^\vee}^+
            \ar[l, "\mathrm{gr}"'] \ar[r, "\mathrm{ev}_1"]
            & \mathcal{X}_{\alpha_2^\vee + \alpha_1^\vee}
            \rlap{ ,}
        \end{tikzcd}
    \end{equation}
    where $\alpha_1, \alpha_2 \in \uppi_0 (\mathcal{X})$,
    and the middle vertical isomorphism is given by
    the $\mathbb{Z}_2$-action on $\mathrm{Filt} (\mathcal{X})$.

    The relation~\cref{eq-hall-mod-unit}
    follows from the isomorphism
    $\mathcal{X}_{0, \theta}^{\smash{\mathrm{sd}, +}} \simto \mathcal{X}_{\theta}^\mathrm{sd}$
    for $\theta \in \uppi_0 (\mathcal{X}^\mathrm{sd})$.

    For~\cref{eq-hall-mod-assoc},
    we have similar diagrams
    \begin{equation}
        \label{eq-cd-hall-mod-assoc}
        \hspace{-1em}
        \adjustbox{scale=0.95}{$
            \begin{tikzcd}[column sep={3em,between origins}, row sep={3em,between origins}]
                &&
                \mathcal{X}_{\alpha_1, \alpha_2, \theta}^{\smash{\mathrm{sd}, +}}
                \ar[dl] \ar[dr]
                \ar[dd, phantom, "\ulcorner" {pos=.15, rotate=-45, shift={(-.08em,0)}}]
                &&
                \\
                & \mathcal{X}_{\alpha_1, \alpha_2}^+ \times \mathcal{X}_\theta^\mathrm{sd}
                \ar[dl] \ar[dr]
                && \mathcal{X}_{\alpha_1 + \alpha_2, \theta}^{\smash[b]{\mathrm{sd}, +}}
                \ar[dl] \ar[dr]
                \\
                \mathcal{X}_{\alpha_1} {\times} \mathcal{X}_{\alpha_2} {\times} \mathcal{X}_\theta^\mathrm{sd}
                && \mathcal{X}_{\alpha_1 + \alpha_2} {\times} \mathcal{X}_\theta^\mathrm{sd}
                && \mathcal{X}_{\alpha_1 + \alpha_2 + \theta + \smash{\alpha_2^\vee + \alpha_1^\vee}}^{\smash[b]{\mathrm{sd}, +}}
                \rlap{ ,}
            \end{tikzcd}
            \enspace
            \begin{tikzcd}[column sep={3.2em,between origins}, row sep={3em,between origins}]
                &&
                \mathcal{X}_{\alpha_1, \alpha_2, \theta}^{\smash{\mathrm{sd}, +}}
                \ar[dl] \ar[dr]
                \ar[dd, phantom, "\ulcorner" {pos=.15, rotate=-45, shift={(-.08em,0)}}]
                &&
                \\
                & \mathcal{X}_{\alpha_1} \times \mathcal{X}_{\alpha_2, \theta}^{\smash[b]{\mathrm{sd}, +}}
                \ar[dl] \ar[dr]
                && \mathcal{X}_{\alpha_1, \alpha_2 + \theta + \smash{\alpha_2^\vee}}^{\smash[b]{\mathrm{sd}, +}}
                \ar[dl] \ar[dr]
                \\
                \mathcal{X}_{\alpha_1} {\times} \mathcal{X}_{\alpha_2} {\times} \mathcal{X}_\theta^\mathrm{sd}
                && \mathcal{X}_{\alpha_1} {\times} \mathcal{X}_{\alpha_2 + \theta + \smash{\alpha_2^\vee}}^\mathrm{sd}
                && \mathcal{X}_{\alpha_1 + \alpha_2 + \theta + \smash{\alpha_2^\vee + \alpha_1^\vee}}^{\smash[b]{\mathrm{sd}, +}}
                \rlap{ ,}
            \end{tikzcd}
        $}
        \hspace{-3em}
    \end{equation}
    where the pullback squares follow from the associativity theorem
    of \textcite[\S6.3]{epsilon-i}.
    Alternatively, these diagrams can be obtained
    by taking $\mathbb{Z}_2$-fixed loci in pullback diagrams analogous
    to~\cref{eq-cd-hall-assoc} for $5$-step filtrations.
    The relation~\cref{eq-hall-mod-assoc} then follows
    from applying the base change formula~\cref{eq-motive-base-change}
    to these diagrams.
\end{proof}

\section{Invariants}

\label{sec-inv}

In this section, we present the definition of
\emph{orthosymplectic DT invariants},
as a special case of the
\emph{intrinsic DT invariants} of
\textcite{epsilon-ii},
which is the main construction of this paper.
The new input here, compared to the cited work,
is the choice of coefficients,
or \emph{stability measures} in the sense of \cite{epsilon-ii},
when defining the epsilon motives in the orthosymplectic setting.
We also use the motivic Hall algebra and module
to make the construction more explicit in our setting.

\subsection{Epsilon motives}
\label{subsec-epsilon}

\begin{para}
    We define \emph{epsilon motives}
    for linear and self-dual linear stacks,
    following \textcite{joyce-2008-configurations-iv} in the linear case
    and the construction of \textcite{epsilon-ii}
    for general algebraic stacks.
    These are elements of the rings of motives
    $\mathbb{M} (\mathcal{X}; \mathbb{Q})$ and
    $\mathbb{M} (\mathcal{X}^\mathrm{sd}; \mathbb{Q})$,
    depending on a stability condition~$\tau$,
    and are obtained from motives of semistable loci,
    $[\mathcal{X}^\mathrm{ss}_\alpha (\tau)]$ and
    $[\mathcal{X}^{\smash{\mathrm{sd,ss}}}_\theta (\tau)]$,
    by removing certain parts of the strictly semistable locus.
    The purpose of doing this step is so that the
    \emph{no-pole theorem}, \cref{thm-no-pole},
    holds, allowing us to take the Euler characteristics of epsilon motives,
    which will then be used to define DT invariants.

    Throughout, we assume that~$\mathcal{X}$ is a linear stack over~$K$,
    with quasi-compact filtrations as in \cref{para-linear-filt}.
\end{para}

\begin{para}[The linear case]
    \label{para-epsilon-linear}
    Let~$\tau$ be a permissible stability condition on~$\mathcal{X}$.
    Following~\textcite{joyce-2008-configurations-iv},
    for each class~$\alpha \in \uppi_0 (\mathcal{X}) \setminus \{ 0 \}$,
    define the \emph{epsilon motive}
    $\epsilon_\alpha (\tau) \in \mathbb{M} (\mathcal{X}_\alpha; \mathbb{Q})$
    by the formula
    \begin{equation}
        \label{eq-def-epsilon-linear}
        \epsilon_\alpha (\tau) =
        \sum_{ \leftsubstack{
            \\[-4ex]
            & n > 0; \,
            \alpha_1, \dotsc, \alpha_n \in
            \uppi_0 (\mathcal{X}) \setminus \{ 0 \} \colon \\[-2ex]
            & \alpha = \alpha_1 + \cdots + \alpha_n, \\[-2ex]
            & \tau (\alpha_1) = \cdots = \tau (\alpha_n)
        }}
        \frac{(-1)^{n-1}}{n} \cdot
        [\mathcal{X}^\mathrm{ss}_{\alpha_1} (\tau)] *
        \cdots *
        [\mathcal{X}^\mathrm{ss}_{\alpha_n} (\tau)] \ ,
    \end{equation}
    where~$*$ denotes multiplication in the motivic Hall algebra
    $\mathbb{M} (\mathcal{X}; \mathbb{Q})$.
    By \cref{lemma-finite-decomposition},
    only finitely many terms in the sum are non-zero.
    Note that $\epsilon_\alpha (\tau)$ is supported on
    $\mathcal{X}^\mathrm{ss}_\alpha (\tau)$.

    Formally inverting the formula~\cref{eq-def-epsilon-linear},
    we obtain the relation
    \begin{equation}
        \label{eq-def-epsilon-linear-inv}
        [\mathcal{X}^\mathrm{ss}_\alpha (\tau)] =
        \sum_{ \leftsubstack{
            \\[-4ex]
            & n > 0; \,
            \alpha_1, \dotsc, \alpha_n \in
            \uppi_0 (\mathcal{X}) \setminus \{ 0 \} \colon \\[-2ex]
            & \alpha = \alpha_1 + \cdots + \alpha_n, \\[-2ex]
            & \tau (\alpha_1) = \cdots = \tau (\alpha_n)
        }}
        \frac{1}{n!} \cdot
        \epsilon_{\alpha_1} (\tau) *
        \cdots *
        \epsilon_{\alpha_n} (\tau) \ .
    \end{equation}
    The relation between the coefficients~$(-1)^{n-1} / n$ and~$1 / n!$
    are explained in \cref{para-epsilon-coefficients} below.

    One can also combine~\cref{eq-def-epsilon-linear-inv}
    with the relation
    \begin{equation}
        [\mathcal{X}_\alpha] =
        \sum_{ \leftsubstack{
            & n > 0; \,
            \alpha_1, \dotsc, \alpha_n \in
            \uppi_0 (\mathcal{X}) \setminus \{ 0 \} \colon \\[-2ex]
            & \alpha = \alpha_1 + \cdots + \alpha_n, \\[-2ex]
            & \tau (\alpha_1) > \cdots > \tau (\alpha_n)
        }}
        [\mathcal{X}^\mathrm{ss}_{\alpha_1} (\tau)] *
        \cdots *
        [\mathcal{X}^\mathrm{ss}_{\alpha_n} (\tau)] \ ,
    \end{equation}
    which comes from the $\Theta$-stratification of~$\mathcal{X}$,
    and can be an infinite but locally finite sum,
    giving the formula
    \begin{equation}
        \label{eq-def-epsilon-linear-inv-2}
        [\mathcal{X}_\alpha] =
        \sum_{ \leftsubstack{
            \\[-3ex]
            & n > 0; \,
            \alpha_1, \dotsc, \alpha_n \in
            \uppi_0 (\mathcal{X}) \setminus \{ 0 \} \colon \\[-2ex]
            & \alpha = \alpha_1 + \cdots + \alpha_n, \\[-2ex]
            & \tau (\alpha_1) \geq \cdots \geq \tau (\alpha_n)
        }}
        \frac{1}{|W_{\alpha_1, \dotsc, \alpha_n} (\tau)|} \cdot
        \epsilon_{\alpha_1} (\tau) *
        \cdots *
        \epsilon_{\alpha_n} (\tau) \ ,
    \end{equation}
    where~$W_{\alpha_1, \dotsc, \alpha_n} (\tau)$
    denotes the group of permutations~$\sigma$ of~$\{ 1, \dotsc, n \}$
    such that~$\tau (\alpha_{\sigma (1)}) \geq \cdots \geq \tau (\alpha_{\sigma (n)})$.
    This can be taken as an alternative definition
    of the invariants~$\epsilon_\alpha (\tau)$,
    that is, they are the unique set of motives
    such that~\cref{eq-def-epsilon-linear-inv-2} holds for all~$\alpha$.

    One can interpret~\cref{eq-def-epsilon-linear-inv-2}
    as considering a generalized version of HN~filtrations,
    where the slopes of the quotients are non-increasing
    rather than strictly decreasing,
    and the sum is averaged over all possible orderings
    satisfying the non-increasing condition.

    The motive $\epsilon_\alpha (\tau)$
    agrees with the motive denoted by
    $\epsilon^{(1)}_{\mathcal{X}_\alpha} (\mu_\tau)$
    in \cite[\S5.2]{epsilon-ii},
    where~$\mu_\tau$ is the \emph{stability measure}
    associated to the stability condition~$\tau$,
    as explained in \cite[Example~4.1.7]{epsilon-ii}.
\end{para}

\begin{para}[The self-dual case]
    \label{para-epsilon-sd}
    Suppose that~$\mathcal{X}$ is equipped with a self-dual structure,
    and let~$\tau$ be a permissible self-dual stability condition on~$\mathcal{X}$.

    For each class $\theta \in \uppi_0 (\mathcal{X}^\mathrm{sd})$,
    define the \emph{epsilon motive}
    $\epsilon^\mathrm{sd}_\theta (\tau) \in
    \mathbb{M} (\mathcal{X}^\mathrm{sd}_\theta; \mathbb{Q})$
    by the formula
    \begin{equation}
        \label{eq-def-epsilon-sd}
        \epsilon^\mathrm{sd}_\theta (\tau) =
        \sum_{ \leftsubstack{
            \\[-3.5ex]
            & n \geq 0; \,
            \alpha_1, \dotsc, \alpha_n \in
            \uppi_0 (\mathcal{X}) \setminus \{ 0 \}, \,
            \rho \in \uppi_0 (\mathcal{X}^\mathrm{sd}) \colon \\[-2ex]
            & \theta = \alpha_1 + \alpha_1^\vee + \cdots +
            \alpha_n + \alpha_n^\vee + \rho, \\[-2ex]
            & \tau (\alpha_1) = \cdots = \tau (\alpha_n) = 0
        }}
        \binom{-1/2}{n} \cdot
        [\mathcal{X}^\mathrm{ss}_{\alpha_1} (\tau)] \diamond
        \cdots \diamond
        [\mathcal{X}^\mathrm{ss}_{\alpha_n} (\tau)] \diamond
        [\mathcal{X}^{\mathrm{sd}, \mathrm{ss}}_\rho (\tau)] \ ,
    \end{equation}
    where~$\diamond$ denotes the multiplication for the motivic Hall module,
    the notation $\alpha_i + \alpha_i^\vee$ is from
    \cref{para-sd-linear-moduli-stacks},
    and~$\binom{-1/2}{n}$ is the binomial coefficient.
    The sum only contains finitely many non-zero terms,
    and~$\epsilon^\mathrm{sd}_\theta (\tau)$
    is supported on the semistable locus
    $\mathcal{X}^{\mathrm{sd}, \mathrm{ss}}_\theta (\tau)
    \subset \mathcal{X}^\mathrm{sd}_\theta$.

    Formally inverting the formula~\cref{eq-def-epsilon-sd},
    we obtain the relation
    \begin{equation}
        \label{eq-def-epsilon-sd-inv}
        [\mathcal{X}^{\mathrm{sd}, \mathrm{ss}}_\theta (\tau)] =
        \sum_{ \leftsubstack{
            \\[-3.5ex]
            & n \geq 0; \,
            \alpha_1, \dotsc, \alpha_n \in
            \uppi_0 (\mathcal{X}) \setminus \{ 0 \}, \,
            \rho \in \uppi_0 (\mathcal{X}^\mathrm{sd}) \colon \\[-2ex]
            & \theta = \alpha_1 + \alpha_1^\vee + \cdots +
            \alpha_n + \alpha_n^\vee + \rho, \\[-2ex]
            & \tau (\alpha_1) = \cdots = \tau (\alpha_n) = 0
        }}
        \frac{1}{2^n \, n!} \cdot
        \epsilon_{\alpha_1} (\tau) \diamond
        \cdots \diamond
        \epsilon_{\alpha_n} (\tau) \diamond
        \epsilon^\mathrm{sd}_\rho (\tau) \ ,
    \end{equation}
    which we explain further in \cref{para-epsilon-coefficients}.
    This can be combined with the relation
    \begin{equation}
        [\mathcal{X}^\mathrm{sd}_\theta] =
        \sum_{ \leftsubstack{
            & n \geq 0; \,
            \alpha_1, \dotsc, \alpha_n \in
            \uppi_0 (\mathcal{X}) \setminus \{ 0 \}, \,
            \rho \in \uppi_0 (\mathcal{X}^\mathrm{sd}) \colon \\[-2ex]
            & \theta = \alpha_1 + \alpha_1^\vee + \cdots +
            \alpha_n + \alpha_n^\vee + \rho, \\[-2ex]
            & \tau (\alpha_1) > \cdots > \tau (\alpha_n) > 0
        }}
        [\mathcal{X}^\mathrm{ss}_{\alpha_1} (\tau)] \diamond
        \cdots \diamond
        [\mathcal{X}^\mathrm{ss}_{\alpha_n} (\tau)] \diamond
        [\mathcal{X}^{\mathrm{sd}, \mathrm{ss}}_\rho (\tau)]
    \end{equation}
    from the $\Theta$-stratification of~$\mathcal{X}^\mathrm{sd}$,
    together with~\cref{eq-def-epsilon-linear-inv},
    to obtain the formula
    \begin{equation}
        \label{eq-def-epsilon-sd-inv-2}
        [\mathcal{X}^\mathrm{sd}_\theta] =
        \sum_{ \leftsubstack{
            \\[-2ex]
            & n \geq 0; \,
            \alpha_1, \dotsc, \alpha_n \in
            \uppi_0 (\mathcal{X}) \setminus \{ 0 \}, \,
            \rho \in \uppi_0 (\mathcal{X}^\mathrm{sd}) \colon \\[-2ex]
            & \theta = \alpha_1 + \alpha_1^\vee + \cdots +
            \alpha_n + \alpha_n^\vee + \rho, \\[-2ex]
            & \tau (\alpha_1) \geq \cdots \geq \tau (\alpha_n) \geq 0
        }}
        \frac{1}{|W^\mathrm{sd}_{\alpha_1, \dotsc, \alpha_n} (\tau)|} \cdot
        \epsilon_{\alpha_1} (\tau) \diamond
        \cdots \diamond
        \epsilon_{\alpha_n} (\tau) \diamond
        \epsilon^\mathrm{sd}_\rho (\tau) \ ,
    \end{equation}
    where~$W^\mathrm{sd}_{\alpha_1, \dotsc, \alpha_n} (\tau)$
    is the group of permutations~$\sigma$
    of~$\{ 1, \dotsc, n, n^\vee, \dotsc, 1^\vee \}$,
    such that $\sigma (i)^\vee = \sigma (i^\vee)$ for all~$i$,
    where we set $(i^\vee)^\vee = i$,
    satisfying the non-increasing condition
    $\tau (\alpha_{\sigma (1)}) \geq \cdots \geq \tau (\alpha_{\sigma (n)}) \geq 0$,
    where we set $\alpha_{i^\vee} = \alpha_i^\vee$.
    For example, we have
    $|W^\mathrm{sd}_{\alpha_1, \dotsc, \alpha_n} (\tau)| = 2^n \, n!$
    if $\tau (\alpha_1) = \cdots = \tau (\alpha_n) = 0$.

    The coefficients
    $1 / |W^\mathrm{sd}_{\alpha_1, \dotsc, \alpha_n} (\tau)|$
    in~\cref{eq-def-epsilon-sd-inv-2}
    can be seen as defining a \emph{stability measure}~$\mu_\tau^\mathrm{sd}$
    on~$\mathcal{X}_\theta^\mathrm{sd}$,
    in the sense of \cite{epsilon-ii}.
    The motive $\epsilon^\mathrm{sd}_\theta (\tau)$
    agrees with the motive
    $\epsilon^{(0)}_{\mathcal{X}^{\smash{\mathrm{sd}}}_\theta} (\mu_\tau^\mathrm{sd})$
    in \cite[\S5.2]{epsilon-ii}.
\end{para}

\begin{para}[Explanations of the coefficients]
    \label{para-epsilon-coefficients}
    The relations between the coefficients in
    \cref{eq-def-epsilon-linear},
    \cref{eq-def-epsilon-linear-inv},
    \cref{eq-def-epsilon-sd},
    and \cref{eq-def-epsilon-sd-inv},
    can be seen more directly by setting
    \begin{alignat*}{2}
        \delta (\tau; t)
        & = [\{ 0 \}] + \sum_{\substack{
                \alpha \in \uppi_0 (\mathcal{X}) \setminus \{ 0 \} \mathrlap{:} \\
            \tau (\alpha) = t
        }} {}
        [\mathcal{X}_\alpha^\mathrm{ss} (\tau)] \ ,
        & \qquad
        \delta^\mathrm{sd} (\tau)
        & = \sum_{\theta \in \uppi_0 (\smash{\mathcal{X}^\mathrm{sd}})} {}
        [\mathcal{X}^{\smash{\mathrm{sd,ss}}}_\theta (\tau)] \ ,
        \\
        \epsilon (\tau; t)
        & = \sum_{\substack{
                \alpha \in \uppi_0 (\mathcal{X}) \setminus \{ 0 \} \mathrlap{:} \\
            \tau (\alpha) = t
        }}
        \epsilon_\alpha (\tau) \ ,
        & \qquad
        \epsilon^\mathrm{sd} (\tau)
        & = \sum_{\theta \in \uppi_0 (\smash{\mathcal{X}^\mathrm{sd}})}
        \epsilon^\mathrm{sd}_\theta (\tau) \ ,
    \end{alignat*}
    as motives on~$\mathcal{X}$ or~$\mathcal{X}^\mathrm{sd}$,
    where $t \in T$,
    so that these relations can be rewritten as
    \begin{alignat*}{2}
        \epsilon (\tau; t)
        & = \log \delta (\tau; t) \ ,
        & \qquad
        \epsilon^\mathrm{sd} (\tau)
        & = \delta (\tau; 0)^{-1/2} \diamond \delta^\mathrm{sd} (\tau) \ ,
        \\
        \delta (\tau; t)
        & = \exp \epsilon (\tau; t) \ ,
        & \qquad
        \delta^\mathrm{sd} (\tau)
        & = \exp \Bigl( \frac{1}{2} \epsilon (\tau; 0) \Bigr) \diamond
        \epsilon^\mathrm{sd} (\tau) \ ,
    \end{alignat*}
    where we take formal power series using the product
    in the motivic Hall algebra.

    The coefficients~$(-1)^{n-1} / n$ and~$\smash{\binom{-1/2}{n}}$
    in~\cref{eq-def-epsilon-linear,eq-def-epsilon-sd}
    are determined by the coefficients~$1 / n!$ and~$1 / (2^n n!)$
    in \cref{eq-def-epsilon-linear-inv,eq-def-epsilon-sd-inv}
    in this way.
    They are the unique choice
    of coefficients only depending on~$n$,
    such that the \emph{no-pole theorem}, \cref{thm-no-pole},
    holds for the epsilon motives.
    The rough reason for this is that
    they ensure the combinatorial descriptions of the coefficients
    $1 / |W_{\alpha_1, \dotsc, \alpha_n} (\tau)|$
    and $1 / |W^\mathrm{sd}_{\alpha_1, \dotsc, \alpha_n} (\tau)|$
    in \cref{eq-def-epsilon-linear-inv-2,eq-def-epsilon-sd-inv-2},
    and from the viewpoint of~\cite{epsilon-ii},
    the no-pole theorem corresponds to the property that
    these coefficients
    sum up to~$1$ for all permutations~$\sigma$
    as described for each of them, for fixed classes~$\alpha_i$.
\end{para}

\begin{para}[Remark on permissibility]
    \label{para-permissibility-comparison}
    In the situations above,
    the permissibility of the stability condition~$\tau$
    implies that the stability measures~$\mu_\tau$ and~$\mu_\tau^\mathrm{sd}$
    are permissible in the sense of \cite[\S4.1.4]{epsilon-ii},
    which follows from \cite[Lemma~5.4.8]{epsilon-ii}.
\end{para}

\begin{para}[The no-pole theorem]
    \label[theorem]{thm-no-pole}
    A key property of the epsilon motives
    is the \emph{no-pole theorem},
    which states that they have pure virtual ranks
    in the sense of~\cref{para-virtual-rank}.
    This will allow us to define numerical invariants,
    including DT invariants,
    by taking their Euler characteristics.
\end{para}

\begin{theorem*}
    Let $\mathcal{X}$ be a linear stack over~$K$,
    with quasi-compact filtrations.

    \begin{enumerate}
        \item
            \label{item-thm-no-pole-linear}
            For any permissible stability condition~$\tau$ on~$\mathcal{X}$,
            and any $\alpha \in \uppi_0 (\mathcal{X}) \setminus \{ 0 \}$,
            the motive
            $\epsilon_\alpha (\tau)$
            has pure virtual rank~$1$.

        \item
            \label{item-thm-no-pole-sd}
            If\/~$\mathcal{X}$ is equipped with a self-dual structure,
            then for any permissible self-dual stability condition~$\tau$ on~$\mathcal{X}$,
            and any $\theta \in \uppi_0 (\mathcal{X}^\mathrm{sd})$,
            the motive
            $\epsilon^\mathrm{sd}_\theta (\tau)$
            has pure virtual rank~$0$.
    \end{enumerate}
\end{theorem*}

These are special cases of the general no-pole theorem
for intrinsic DT invariants
in \cite[Theorem~5.3.7]{epsilon-ii},
and we refer to the cited work for the proof.
The linear case~\cref{item-thm-no-pole-linear}
was originally proved by
\textcite[Theorem~8.7]{joyce-2007-configurations-iii},
under a slightly different setting.
The self-dual case~\cref{item-thm-no-pole-sd}
was originally proved in an earlier version of this paper,
\cite[Appendix~E]{bu-self-dual-i},
under another slightly different setting.

\subsection{DT invariants}
\label{subsec-dt}

\begin{para}
    We now turn to the definition of \emph{DT invariants}
    for linear and self-dual linear stacks,
    the main construction of this paper.
    The linear case was first due to
    \textcite{joyce-song-2012} and
    \textcite{kontsevich-soibelman-motivic-dt},
    and the self-dual case was first constructed by the author
    in an earlier version of this paper \cite{bu-self-dual-i}.
    Here, we continue to follow the general construction of
    \textcite{epsilon-ii},
    specializing it to the self-dual linear case.

    Throughout this section,
    we assume that the base field~$K$ is algebraically closed
    and has characteristic zero.
    We work with \emph{$(-1)$-shifted symplectic stacks} over~$K$
    in the sense of \textcite{pantev-toen-vaquie-vezzosi-2013},
    which are derived algebraic stacks
    locally finitely presented over~$K$,
    equipped with a $(-1)$-shifted symplectic form~$\omega$.
\end{para}

\begin{para}[Local structure]
    \label{para-local-structure}
    Following \textcite[\S2.2.4]{bu-integral},
    we introduce the following local conditions on algebraic stacks.

    A stack is \emph{étale} (or~\emph{Nisnevich}) \emph{locally a quotient stack},
    if it admits a representable étale (or~Nisnevich) cover
    by quotient stacks of the form $U / \mathrm{GL} (n)$,
    with~$U$ an algebraic space.

    A stack is \emph{étale} (or~\emph{Nisnevich}) \emph{locally fundamental},
    if it admits a representable étale (or~Nisnevich) cover
    by quotient stacks of the form $U / \mathrm{GL} (n)$,
    with~$U$ an affine scheme.

    These conditions are preserved by taking $\mathbb{Z}_2$-fixed points
    under the assumptions in \cref{para-intro-conventions},
    by \cref{lemma-fixed-affine}.
\end{para}

\begin{para}[Derived linear stacks]
    \label{para-derived-linear-stacks}
    Following \textcite[\S2.4.6]{bu-davison-ibanez-nunez-kinjo-padurariu},
    define a \emph{derived linear stack} over~$K$
    to be a derived algebraic stack~$\mathcal{X}$,
    locally finitely presented over~$K$,
    equipped with a monoid structure~$\oplus$
    and a compatible $*/\mathbb{G}_\mathrm{m}$-action~$\odot$,
    such that the isomorphism~\cref{eq-linear-moduli-stack-grad}
    holds without taking the classical truncations,
    where we use the derived stack of graded points
    $\mathrm{Grad} (\mathcal{X})$,
    defined as the derived mapping stack from~$*/\mathbb{G}_\mathrm{m}$ to~$\mathcal{X}$.

    As in \cite[\S3.1.7]{bu-davison-ibanez-nunez-kinjo-padurariu},
    define a \emph{$(-1)$-shifted symplectic linear stack} over~$K$
    to be a derived linear stack~$\mathcal{X}$ as above,
    equipped with a $(-1)$-shifted symplectic form~$\omega$,
    such that there exists an equivalence
    $\oplus^* (\omega) \simeq \omega \boxplus \omega$
    on $\mathcal{X} \times \mathcal{X}$,
    where we do not require extra coherence conditions.

    We further assume that the classical truncation~$\mathcal{X}_\mathrm{cl}$
    of~$\mathcal{X}$
    satisfies the conditions in \cref{para-intro-conventions},
    has quasi-compact filtrations,
    and is étale locally a quotient stack.

    We will often denote $\mathcal{X}_{\mathrm{cl}}$ simply by~$\mathcal{X}$.
\end{para}

\begin{para}[The linear case]
    \label{para-dt-linear}
    Let~$\tau$ be a permissible stability condition on~$\mathcal{X}$.
    Following the construction of
    \textcite[Definition~5.15]{joyce-song-2012},
    but adapting it to our setting of linear stacks,
    for a class $\alpha \in \uppi_0 (\mathcal{X}) \setminus \{ 0 \}$,
    define the \emph{DT invariant}
    $\mathrm{DT}_\alpha (\tau) \in \mathbb{Q}$ by the formula
    \begin{equation}
        \label{eq-def-dt-linear}
        \mathrm{DT}_\alpha (\tau) =
        \int_{\mathcal{X}_\alpha} {}
        (1 - \mathbb{L}) \cdot
        \epsilon_\alpha (\tau) \cdot \nu_{\mathcal{X}} \, d \chi \ ,
    \end{equation}
    where the notation $\int {} (-) \, d \chi$
    is defined in \cref{para-virtual-rank},
    and~$\nu_\mathcal{X}$ is the \emph{Behrend function} of~$\mathcal{X}$,
    which is a constructible function on~$\mathcal{X}$
    only depending on $\mathcal{X}_\mathrm{cl}$,
    as in \textcite[\S4.1]{joyce-song-2012}
    or \textcite[\S2.5.6]{bu-integral},
    and originally due to
    \textcite{behrend-2009-dt} for Deligne--Mumford stacks.

    This integral is well-defined
    since $\epsilon_\alpha (\tau)$ is supported on the semistable locus
    $\mathcal{X}^\mathrm{ss}_\alpha (\tau)$, which is quasi-compact,
    and by the no-pole theorem,
    \cref{thm-no-pole}~\cref{item-thm-no-pole-linear}.
\end{para}

\begin{para}[The self-dual case]
    \label{para-dt-sd}
    Assume further that~$\mathcal{X}$ is equipped with
    a $\mathbb{Z}_2$-action, preserving the
    $(-1)$-shifted symplectic form,
    such that the induced $\mathbb{Z}_2$-action on~$\mathcal{X}$
    establishes it as a self-dual linear stack.

    Let~$\tau$ be a permissible self-dual stability condition on~$\mathcal{X}$.
    For a class~$\theta \in \uppi_0 (\mathcal{X}^\mathrm{sd})$,
    define the \emph{self-dual DT invariant}
    $\mathrm{DT}^\mathrm{sd}_\theta (\tau) \in \mathbb{Q}$
    by the formula
    \begin{equation}
        \label{eq-def-dt-sd}
        \mathrm{DT}^\mathrm{sd}_\theta (\tau) =
        \int_{\mathcal{X}^\mathrm{sd}_\theta}
        \epsilon^\mathrm{sd}_\theta (\tau) \cdot
        \nu_{\subXsd} \, d \chi \ .
    \end{equation}
    Again, this is well-defined by the fact that
    $\epsilon^\mathrm{sd}_\theta (\tau)$ is supported on
    $\mathcal{X}^{\smash{\mathrm{sd}, \mathrm{ss}}}_\theta (\tau)$,
    which is quasi-compact,
    and by the no-pole theorem,
    \cref{thm-no-pole}~\cref{item-thm-no-pole-sd}.

    This is one of the main constructions of this paper,
    and is a special case of the intrinsic DT invariants
    in \cite[\S6.1]{epsilon-ii}
    for the stability measure~$\mu_\tau^\mathrm{sd}$
    described in \cref{para-epsilon-sd}.
\end{para}

\begin{para}[For smooth stacks]
    \label{para-dt-smooth-stacks}
    Let~$\mathcal{X}$ be a classical smooth linear stack
    which is étale locally a quotient stack,
    and consider its $(-1)$-shifted cotangent stack
    $\mathrm{T}^* [-1] \, \mathcal{X}$,
    which has a canonical $(-1)$-shifted symplectic structure,
    making it a $(-1)$-shifted symplectic linear stack.
    We have $(\mathrm{T}^* [-1] \, \mathcal{X})_\mathrm{cl} \simeq \mathcal{X}$.
    If~$\mathcal{X}$ is equipped with a self-dual structure,
    then the fixed locus~$\mathcal{X}^\mathrm{sd}$ is also smooth,
    and $(\mathrm{T}^* [-1] \, \mathcal{X})^\mathrm{sd} \simeq
    \mathrm{T}^* [-1] \, \mathcal{X}^\mathrm{sd}$.

    In this case,
    we have $\nu_\mathcal{X} = (-1)^{\dim \mathcal{X}}$
    and $\nu_{\subXsd} = (-1)^{\dim \mathcal{X}^\mathrm{sd}}$,
    and \crefrange{eq-def-dt-linear}{eq-def-dt-sd}
    become
    \begin{align}
        \label{eq-dt-smooth}
        \mathrm{DT}_\alpha (\tau)
        & =
        (-1)^{\dim \mathcal{X}_\alpha} \cdot
        \int _{\mathcal{X}_\alpha} {}
        (1 - \mathbb{L}) \cdot \epsilon_\alpha (\tau) \, d \chi \ ,
        \\
        \label{eq-dt-smooth-sd}
        \mathrm{DT}^{\smash{\mathrm{sd}}}_\theta (\tau)
        & =
        (-1)^{\dim \mathcal{X}^\mathrm{sd}_\theta} \cdot
        \int _{\mathcal{X}^\mathrm{sd}_\theta} {}
        \epsilon^\mathrm{sd}_\theta (\tau) \, d \chi \ .
    \end{align}
    The invariants~$\mathrm{DT}_\alpha (\tau)$
    are essentially the same as
    those defined by \textcite[\S6.2]{joyce-2008-configurations-iv},
    denoted by $J^\alpha (\tau)^\Omega$ there,
    while the invariants~$\mathrm{DT}^\mathrm{sd}_\theta (\tau)$ are new.

    Note that the formulae \crefrange{eq-dt-smooth}{eq-dt-smooth-sd}
    also make sense for smooth (self-dual)
    linear stacks over an arbitrary base field~$K$,
    allowing us to also define DT invariants in this case.
\end{para}

\subsection{Motivic DT invariants}
\label{subsec-motivic-dt}

\begin{para}
    We also introduce motivic enhancements of
    the linear and self-dual DT invariants defined above,
    following the formalism of \textcite[\S6.2]{epsilon-ii}.
    These generalize the construction of
    \textcite{kontsevich-soibelman-motivic-dt}
    in the linear case.
\end{para}

\begin{para}[Monodromic motives]
    \label{para-monodromic-motives}
    For a stack~$\mathcal{X}$ over~$K$
    and a commutative ring~$A$,
    we have the ring of \emph{monodromic motives},
    denoted by $\hat{\mathbb{M}}^{\mathrm{mon}} (\mathcal{X}; A)$.
    It is similar to $\hat{\mathbb{M}} (\mathcal{X}; A)$
    defined in~\cref{para-motivic-integration},
    but its elements have the additional structure of a monodromy action.
    See \cite[\S6.2.2]{epsilon-ii} for details.

    When~$A$ contains~$\mathbb{Q}$, there is an Euler characteristic map
    $\chi \colon \hat{\mathbb{M}}^{\mathrm{mon}, \mathrm{reg}} (K; A) \to A$,
    where
    $\hat{\mathbb{M}}^{\mathrm{mon}, \mathrm{reg}} (K; A)
    \subset \hat{\mathbb{M}}^{\mathrm{mon}} (K; A)$
    is the subspace of elements that are regular at~$\mathbb{L} = 1$,
    defined similarly to \cref{para-euler-characteristic}.

    There is an element
    $\mathbb{L}^{1/2} \in \hat{\mathbb{M}}^{\mathrm{mon, reg}} (K; A)$
    satisfying $(\mathbb{L}^{1/2})^2 = \mathbb{L}$
    and $\chi (\mathbb{L}^{1/2}) = -1$.
\end{para}

\begin{para}[Orientations]
    \label{para-orientations}
    For a $(-1)$-shifted symplectic stack~$\mathcal{X}$ over~$K$,
    the \emph{canonical bundle} of~$\mathcal{X}$
    is the determinant line bundle of its cotangent complex,
    $K_{\mathcal{X}} = \det \mathbb{L}_{\mathcal{X}}$.

    An \emph{orientation} of~$\mathcal{X}$
    is a line bundle $K_\mathcal{X}^{\smash{1/2}}$ on~$\mathcal{X}$,
    with an isomorphism
    $o_\mathcal{X} \colon
    (K_\mathcal{X}^{\smash{1/2}})^{\otimes 2} \simto K_\mathcal{X}$.
    We sometimes abbreviate the pair
    $(K_\mathcal{X}^{\smash{1/2}}, o_\mathcal{X})$
    as~$o_\mathcal{X}$.

    Given such an orientation,
    if the classical truncation
    $\mathcal{X}_\mathrm{cl}$ is Nisnevich locally a quotient stack
    as in \cref{para-local-structure},
    then there is an element
    $\nu_\mathcal{X}^{\mathrm{mot}} \in
    \hat{\mathbb{M}}^{\mathrm{mot}} (\mathcal{X}; \mathbb{Z})$,
    as in \textcite[\S2.5.4]{bu-integral},
    called the \emph{motivic Behrend function},
    originally constructed by
    \textcite{bussi-joyce-meinhardt-2019}
    and \textcite{ben-bassat-brav-bussi-joyce-2015-darboux}.
\end{para}

\begin{para}[Orientation data]
    \label{para-orientation-data}
    By \textcite[Theorem~3.1.6]{bu-integral}
    or \cite[\S6.1.6]{bu-davison-ibanez-nunez-kinjo-padurariu},
    an orientation~$o_\mathcal{X}$ induces an orientation
    $o_{\mathrm{Grad} (\mathcal{X})}$
    of $\mathrm{Grad} (\mathcal{X})$.
    An orientation $o_\mathcal{X}$
    is called an \emph{orientation data},
    if it satisfies the following compatibility condition:

    \begin{itemize}
        \item
            Under the isomorphism
            \cref{eq-linear-moduli-stack-grad},
            the induced orientation
            $o_{\mathrm{Grad} (\mathcal{X})}$
            of $\mathrm{Grad} (\mathcal{X})$
            agrees with the product orientations on the left-hand side.
    \end{itemize}
    By \textcite[Theorem~3.6]{joyce-upmeier-2021-orientation-data},
    such an orientation data exists canonically
    on moduli stacks of coherent sheaves on
    Calabi--Yau threefolds.
\end{para}

\begin{para}[Self-dual orientation data]
    \label{para-self-dual-orientation-data}
    Now, let~$\mathcal{X}$ be a
    \emph{self-dual $(-1)$-shifted symplectic linear stack},
    that is, a stack~$\mathcal{X}$ as in~\cref{para-orientation-data},
    equipped with a $\mathbb{Z}_2$-action
    preserving the symplectic form~$\omega$,
    compatible with the monoid structure~$\oplus$
    and inverting the $*/\mathbb{G}_\mathrm{m}$-action~$\odot$.

    In this case, the fixed locus
    $\mathcal{X}^\mathrm{sd} = \mathcal{X}^{\mathbb{Z}_2}$
    carries an induced $(-1)$-shifted symplectic structure.
    However, an orientation of~$\mathcal{X}$
    does not naturally induce one on~$\mathcal{X}^\mathrm{sd}$.

    We define a \emph{self-dual orientation data} on~$\mathcal{X}$
    to be a pair $(o_\mathcal{X}, o_{\smash{\mathcal{X}^\mathrm{sd}}})$
    of orientations of~$\mathcal{X}$ and~$\mathcal{X}^\mathrm{sd}$,
    respectively, satisfying the following conditions:

    \begin{enumerate}
        \item
            $o_\mathcal{X}$ is an orientation data.
        \item
            Under the isomorphism~\cref{eq-linear-moduli-stack-sd-grad},
            the induced orientation of~$\mathrm{Grad} (\mathcal{X}^\mathrm{sd})$
            agrees with the product orientations on the right-hand side.
    \end{enumerate}
    The author does not know if such a self-dual orientation data,
    or even an orientation,
    exists in the case of coherent sheaves on Calabi--Yau threefolds,
    which we will discuss in \cref{subsec-threefolds} below.
\end{para}

\begin{para}[Motivic DT invariants]
    \label{para-motivic-dt}
    Let~$\mathcal{X}$ be a $(-1)$-shifted symplectic linear stack over~$K$,
    equipped with an orientation data as in \cref{para-orientation-data}.
    Assume that its classical truncation $\mathcal{X}_\mathrm{cl}$
    is Nisnevich locally a quotient stack,
    as in \cref{para-local-structure}.

    For a permissible stability condition~$\tau$ on~$\mathcal{X}$,
    and a class $\alpha \in \uppi_0 (\mathcal{X}) \setminus \{ 0 \}$,
    following the construction of
    \textcite{kontsevich-soibelman-motivic-dt},
    define the \emph{motivic DT invariant}
    $\mathrm{DT}^\mathrm{mot}_\alpha (\tau) \in
    \hat{\mathbb{M}}^\mathrm{mot} (K; \mathbb{Q})$
    by the formula
    \begin{equation}
        \label{eq-def-dt-mot}
        \mathrm{DT}^\mathrm{mot}_\alpha (\tau) =
        \int_{\mathcal{X}_\alpha} {}
        (\mathbb{L}^{1/2} - \mathbb{L}^{-1/2}) \cdot
        \epsilon_\alpha (\tau) \cdot \nu^\mathrm{mot}_{\mathcal{X}} \ ,
    \end{equation}
    where $\nu^\mathrm{mot}_{\mathcal{X}}$
    is the motivic Behrend function of~$\mathcal{X}$
    defined in \cref{para-orientations}.

    Now, suppose further that~$\mathcal{X}$
    is equipped with a self-dual structure
    as in \cref{para-self-dual-orientation-data},
    together with a self-dual orientation data.

    For a self-dual permissible stability condition~$\tau$
    and a class~$\theta \in \uppi_0 (\mathcal{X}^\mathrm{sd})$,
    define the \emph{self-dual motivic DT invariant}
    $\mathrm{DT}^{\smash{\mathrm{mot, sd}}}_\theta (\tau) \in
    \hat{\mathbb{M}}^\mathrm{mot} (K; \mathbb{Q})$
    by
    \begin{equation}
        \label{eq-def-dt-mot-sd}
        \mathrm{DT}^{\smash{\mathrm{mot, sd}}}_\theta (\tau) =
        \int_{\mathcal{X}^\mathrm{sd}_\theta} {}
        \epsilon^\mathrm{sd}_\theta (\tau) \cdot
        \nu^\mathrm{mot}_{\mathcal{X}^{\smash{\mathrm{sd}}}} \ .
    \end{equation}
    This is also a main construction of this paper,
    and is a special case of the intrinsic motivic DT invariants
    in \cite[\S6.2]{epsilon-ii}
    for the stability measure~$\mu_\tau^\mathrm{sd}$
    described in \cref{para-epsilon-sd}.
\end{para}

\begin{para}[For smooth stacks]
    \label{para-mot-dt-smooth-stacks}
    Let~$\mathcal{X}$ be a linear stack
    which is smooth and Nisnevich locally a quotient stack,
    and consider its $(-1)$-shifted cotangent stack
    $\mathrm{T}^* [-1] \, \mathcal{X}$,
    as in \cref{para-dt-smooth-stacks}.
    It has a canonical $(-1)$-shifted symplectic linear structure
    and orientation data,
    and in the self-dual case,
    also a canonical self-dual orientation data.

    The motivic Behrend function of~$\mathcal{X}$ is
    $\nu_\mathcal{X}^\mathrm{mot} = \mathbb{L}^{-{\dim \mathcal{X} / 2}}$
    by \textcite[Theorem~2.5.5]{bu-integral},
    where $\dim \mathcal{X}$ refers to the dimension of
    the classical smooth stack~$\mathcal{X}$.
    The formulae \crefrange{eq-def-dt-mot}{eq-def-dt-mot-sd}
    can be simplified to
    \begin{align}
        \label{eq-dt-mot-smooth}
        \mathrm{DT}^{\mathrm{mot}}_\alpha (\tau)
        & =
        \frac{\mathbb{L}^{1/2} - \mathbb{L}^{-1/2}}{\mathbb{L}^{\dim \mathcal{X}_\alpha / 2}} \cdot
        \int _{\mathcal{X}_\alpha} {}
        \epsilon_\alpha (\tau) \ ,
        \\
        \label{eq-dt-mot-smooth-sd}
        \mathrm{DT}^{\smash{\mathrm{mot, sd}}}_\theta (\tau)
        & =
        \mathbb{L}^{-{\dim \mathcal{X}^\mathrm{sd}_\theta / 2}} \cdot
        \int _{\mathcal{X}^\mathrm{sd}_\theta} {}
        \epsilon^\mathrm{sd}_\theta (\tau) \ .
    \end{align}
\end{para}

\section{Wall-crossing}

\label{sec-wcf}

\subsection{Wall-crossing for epsilon motives}
\label{subsec-wcf-epsilon}

\begin{para}
    \label{para-wcf-intro}
    We now discuss how to relate the
    epsilon motives and DT invariants defined in
    \cref{subsec-epsilon,subsec-dt,subsec-motivic-dt}
    when we change the stability condition~$\tau$.
    These relations are called \emph{wall-crossing formulae}.
    We first prove wall-crossing formulae for epsilon motives
    in \cref{thm-wcf-epsilon},
    which we then use in~\cref{subsec-wcf-dt}
    to obtain wall-crossing formulae for DT invariants.

    Throughout, let~$\mathcal{X}$ be a self-dual linear stack
    with quasi-compact filtrations as in \cref{para-linear-filt}.
    Results in the linear case will not need the self-dual structure on~$\mathcal{X}$,
    and we will indicate this when it is the case.
\end{para}

\begin{para}[Dominance of stability conditions]
    For stability conditions $\tau_0, \tau$ on~$\mathcal{X}$,
    following \textcite[Definition~4.10]{joyce-2007-configurations-iii},
    we say that~$\tau_0$ \emph{dominates}~$\tau$,
    if $\tau (\alpha_1) \leq \tau (\alpha_2)$
    implies $\tau_0 (\alpha_1) \leq \tau_0 (\alpha_2)$
    for all $\alpha_1, \alpha_2 \in \pi_0 (\mathcal{X}) \setminus \{ 0 \}$.

    In this case, the $\Theta$-stratification of~$\mathcal{X}$
    given by~$\tau$ refines the one given by~$\tau_0$,
    and in particular, we have
    $\mathcal{X}^\mathrm{ss}_\alpha (\tau) \subset
    \mathcal{X}^\mathrm{ss}_\alpha (\tau_0)$
    for all $\alpha \in \pi_0 (\mathcal{X}) \setminus \{ 0 \}$.

    For example, every stability condition
    is dominated by the trivial stability condition.
\end{para}

\begin{theorem}
    \label{thm-wcf-epsilon}
    \allowdisplaybreaks
    Let $\tau_+, \tau_-, \tau_0$ be
    permissible self-dual stability conditions on~$\mathcal{X}$,
    with $\tau_0$ dominating both~$\tau_+$ and~$\tau_-$.
    Then for any $\alpha \in \pi_0 (\mathcal{X})$
    and $\theta \in \pi_0 (\mathcal{X}^\mathrm{sd})$,
    we have the relations
    \begin{align}
        \label{eq-wcf-delta}
        [\mathcal{X}^\mathrm{ss}_\alpha (\tau_-)]
        & =
        \sum_{ \leftsubstack[5em]{
            & n \geq 0; \, \alpha_1, \dotsc, \alpha_n \in
            \pi_0 (\mathcal{X}) \setminus \{ 0 \} \colon \\[-1ex]
            & \alpha = \alpha_1 + \cdots + \alpha_n 
        } } {}
        S (\alpha_1, \dotsc, \alpha_n; \tau_+, \tau_-) \cdot
        [\mathcal{X}^\mathrm{ss}_{\alpha_1} (\tau_+)] * \cdots *
        [\mathcal{X}^\mathrm{ss}_{\alpha_n} (\tau_+)] \ ,
        \\*[1ex]
        \label{eq-wcf-delta-sd}
        [\mathcal{X}^{\smash{\mathrm{sd,ss}}}_\theta (\tau_-)]
        & =
        \sum_{ \leftsubstack[5em]{
            & n \geq 0; \, \alpha_1, \dotsc, \alpha_n \in
            \pi_0 (\mathcal{X}) \setminus \{ 0 \}, \,
            \rho \in \pi_0 (\mathcal{X}^\mathrm{sd}) \colon
            \\[-1ex]
            & \theta = \alpha_1 + \alpha_1^\vee + \cdots +
            \alpha_n + \alpha_n^\vee + \rho
        } } {}
        S^\mathrm{sd} (\alpha_1, \dotsc, \alpha_n; \tau_+, \tau_-) \cdot
        [\mathcal{X}^\mathrm{ss}_{\alpha_1} (\tau_+)] \diamond \cdots \diamond
        [\mathcal{X}^\mathrm{ss}_{\alpha_n} (\tau_+)] \diamond
        [\mathcal{X}^\mathrm{sd,ss}_{\rho} (\tau_+)] \ ,
        \raisetag{4ex}
        \\[1ex]
        \label{eq-wcf-epsilon}
        \epsilon_\alpha (\tau_-)
        & =
        \sum_{ \leftsubstack[5em]{
            & n \geq 0; \, \alpha_1, \dotsc, \alpha_n
            \in \pi_0 (\mathcal{X}) \setminus \{ 0 \} \colon
            \\[-1ex]
            & \alpha = \alpha_1 + \cdots + \alpha_n 
        } } {}
        U (\alpha_1, \dotsc, \alpha_n; \tau_+, \tau_-) \cdot
        \epsilon_{\alpha_1} (\tau_+) * \cdots *
        \epsilon_{\alpha_n} (\tau_+) \ ,
        \\*[1ex]
        \label{eq-wcf-epsilon-sd}
        \epsilon^\mathrm{sd}_\theta (\tau_-)
        & =
        \sum_{ \leftsubstack[5em]{
            & n \geq 0; \, \alpha_1, \dotsc, \alpha_n \in \pi_0 (\mathcal{X}) \setminus \{ 0 \}, \,
            \rho \in \pi_0 (\mathcal{X}^\mathrm{sd}) \colon
            \\[-1ex]
            & \theta = \alpha_1 + \alpha_1^\vee + \cdots +
            \alpha_n + \alpha_n^\vee + \rho
        } } {}
        U^\mathrm{sd} (\alpha_1, \dotsc, \alpha_n; \tau_+, \tau_-) \cdot
        \epsilon_{\alpha_1} (\tau_+) \diamond \cdots \diamond
        \epsilon_{\alpha_n} (\tau_+) \diamond
        \epsilon^\mathrm{sd}_{\rho} (\tau_+) \ ,
    \end{align}
    in $\mathbb{M} (\mathcal{X}_\alpha; \mathbb{Q})$
    and\/ $\mathbb{M} (\mathcal{X}^\mathrm{sd}_\theta; \mathbb{Q})$,
    where the sums are finite, and
    \upshape
    \begin{align}
        \label{eq-def-s}
        S (\alpha_1, \dotsc, \alpha_n; \tau_+, \tau_-) & =
        \prod_{i=1}^{n-1} {} \left\{
            \mathrlap{ \begin{array}{ll}
                1, & \tau_+ (\alpha_i) > \tau_+ (\alpha_{i+1}) \text{ and } \\
                & \hspace{2em} \tau_- (\alpha_1 + \cdots + \alpha_i)
                \leq \tau_- (\alpha_{i+1} + \cdots + \alpha_n) \\
                -1, & \tau_+ (\alpha_i) \leq \tau_+ (\alpha_{i+1}) \text{ and } \\
                & \hspace{2em} \tau_- (\alpha_1 + \cdots + \alpha_i)
                > \tau_- (\alpha_{i+1} + \cdots + \alpha_n) \\
                0, & \text{otherwise}
            \end{array} }
            \hspace{20em}
        \right\} \ ,
        \\
        \label{eq-def-ssd}
        S^\mathrm{sd} (\alpha_1, \dotsc, \alpha_n; \tau_+, \tau_-) & =
        \prod_{i=1}^n {} \left\{
            \mathrlap{ \begin{array}{ll}
                1, & \tau_+ (\alpha_i) > \tau_+ (\alpha_{i+1}) \text{ and } 
                \tau_- (\alpha_1 + \cdots + \alpha_i) \leq 0 \\
                -1, & \tau_+ (\alpha_i) \leq \tau_+ (\alpha_{i+1}) \text{ and }
                \tau_- (\alpha_1 + \cdots + \alpha_i) > 0 \\
                0, & \text{otherwise}
            \end{array} }
            \hspace{20em}
        \right\} \ ,
        \\[1ex]
        U (\alpha_1, \dotsc, \alpha_n; \tau_+, \tau_-)
        & =
        \notag \\*[-.5ex]
        & \hspace{-7em}
        \sum_{ \leftsubstack[8em]{
            \\[-2ex]
            & 0 = a_0 < \cdots < a_m = n, \ 
            0 = b_0 < \cdots < b_\ell = m \colon
            \\[-1ex]
            & \text{Writing }
            \beta_i = \alpha_{a_{i-1}+1} + \cdots + \alpha_{a_i}
            \text{ for }
            i = 1, \dotsc, m, \\[-1ex]
            & \text{and }
            \gamma_i = \beta_{b_{i-1}+1} + \cdots + \beta_{b_i}
            \text{ for }
            i = 1, \dotsc, \ell, \\[-1ex]
            & \text{we have } \tau_+ (\alpha_j) = \tau_+ (\beta_i)
            \text{ for all } a_{i-1} < j \leq a_i , \\[-1ex]
            & \text{and } \tau_- (\gamma_i) = \tau_- (\alpha_1 + \cdots + \alpha_n)
            \text{ for all } i = 1, \dotsc, \ell
        } } {}
        \frac{(-1)^{\ell-1}}{\ell} \cdot \biggl(
            \prod_{i=1}^{\ell}
            S (\beta_{b_{i-1}+1}, \dotsc, \beta_{b_i}; \tau_+, \tau_-)
        \biggr) \cdot
        \biggl(
            \prod_{i=1}^m \frac{1}{(a_i - a_{i-1})!} 
        \biggr) \ , 
        \raisetag{4ex}
        \label{eq-def-u}
        \\[1ex]
        U^\mathrm{sd} (\alpha_1, \dotsc, \alpha_n; \tau_+, \tau_-)
        & =
        \notag \\*[-.5ex]
        & \hspace{-7em}
        \sum_{ \leftsubstack[8em]{
            \\[-2ex]
            & 0 = a_0 < \cdots < a_m \leq n, \ 
            0 = b_0 < \cdots < b_\ell \leq m \colon
            \\[-1ex]
            & \text{Writing }
            \beta_i = \alpha_{a_{i-1}+1} + \cdots + \alpha_{a_i}
            \text{ for }
            i = 1, \dotsc, m, \\[-1ex]
            & \text{and }
            \gamma_i = \beta_{b_{i-1}+1} + \cdots + \beta_{b_i}
            \text{ for }
            i = 1, \dotsc, \ell, \\[-1ex]
            & \text{we have } \tau_+ (\alpha_j) = \tau_+ (\beta_i)
            \text{ for all } a_{i-1} < j \leq a_i , \\[-1ex]
            & \tau_+ (\alpha_j) = 0
            \text{ for all } j > a_m , \\[-1ex]
            & \text{and } \tau_- (\gamma_i) = 0
            \text{ for all } i = 1, \dotsc, \ell
        } } {}
        \binom{-1/2}{\ell} \cdot \biggl(
            \prod_{i=1}^{\ell}
            S (\beta_{b_{i-1}+1}, \dotsc, \beta_{b_i}; \tau_+, \tau_-)
        \biggr) \cdot
        S^\mathrm{sd} (\beta_{b_\ell+1}, \dotsc, \beta_{m}; \tau_+, \tau_-)
        \cdot {}
        \notag \\*[-7ex]
        & \hspace{10em}
        \biggl(
            \prod_{i=1}^m \frac{1}{(a_i - a_{i-1})!} 
        \biggr) \cdot
        \frac{1}{2^{n-a_m} \, (n - a_m)!} \ ,
        \label{eq-def-usd}
    \end{align}
    \itshape
    where we set $\tau_+ (\alpha_{n+1}) = 0$
    in \cref{eq-def-ssd}.

    For
    \cref{eq-wcf-delta,eq-wcf-epsilon},
    we do not need~$\mathcal{X}$ or $\tau_+, \tau_-, \tau_0$ to be self-dual.
\end{theorem}

The formulae \cref{eq-wcf-delta,eq-wcf-epsilon}
were originally due to \textcite[Theorem~5.2]{joyce-2008-configurations-iv},
under a slightly different setting.
The self-dual versions \cref{eq-wcf-delta-sd,eq-wcf-epsilon-sd}
are new.

The coefficients \crefrange{eq-def-s}{eq-def-usd}
are combinatorial,
and are defined whenever~$\tau_\pm$ are maps from the set
$C = \{ \alpha_i + \cdots + \alpha_j \mid 1 \leq i \leq j \leq n \}$
of symbolic sums to totally ordered sets~$T_\pm$,
such that $\tau_+ (\gamma_1) \leq \tau_+ (\gamma_2)$
implies $\tau_+ (\gamma_1) \leq \tau_+ (\gamma_1 + \gamma_2) \leq \tau_+ (\gamma_2)$
whenever $\gamma_1, \gamma_2, \gamma_1 + \gamma_2 \in C$,
and similarly for~$\tau_-$.
For \cref{eq-def-ssd,eq-def-usd},
we also require distinguished elements $0 \in T_\pm$.

\begin{proof}
    \allowdisplaybreaks
    The $\Theta$-stratifications of
    $\mathcal{X}^\mathrm{ss}_\alpha (\tau_0)$
    and
    $\mathcal{X}^{\smash{\mathrm{sd,ss}}}_\theta (\tau_0)$
    defined by~$\tau_+$ and~$\tau_-$
    give the relations
    \begin{align}
        \label{eq-wcf-dom-delta}
        [\mathcal{X}^\mathrm{ss}_\alpha (\tau_0)]
        & =
        \sum_{ \leftsubstack[5em]{
            & n > 0; \, \alpha_1, \dotsc, \alpha_n \in
            \pi_0 (\mathcal{X}) \setminus \{ 0 \} \colon \\[-1ex]
            & \alpha = \alpha_1 + \cdots + \alpha_n, \\[-1ex]
            & \tau_0 (\alpha_1) = \cdots = \tau_0 (\alpha_n), \\[-1ex]
            & \tau_\pm (\alpha_1) > \cdots > \tau_\pm (\alpha_n)
        } } {}
        [\mathcal{X}^\mathrm{ss}_{\alpha_1} (\tau_\pm)] * \cdots *
        [\mathcal{X}^\mathrm{ss}_{\alpha_n} (\tau_\pm)] \ ,
        \\[1ex]
        \label{eq-wcf-dom-delta-sd}
        [\mathcal{X}^{\smash{\mathrm{sd,ss}}}_\theta (\tau_0)]
        & =
        \sum_{ \leftsubstack[5em]{
            & n \geq 0; \, \alpha_1, \dotsc, \alpha_n \in
            \pi_0 (\mathcal{X}) \setminus \{ 0 \}, \,
            \rho \in \pi_0 (\mathcal{X}^\mathrm{sd}) \colon
            \\[-1ex]
            & \theta = \alpha_1 + \alpha_1^\vee + \cdots +
            \alpha_n + \alpha_n^\vee + \rho, \\[-1ex]
            & \tau_0 (\alpha_1) = \cdots = \tau_0 (\alpha_n) = 0, \\[-1ex]
            & \tau_\pm (\alpha_1) > \cdots > \tau_\pm (\alpha_n) > 0
        } } {}
        [\mathcal{X}^\mathrm{ss}_{\alpha_1} (\tau_\pm)] \diamond \cdots \diamond
        [\mathcal{X}^\mathrm{ss}_{\alpha_n} (\tau_\pm)] \diamond
        [\mathcal{X}^\mathrm{sd,ss}_{\rho} (\tau_\pm)] \ ,
    \end{align}
    where the `$\pm$' signs mean that we have
    a relation for~$\tau_+$, and another for~$\tau_-$.
    These are finite sums by \cref{lemma-finite-decomposition},
    and agree with \crefrange{eq-wcf-delta}{eq-wcf-delta-sd}
    with $\tau_{\pm}, \tau_0$ in place of~$\tau_+, \tau_-$.

    These relations then imply the relations
    \begin{align}
        \label{eq-wcf-anti-dom-delta}
        [\mathcal{X}^\mathrm{ss}_\alpha (\tau_\pm)]
        & =
        \sum_{ \leftsubstack[5em]{
            & n > 0; \, \alpha_1, \dotsc, \alpha_n \in
            \pi_0 (\mathcal{X}) \setminus \{ 0 \} \colon \\[-1ex]
            & \alpha = \alpha_1 + \cdots + \alpha_n, \\[-1ex]
            & \tau_0 (\alpha_1) = \cdots = \tau_0 (\alpha_n), \\[-1ex]
            & \tau_\pm (\alpha_1 + \cdots + \alpha_i) >
            \tau_\pm (\alpha_{i+1} + \cdots + \alpha_n)
            \text{ for } i = 1, \dotsc, n-1
        } } {}
        (-1)^{n-1} \cdot
        [\mathcal{X}^\mathrm{ss}_{\alpha_1} (\tau_0)] * \cdots *
        [\mathcal{X}^\mathrm{ss}_{\alpha_n} (\tau_0)] \ ,
        \\[1ex]
        \label{eq-wcf-anti-dom-delta-sd}
        [\mathcal{X}^{\smash{\mathrm{sd,ss}}}_\theta (\tau_\pm)]
        & =
        \sum_{ \leftsubstack[5em]{
            & n \geq 0; \, \alpha_1, \dotsc, \alpha_n \in
            \pi_0 (\mathcal{X}) \setminus \{ 0 \}, \,
            \rho \in \pi_0 (\mathcal{X}^\mathrm{sd}) \colon
            \\[-1ex]
            & \theta = \alpha_1 + \alpha_1^\vee + \cdots +
            \alpha_n + \alpha_n^\vee + \rho, \\[-1ex]
            & \tau_0 (\alpha_1) = \cdots = \tau_0 (\alpha_n) = 0, \\[-1ex]
            & \tau_\pm (\alpha_1 + \cdots + \alpha_i) > 0
            \text{ for } i = 1, \dotsc, n
        } } {}
        (-1)^n \cdot
        [\mathcal{X}^\mathrm{ss}_{\alpha_1} (\tau_0)] \diamond \cdots \diamond
        [\mathcal{X}^\mathrm{ss}_{\alpha_n} (\tau_0)] \diamond
        [\mathcal{X}^\mathrm{sd,ss}_{\rho} (\tau_0)] \ ,
    \end{align}
    which agree with
    \crefrange{eq-wcf-delta}{eq-wcf-delta-sd}
    with $\tau_0, \tau_{\pm}$ in place of~$\tau_+, \tau_-$.
    Indeed, these can be verified by expanding the right-hand sides of
    \crefrange{eq-wcf-anti-dom-delta}{eq-wcf-anti-dom-delta-sd}
    using \crefrange{eq-wcf-dom-delta}{eq-wcf-dom-delta-sd},
    then applying \cref{lemma-s-comp} below to see that
    the results are equal to the left-hand sides.

    Now, expanding the right-hand sides of
    \crefrange{eq-wcf-anti-dom-delta}{eq-wcf-anti-dom-delta-sd}
    for~$\tau_+$
    using \crefrange{eq-wcf-dom-delta}{eq-wcf-dom-delta-sd}
    for~$\tau_-$,
    then applying \cref{lemma-s-comp} below,
    gives the general case of
    \crefrange{eq-wcf-delta}{eq-wcf-delta-sd}.

    To verify the relations
    \crefrange{eq-wcf-epsilon}{eq-wcf-epsilon-sd},
    we first substitute the relations
    \crefrange{eq-wcf-delta}{eq-wcf-delta-sd},
    in
    \cref{eq-def-epsilon-linear}, \cref{eq-def-epsilon-sd}
    for~$\tau_-$,
    then substitute in
    \cref{eq-def-epsilon-linear-inv}, \cref{eq-def-epsilon-sd-inv}
    for~$\tau_+$.
    Keeping track of the coefficients
    gives the desired relations.
\end{proof}

\begin{lemma}
    \label{lemma-s-comp}
    \allowdisplaybreaks
    For symbols $\alpha_1, \dotsc, \alpha_n$
    and maps $\tau_1, \tau_2, \tau_3$
    from $\{ \alpha_i + \cdots + \alpha_j \mid 1 \leq i \leq j \leq n \}$
    to totally ordered sets
    with distinguished elements $0$,
    we have the identities
    \begin{align}
        \label{eq-comb-s-comp}
        S (\alpha_1, \dotsc, \alpha_n; \tau_1, \tau_3) & =
        \hspace{2em}
        \mathclap{\sum_{(\beta_1, \dotsc, \beta_m) \in Q}}
        \hspace{2em}
        S(\beta_1, \dotsc, \beta_m; \tau_2, \tau_3) \cdot
        \prod_{i=1}^m S (\alpha_{a_{i-1}+1} \, , \dotsc, \alpha_{a_i}; \tau_1, \tau_2) \ ,
        \\
        S^\mathrm{sd} (\alpha_1, \dotsc, \alpha_n; \tau_1, \tau_3) & =
        \hspace{2em}
        \mathclap{\sum_{(\beta_1, \dotsc, \beta_m) \in \smash{Q^\mathrm{sd}}}}
        \hspace{2em}
        S^\mathrm{sd}(\beta_1, \dotsc, \beta_m; \tau_2, \tau_3) \cdot {}
        \notag \\*
        \label{eq-comb-ssd-comp}
        & \hspace{2em}
        \smash[b]{\biggl(
                \prod_{i=1}^m S (\alpha_{a_{i-1}+1} \, , \dotsc, \alpha_{a_i}; \tau_1, \tau_2)
        \biggr) \cdot S^\mathrm{sd} (\alpha_{a_m+1} \, , \dotsc, \alpha_n; \tau_1, \tau_2)} \ ,
    \end{align}
    where
    \upshape
    \vspace{-12pt}
    \begin{align*}
        Q & = \biggl\{ (\beta_1, \dotsc, \beta_m) \biggm| \begin{array}{l}
            m \geq 1, \ 0 = a_0 < \cdots < a_m = n , \\
            \beta_i = \alpha_{a_{i-1}+1} + \cdots + \alpha_{a_i} \text{ for all } i
        \end{array} \biggr\} \ ,
        \\
        Q^\mathrm{sd} & = \biggl\{ (\beta_1, \dotsc, \beta_m) \biggm| \begin{array}{l}
            m \geq 0, \ 0 = a_0 < \cdots < a_m \leq n , \\
            \beta_i = \alpha_{a_{i-1}+1} + \cdots + \alpha_{a_i} \text{ for all } i
        \end{array} \biggr\} \ .
    \end{align*}
\end{lemma}

\begin{proof}
    The identity \cref{eq-comb-s-comp}
    was proved in \textcite[Theorem~4.5]{joyce-2008-configurations-iv}.
    The identity \cref{eq-comb-ssd-comp}
    follows from \cref{eq-comb-s-comp}
    and the fact that
    $S^\mathrm{sd} (\alpha_1, \dotsc, \alpha_n; \tau_i, \tau_j) =
    S (\alpha_1, \dotsc, \alpha_n, \infty; \tau_i, \tau_j)$,
    where we set $\tau_i (\alpha_j + \cdots + \alpha_n + \infty) = 0$
    for all~$i$ and all $1 \leq j \leq n + 1$.
\end{proof}

\begin{para}[Weakening the assumptions]
    In \cref{thm-wcf-epsilon},
    we can slightly weaken the assumptions
    by allowing~$\tau_0$ to be non-permissible,
    so that $\mathcal{X}^\mathrm{ss}_\alpha (\tau_0)$
    can be non-quasi-compact,
    and we add the extra assumption that the sums
    \crefrange{eq-wcf-anti-dom-delta}{eq-wcf-anti-dom-delta-sd}
    are locally finite for all classes~$\alpha, \theta$.
    In this case, the relations
    \crefrange{eq-wcf-dom-delta}{eq-wcf-dom-delta-sd}
    are always valid as locally finite sums,
    and the proof shows that the relations
    \crefrange{eq-wcf-delta}{eq-wcf-epsilon-sd}
    still hold as locally finite sums.
\end{para}

\subsection{An anti-symmetrized version}
\label{subsec-wcf-anti-sym}

\begin{para}
    In this section, we rewrite the relations
    \crefrange{eq-wcf-epsilon}{eq-wcf-epsilon-sd}
    in terms of anti-symmetrized product operations,
    instead of the operations~$*$ and~$\diamond$.
    This will be useful in writing down wall-crossing formulae
    for DT invariants in \cref{subsec-wcf-dt} below.

    As in \cref{para-wcf-intro},
    let~$\mathcal{X}$ be a self-dual linear stack
    with quasi-compact filtrations.
\end{para}

\begin{para}[Lie algebras and twisted modules]
    \label{para-twisted-modules}
    The motivic Hall algebra $\mathbb{M} (\mathcal{X})$
    can be seen as a Lie algebra
    using the commutator
    \begin{equation}
        [a, b] = a * b - b * a \ .
    \end{equation}
    This was considered in \textcite[\S5.2]{joyce-2007-configurations-ii}.
    It is equipped with a contravariant involution~$(-)^\vee$,
    meaning that $[a^\vee, b^\vee] = [b, a]^\vee$
    for $a, b \in \mathbb{M} (\mathcal{X})$,
    which follows from \cref{thm-hall-assoc}.

    We define a similar anti-symmetrized operation
    ${\heart} \colon \mathbb{M} (\mathcal{X}) \otimes \mathbb{M} (\mathcal{X}^\mathrm{sd}) \to \mathbb{M} (\mathcal{X}^\mathrm{sd})$
    by
    \begin{equation}
        a \heart m = a \diamond m - a^\vee \diamond m \ ,
    \end{equation}
    where~$\diamond$ is the multiplication in the motivic Hall module.
    This does not define a Lie algebra module,
    but a \emph{twisted module},
    in that it satisfies the relations
    \begin{align}
        \label{eq-twisted-anti-symmetry}
        a \heart m
        & = -a^\vee \heart m \ ,
        \\
        \label{eq-twisted-jacobi}
        a \heart (b \heart m) - b \heart (a \heart m)
        & = [a, b] \heart m - [a^\vee, b] \heart m \ .
    \end{align}
    We see \cref{eq-twisted-jacobi}
    as a Jacobi identity twisted by the
    contravariant involution of the Lie algebra,
    giving the extra term $[a^\vee, b] \heart m$.

    Note that over~$\mathbb{Q}$,
    a twisted module in this sense is equivalent to
    a usual module for the Lie subalgebra consisting of elements~$a$
    with $a^\vee = -a$,
    with the action $a \cdot m = (1/2) (a \heart m)$.
\end{para}

\begin{theorem}
    \label{thm-wcf-epsilon-antisym}
    \allowdisplaybreaks
    The relations
    \crefrange{eq-wcf-epsilon}{eq-wcf-epsilon-sd}
    can be written only using the Lie bracket~$[-, -]$
    and the operation~$\heart$,
    without using the products~$*$ or~$\diamond$.

    More precisely, using the notations of
    \cref{thm-wcf-epsilon}, we have the relations
    \begin{align}
        \label{eq-wcf-epsilon-lie}
        \epsilon_\alpha (\tau_-)
        & =
        \sum_{ \leftsubstack[5em]{
            & n \geq 0; \, \alpha_1, \dotsc, \alpha_n
            \in \pi_0 (\mathcal{X}) \setminus \{ 0 \} \colon
            \\[-1ex]
            & \alpha = \alpha_1 + \cdots + \alpha_n 
        } } {}
        \tilde{U} (\alpha_1, \dotsc, \alpha_n; \tau_+, \tau_-) \cdot
        \bigl[ \dotsc \bigl[ \bigl[ \epsilon_{\alpha_1} (\tau_+),
        \epsilon_{\alpha_2} (\tau_+) \bigr], \dotsc \bigr] ,
        \epsilon_{\alpha_n} (\tau_+) \bigr] \ ,
        \\[1ex]
        \label{eq-wcf-epsilon-lie-sd}
        \epsilon^\mathrm{sd}_\theta (\tau_-)
        & =
        \sum_{ \leftsubstack[5em]{
            & n \geq 0; \, m_1, \dotsc, m_n > 0;
            \\[-1ex]
            & \alpha_{1,1}, \dotsc, \alpha_{1,m_1}; \dotsc;
            \alpha_{n,1}, \dotsc, \alpha_{n,m_n} \in \uppi_0 (\mathcal{X}) \setminus \{ 0 \}; \,
            \rho \in \uppi_0 (\mathcal{X}^\mathrm{sd}) \colon
            \\[-.5ex]
            & \theta =
            (\alpha_{1,1} + \alpha_{1,1}^\vee + \cdots +
            \alpha_{1,m_1} + \alpha_{1,m_1}^\vee) + \cdots +
            (\alpha_{n,1} + \alpha_{n,1}^\vee + \cdots +
            \alpha_{n,m_n} + \alpha_{n,m_n}^\vee) + \rho
        } } {}
        \tilde{U}^\mathrm{sd} (\alpha_{1,1}, \dotsc, \alpha_{1,m_1}; \dotsc;
        \alpha_{n,1}, \dotsc, \alpha_{n,m_n}; \tau_+, \tau_-) \cdot {}
        \notag \\
        & \hspace{2em}
        \bigl[ \bigl[ \epsilon_{\alpha_{1,1}} (\tau_+), \dotsc \bigr] ,
        \epsilon_{\alpha_{1,m_1}} (\tau_+) \bigr] \heart \cdots \heart
        \bigl[ \bigl[ \epsilon_{\alpha_{n,1}} (\tau_+), \dotsc \bigr] ,
        \epsilon_{\alpha_{n,m_n}} (\tau_+) \bigr] \heart 
        \epsilon^\mathrm{sd}_{\rho} (\tau_+) \ ,
    \end{align}
    where $\tilde{U} (\ldots)$ and $\tilde{U}^\mathrm{sd} (\ldots)$
    are certain combinatorial coefficients,
    whose choices are not unique.
\end{theorem}

Here, the formulae~\crefrange{eq-wcf-epsilon-lie}{eq-wcf-epsilon-lie-sd}
are just~\crefrange{eq-wcf-epsilon}{eq-wcf-epsilon-sd}
with the terms grouped differently,
and this theorem is essentially a combinatorial property
of the coefficients~$U (\ldots)$ and~$U^\mathrm{sd} (\ldots)$
stating that such regrouping is always possible.
The non-uniqueness of the coefficients
is due to relations in the Lie brackets
and the twisted module operation,
such as the Jacobi identity and \crefrange{eq-twisted-anti-symmetry}{eq-twisted-jacobi}.

\begin{proof}
    The relation \cref{eq-wcf-epsilon-lie}
    was shown in \textcite[Theorem~5.4]{joyce-2008-configurations-iv}.
    The relation \cref{eq-wcf-epsilon-lie-sd}
    will follow from a more general result to appear in
    \cite{epsilon-iii};
    an earlier version of this paper,
    \cite[Appendix~D]{bu-self-dual-i}, contains a direct but rather complicated proof
    of this combinatorial property.
\end{proof}

\subsection{Wall-crossing for DT invariants}
\label{subsec-wcf-dt}

\begin{para}
    In this section, we prove wall-crossing formulae
    for our self-dual DT invariants defined in \cref{subsec-dt,subsec-motivic-dt},
    using the wall-crossing formulae for epsilon motives
    established in \cref{thm-wcf-epsilon,thm-wcf-epsilon-antisym}.
    A key ingredient is the \emph{motivic integral identity}
    for Behrend functions proved by
    \textcite{bu-integral},
    generalizing the integral identities in the linear case of
    \textcite[Conjecture~4]{kontsevich-soibelman-motivic-dt},
    proved by \textcite{le-2015-integral},
    and \textcite[Theorem~5.11]{joyce-song-2012}.

    Throughout, let~$K$ be an algebraically closed field of characteristic~$0$,
    and let~$\mathcal{X}$ be a self-dual $(-1)$-shifted symplectic linear stack over~$K$,
    as in \cref{para-derived-linear-stacks}.
    We further assume that the classical truncation~$\mathcal{X}_\mathrm{cl}$
    is Nisnevich locally fundamental,
    as in \cref{para-local-structure}.
\end{para}

\begin{theorem}
    \label{thm-wcf-dt}
    \allowdisplaybreaks
    Let $\tau_+, \tau_-, \tau_0$ be
    permissible self-dual stability conditions on~$\mathcal{X}$,
    with $\tau_0$ dominating both~$\tau_+$ and~$\tau_-$.
    Then for any $\alpha \in \pi_0 (\mathcal{X})$
    and $\theta \in \pi_0 (\mathcal{X}^\mathrm{sd})$,
    we have the wall-crossing formulae
    \begin{align}
        \label{eq-wcf-dt}
        \mathrm{DT}_\alpha (\tau_-)
        & =
        \sum_{ \leftsubstack[5em]{
            & n \geq 0; \, \alpha_1, \dotsc, \alpha_n
            \in \pi_0 (\mathcal{X}) \setminus \{ 0 \} \colon
            \\[-1ex]
            & \alpha = \alpha_1 + \cdots + \alpha_n 
        } } {}
        \tilde{U} (\alpha_1, \dotsc, \alpha_n; \tau_+, \tau_-) \cdot
        \ell (\alpha_1, \dotsc, \alpha_n) \cdot
        \mathrm{DT}_{\alpha_1} (\tau_+) \cdots
        \mathrm{DT}_{\alpha_n} (\tau_+) \ ,
        \raisetag{6ex}
        \\[1ex]
        \label{eq-wcf-dt-sd}
        \mathrm{DT}^\mathrm{sd}_\theta (\tau_-)
        & =
        \sum_{ \leftsubstack[5em]{
            & n \geq 0; \, m_1, \dotsc, m_n > 0;
            \\[-1ex]
            & \alpha_{1,1}, \dotsc, \alpha_{1,m_1}; \dotsc;
            \alpha_{n,1}, \dotsc, \alpha_{n,m_n} \in \uppi_0 (\mathcal{X}) \setminus \{ 0 \}; \,
            \rho \in \uppi_0 (\mathcal{X}^\mathrm{sd}) \colon
            \\[-.5ex]
            & \theta =
            (\alpha_{1,1} + \alpha_{1,1}^\vee + \cdots +
            \alpha_{1,m_1} + \alpha_{1,m_1}^\vee) + \cdots +
            (\alpha_{n,1} + \alpha_{n,1}^\vee + \cdots +
            \alpha_{n,m_n} + \alpha_{n,m_n}^\vee) + \rho
        } } {}
        \tilde{U}^\mathrm{sd} (\alpha_{1,1}, \dotsc, \alpha_{1,m_1}; \dotsc;
        \alpha_{n,1}, \dotsc, \alpha_{n,m_n}; \tau_+, \tau_-) \cdot {}
        \\
        & \hspace{2em}
        \ell^\mathrm{sd} (\alpha_{1,1}, \dotsc, \alpha_{1,m_1}; \dotsc;
        \alpha_{n,1}, \dotsc, \alpha_{n,m_n}; \rho) \cdot {}
        \notag \\
        & \hspace{2em}
        \bigl( \mathrm{DT}_{\alpha_{1,1}} (\tau_+) \cdots
        \mathrm{DT}_{\alpha_{1,m_1}} (\tau_+) \bigr) \cdots
        \bigl( \mathrm{DT}_{\alpha_{n,1}} (\tau_+) \cdots
        \mathrm{DT}_{\alpha_{n,m_n}} (\tau_+) \bigr) \cdots
        \mathrm{DT}^\mathrm{sd}_{\rho} (\tau_+) \ ,
        \notag
    \end{align}
    where the sums contain finitely many non-zero terms,
    the coefficients
    $\tilde{U} (\ldots), \tilde{U}^\mathrm{sd} (\ldots) \in \mathbb{Q}$
    are defined in \cref{thm-wcf-epsilon-antisym},
    and the coefficients
    $\ell (\ldots), \ell^\mathrm{sd} (\ldots) \in \mathbb{Z}$
    are defined in \cref{para-coef-l} below.

    If, moreover, $\mathcal{X}$ is equipped with
    an orientation data~$o_\mathcal{X}$
    or a self-dual orientation data
    $(o_\mathcal{X}, o_{\smash{\mathcal{X}^\mathrm{sd}}})$,
    then we have the wall-crossing formulae
    \begin{align}
        \label{eq-wcf-dt-mot}
        \mathrm{DT}^\mathrm{mot}_\alpha (\tau_-)
        & =
        \sum_{ \leftsubstack[5em]{
            & n \geq 0; \, \alpha_1, \dotsc, \alpha_n
            \in \pi_0 (\mathcal{X}) \setminus \{ 0 \} \colon
            \\[-1ex]
            & \alpha = \alpha_1 + \cdots + \alpha_n 
        } } {}
        \tilde{U} (\alpha_1, \dotsc, \alpha_n; \tau_+, \tau_-) \cdot
        L (\alpha_1, \dotsc, \alpha_n) \cdot
        \mathrm{DT}^\mathrm{mot}_{\alpha_1} (\tau_+) \cdots
        \mathrm{DT}^\mathrm{mot}_{\alpha_n} (\tau_+) \ ,
        \raisetag{6ex}
        \\[1ex]
        \label{eq-wcf-dt-mot-sd}
        \mathrm{DT}^{\smash{\mathrm{mot,sd}}}_\theta (\tau_-)
        & =
        \sum_{ \leftsubstack[5em]{
            & n \geq 0; \, m_1, \dotsc, m_n > 0;
            \\[-1ex]
            & \alpha_{1,1}, \dotsc, \alpha_{1,m_1}; \dotsc;
            \alpha_{n,1}, \dotsc, \alpha_{n,m_n} \in \uppi_0 (\mathcal{X}) \setminus \{ 0 \}; \,
            \rho \in \uppi_0 (\mathcal{X}^\mathrm{sd}) \colon
            \\[-.5ex]
            & \theta =
            (\alpha_{1,1} + \alpha_{1,1}^\vee + \cdots +
            \alpha_{1,m_1} + \alpha_{1,m_1}^\vee) + \cdots +
            (\alpha_{n,1} + \alpha_{n,1}^\vee + \cdots +
            \alpha_{n,m_n} + \alpha_{n,m_n}^\vee) + \rho
        } } {}
        \tilde{U}^\mathrm{sd} (\alpha_{1,1}, \dotsc, \alpha_{1,m_1}; \dotsc;
        \alpha_{n,1}, \dotsc, \alpha_{n,m_n}; \tau_+, \tau_-) \cdot {}
        \\
        & \hspace{2em}
        L^\mathrm{sd} (\alpha_{1,1}, \dotsc, \alpha_{1,m_1}; \dotsc;
        \alpha_{n,1}, \dotsc, \alpha_{n,m_n}; \rho) \cdot {}
        \notag \\
        & \hspace{2em}
        \bigl( \mathrm{DT}^\mathrm{mot}_{\alpha_{1,1}} (\tau_+) \cdots
        \mathrm{DT}^\mathrm{mot}_{\alpha_{1,m_1}} (\tau_+) \bigr) \cdots
        \bigl( \mathrm{DT}^\mathrm{mot}_{\alpha_{n,1}} (\tau_+) \cdots
        \mathrm{DT}^\mathrm{mot}_{\alpha_{n,m_n}} (\tau_+) \bigr) \cdots
        \mathrm{DT}^{\smash{\mathrm{mot,sd}}}_{\rho} (\tau_+) \ ,
        \notag
    \end{align}
    respectively,
    where the coefficients
    $L (\ldots), L^{\smash{\mathrm{sd}}} (\ldots)
    \in \mathbb{Z} [\mathbb{L}^{\pm 1/2}]$
    are defined in \cref{para-coef-l} below.
\end{theorem}

The proof of the theorem will be given in
\cref{para-pf-wcf-dt}.

\begin{para}[Symmetric stacks]
    \label{para-symmetric-stacks}
    The wall-crossing formulae in \cref{thm-wcf-dt}
    provide a condition for the
    DT invariants to be independent
    of the choice of the stability condition.

    We say that a $(-1)$-shifted symplectic stack~$\mathcal{X}$
    is \emph{numerically symmetric},
    if $\vdim \mathrm{Filt} (\mathcal{X}) = 0$,
    meaning that this holds on every connected component
    of $\mathrm{Filt} (\mathcal{X})$.
    See \textcite[\S4.3]{bu-davison-ibanez-nunez-kinjo-padurariu}
    for examples of stacks satisfying this condition.

    For example,
    if~$\mathcal{X}$ is a self-dual $(-1)$-shifted symplectic linear stack,
    then~$\mathcal{X}$ is numerically symmetric
    if and only if $\vdim \mathcal{X}_{\alpha, \beta}^+ = 0$
    for all~$\alpha, \beta \in \uppi_0 (\mathcal{X})$,
    and~$\mathcal{X}^\mathrm{sd}$ is numerically symmetric
    if and only if $\vdim \mathcal{X}_{\alpha, \theta}^{\smash{\mathrm{sd}, +}} = 0$
    for all~$\alpha \in \uppi_0 (\mathcal{X})$ and~$\theta \in \uppi_0 (\mathcal{X}^\mathrm{sd})$.

    When~$\mathcal{X}$ and~$\mathcal{X}^\mathrm{sd}$
    are numerically symmetric, the coefficients
    $L (\ldots), \ell (\ldots)$ are zero unless $n \leq 1$,
    and the coefficients
    $L^\mathrm{sd} (\ldots), \ell^\mathrm{sd} (\ldots)$ are zero unless $n = 0$,
    which follow from their definitions.
    This immediately implies the following:
\end{para}

\begin{corollary}
    \label{cor-wcf-symmetric}
    In the situation of \cref{thm-wcf-dt},
    assume that~$\mathcal{X}$ and~$\mathcal{X}^\mathrm{sd}$
    are numerically symmetric.
    Then the relations
    \crefrange{eq-wcf-dt}{eq-wcf-dt-mot-sd}
    simplify to
    \begin{alignat}{2}
        \mathrm{DT}_\alpha (\tau_-)
        & = \mathrm{DT}_\alpha (\tau_+) \ ,
        & \qquad
        \mathrm{DT}^{\smash{\mathrm{sd}}}_\theta (\tau_-)
        & = \mathrm{DT}^{\smash{\mathrm{sd}}}_\theta (\tau_+) \ ,
        \\
        \mathrm{DT}^\mathrm{mot}_\alpha (\tau_-)
        & = \mathrm{DT}^\mathrm{mot}_\alpha (\tau_+) \ ,
        & \qquad
        \mathrm{DT}^{\smash{\mathrm{mot,sd}}}_\theta (\tau_-)
        & = \mathrm{DT}^{\smash{\mathrm{mot,sd}}}_\theta (\tau_+) \ .
    \end{alignat}
    In particular, if\/~$\mathcal{X}$ has quasi-compact connected components,
    then all the above invariants are independent
    of the choice of the stability condition.
\end{corollary}

Here, the final claim follows from taking
$\tau_0$ and $\tau_+$
to be the trivial stability condition,
which is permissible when~$\mathcal{X}$ has quasi-compact connected components.

The remaining part of this section is devoted to
the proof of \cref{thm-wcf-dt}.

\begin{para}[Lattice algebras and modules]
    Define
    \begin{equation*}
        \Lambda_\mathcal{X} =
        \bigoplus_{\alpha \in \uppi_0 (\mathcal{X})}
        \hat{\mathbb{M}}^\mathrm{mon} (K; \mathbb{Q}) \cdot \lambda_\alpha \ ,
        \qquad
        \Lambda_\mathcal{X}^\mathrm{sd} =
        \bigoplus_{\theta \in \uppi_0 (\mathcal{X}^\mathrm{sd})}
        \hat{\mathbb{M}}^\mathrm{mon} (K; \mathbb{Q}) \cdot \lambda^\mathrm{sd}_\theta \ ,
    \end{equation*}
    where $\hat{\mathbb{M}}^\mathrm{mon} (K; \mathbb{Q})$
    is the ring of monodromic motives
    defined in \cref{para-monodromic-motives}.
    We define a product~$*$ on~$\Lambda_\mathcal{X}$,
    and a $\Lambda_\mathcal{X}$-module structure~$\diamond$ on~$\Lambda_\mathcal{X}^\mathrm{sd}$,
    by setting
    \begin{equation}
        \label{eq-lambda-mult}
        \lambda_\alpha * \lambda_\beta
        =
        \frac
        {\mathbb{L}^{\vdim \mathcal{X}_{\alpha, \beta}^+ / 2}}
        {\mathbb{L}^{1/2} - \mathbb{L}^{-1/2}}
        \cdot \lambda_{\alpha + \beta} \ ,
        \qquad
        \lambda_\alpha \diamond \lambda^\mathrm{sd}_\theta
        =
        \frac
        {\mathbb{L}^{\vdim \mathcal{X}_{\alpha, \theta}^{\smash{\mathrm{sd}, +}} / 2}}
        {\mathbb{L}^{1/2} - \mathbb{L}^{-1/2}}
        \cdot \lambda_{\alpha + \theta + \smash{\alpha^\vee}}^\mathrm{sd}
    \end{equation}
    for $\alpha, \beta \in \uppi_0 (\mathcal{X})$
    and $\theta \in \uppi_0 (\mathcal{X}^\mathrm{sd})$.
    The associativity of these operations follow from the relations
    \begin{alignat*}{2}
        \vdim \mathcal{X}_{\alpha, \beta}^+ +
        \vdim \mathcal{X}_{\alpha + \beta, \gamma}^+
        & =
        \vdim \mathcal{X}_{\alpha, \beta, \gamma}^+
        && =
        \vdim \mathcal{X}_{\alpha, \beta + \gamma}^+ +
        \vdim \mathcal{X}_{\beta, \gamma}^+ \ ,
        \\
        \vdim \mathcal{X}_{\alpha, \beta}^+ +
        \vdim \mathcal{X}_{\alpha + \beta, \theta}^{\smash{\mathrm{sd}, +}}
        & =
        \vdim \mathcal{X}_{\alpha, \beta, \theta}^{\smash{\mathrm{sd}, +}}
        && =
        \vdim \mathcal{X}_{\alpha, \beta + \theta + \smash{\beta^\vee}}^{\smash{\mathrm{sd}, +}} +
        \vdim \mathcal{X}_{\beta, \theta}^{\smash{\mathrm{sd}, +}} \ ,
    \end{alignat*}
    which follow from the derived versions of the associativity diagrams
    \cref{eq-cd-hall-assoc,eq-cd-hall-mod-assoc}.
    The algebra $\Lambda_\mathcal{X}$ is often called
    the \emph{quantum torus} in the literature,
    such as in \textcite[\S6.2]{kontsevich-soibelman-motivic-dt}.

    The map $\lambda_\alpha \mapsto \lambda_{\smash{\alpha^\vee}}$
    defines a contravariant involution~$(-)^\vee$
    of~$\Lambda_\mathcal{X}$.
    We also write
    $a \heart m = a \diamond m - a^\vee \diamond m$
    for $a \in \Lambda_\mathcal{X}$ and $m \in \Lambda_\mathcal{X}^\mathrm{sd}$,
    as in \cref{para-twisted-modules},
    which gives~$\Lambda_\mathcal{X}^\mathrm{sd}$ the structure of
    a twisted module over the involutive Lie algebra~$\Lambda_\mathcal{X}$,
    with the commutator Lie bracket.

    We also define the numerical versions
    \begin{equation*}
        \bar{\Lambda}_\mathcal{X} =
        \bigoplus_{\alpha \in \uppi_0 (\mathcal{X})}
        \mathbb{Q} \cdot \bar{\lambda}_\alpha \ ,
        \qquad
        \bar{\Lambda}_\mathcal{X}^\mathrm{sd} =
        \bigoplus_{\theta \in \uppi_0 (\mathcal{X}^\mathrm{sd})}
        \mathbb{Q} \cdot \bar{\lambda}^\mathrm{sd}_\theta \ ,
    \end{equation*}
    which are no longer equipped with algebra structures,
    but have a Lie bracket and a twisted module operation~$\heart$,
    respectively, given by
    \begin{align}
        \label{eq-bar-lambda-mult}
        [\bar{\lambda}_\alpha, \bar{\lambda}_\beta]
        & = (-1)^{1 + \vdim \mathcal{X}_{\alpha, \beta}^+}
        \cdot \vdim \mathcal{X}_{\alpha, \beta}^+
        \cdot \bar{\lambda}_{\alpha + \beta} \ ,
        \\
        \label{eq-bar-lambda-sd-mult}
        \bar{\lambda}_\alpha \heart \bar{\lambda}^\mathrm{sd}_\theta
        & = (-1)^{1 + \vdim \mathcal{X}_{\alpha, \theta}^{\smash{\mathrm{sd}, +}}}
        \cdot \vdim \mathcal{X}_{\alpha, \theta}^{\smash{\mathrm{sd}, +}}
        \cdot \bar{\lambda}_{\alpha + \theta + \smash{\alpha^\vee}}^\mathrm{sd} \ .
    \end{align}
    By \cite[Lemma~3.1.7]{bu-integral},
    we have
    $\vdim \mathcal{X}_{\beta, \alpha}^+ =
    -{\vdim \mathcal{X}_{\alpha, \beta}^+}$
    and
    $\vdim \mathcal{X}_{\smash{\alpha^\vee}, \theta}^{\smash{\mathrm{sd}, +}} =
    -{\vdim \mathcal{X}_{\alpha, \theta}^{\smash{\mathrm{sd}, +}}}$,
    establishing
    \cref{eq-bar-lambda-mult,eq-bar-lambda-sd-mult}
    as limits of \cref{eq-lambda-mult}
    as $\mathbb{L}^{1/2} \to -1$.
\end{para}

\begin{para}[Coefficients]
    \label{para-coef-l}
    We can now define the coefficients
    $L (\ldots), L^{\smash{\mathrm{sd}}} (\ldots)$, etc.,
    which appear in~\cref{eq-wcf-dt}.

    For $\alpha_1, \dotsc, \alpha_n \in \uppi_0 (\mathcal{X})$,
    we record the coefficients of the Lie brackets in
    $\Lambda_\mathcal{X}$ and $\bar{\Lambda}_\mathcal{X}$ as
    \begin{align}
        [[ \dotsc [ \lambda_{\alpha_1}, \lambda_{\alpha_2} ], \dotsc ], \lambda_{\alpha_n} ]
        & =
        L (\alpha_1, \dotsc, \alpha_n) \cdot
        \lambda_{\alpha_1 + \cdots + \alpha_n} \ ,
        \\
        [[ \dotsc [ \bar{\lambda}_{\alpha_1}, \bar{\lambda}_{\alpha_2} ], \dotsc ], \bar{\lambda}_{\alpha_n} ]
        & =
        \ell (\alpha_1, \dotsc, \alpha_n) \cdot
        \bar{\lambda}_{\alpha_1 + \cdots + \alpha_n} \ ,
    \end{align}
    where $L (\alpha_1, \dotsc, \alpha_n) \in \mathbb{Z} [\mathbb{L}^{\pm 1/2}]$
    and $\ell (\alpha_1, \dotsc, \alpha_n) \in \mathbb{Z}$.

    Similarly, for $\alpha_{1,1}, \dotsc, \alpha_{1,m_1}; \dotsc;
    \alpha_{n,1}, \dotsc, \alpha_{n,m_n} \in \uppi_0 (\mathcal{X})$
    and $\rho \in \uppi_0 (\mathcal{X}^\mathrm{sd})$,
    we also record the coefficients in
    \begin{align}
        &{}
        [[ \dotsc [ \lambda_{\alpha_{1,1}}, \lambda_{\alpha_{1,2}} ], \dotsc ], \lambda_{\alpha_{1,m_1}} ] \heart
        \cdots \heart
        [[ \dotsc [ \lambda_{\alpha_{n,1}}, \lambda_{\alpha_{n,2}} ], \dotsc ], \lambda_{\alpha_{n,m_n}} ] \heart
        \lambda^\mathrm{sd}_\rho
        \notag \\
        & \hspace{4em}
        = L^\mathrm{sd} (\alpha_{1,1}, \dotsc, \alpha_{1,m_1}; \dotsc;
        \alpha_{n,1}, \dotsc, \alpha_{n,m_n}; \rho) \cdot
        \lambda_{\smash{\alpha_{1,1} + \alpha_{1,1}^\vee + \cdots + \alpha_{n, m_n} + \alpha_{n, m_n}^\vee + \rho}}^\mathrm{sd} \ ,
        \\[1ex]
        &{}
        [[ \dotsc [ \bar{\lambda}_{\alpha_{1,1}}, \bar{\lambda}_{\alpha_{1,2}} ], \dotsc ], \bar{\lambda}_{\alpha_{1,m_1}} ] \heart
        \cdots \heart
        [[ \dotsc [ \bar{\lambda}_{\alpha_{n,1}}, \bar{\lambda}_{\alpha_{n,2}} ], \dotsc ], \bar{\lambda}_{\alpha_{n,m_n}} ] \heart
        \bar{\lambda}^\mathrm{sd}_\rho
        \notag \\
        & \hspace{4em}
        = \ell^\mathrm{sd} (\alpha_{1,1}, \dotsc, \alpha_{1,m_1}; \dotsc;
        \alpha_{n,1}, \dotsc, \alpha_{n,m_n}; \rho) \cdot
        \bar{\lambda}_{\smash{\alpha_{1,1} + \alpha_{1,1}^\vee + \cdots + \alpha_{n, m_n} + \alpha_{n, m_n}^\vee + \rho}}^\mathrm{sd} \ ,
    \end{align}
    where $L^\mathrm{sd} (\ldots) \in \mathbb{Z} [\mathbb{L}^{\pm 1/2}]$
    and $\ell^\mathrm{sd} (\ldots) \in \mathbb{Z}$.

    These coefficients only depend on the numbers
    $\vdim \mathcal{X}_{\alpha, \beta}^+$ and
    $\vdim \mathcal{X}_{\alpha, \theta}^{\smash{\mathrm{sd}, +}}$
    for $\alpha, \beta \in \uppi_0 (\mathcal{X})$
    and $\theta \in \uppi_0 (\mathcal{X}^\mathrm{sd})$.
    They have straightforward explicit expressions,
    which we omit.

    We have the relations
    $\ell (\ldots) = L (\ldots) |_{\mathbb{L}^{1/2} = -1}$
    and $\ell^\mathrm{sd} (\ldots) =
    L^\mathrm{sd} (\ldots) |_{\mathbb{L}^{1/2} = -1}$.
    Also, $L (\ldots)$ and $L^\mathrm{sd} (\ldots)$
    are symmetric Laurent polynomials in~$\mathbb{L}^{1/2}$,
    in that they are invariant under the transformation
    $\mathbb{L}^{1/2} \mapsto \mathbb{L}^{-1/2}$.
\end{para}

\begin{para}[The motivic integral identity]
    \label{para-mot-integral-identity}
    A crucial ingredient in proving
    wall-crossing formulae for DT invariants
    is the \emph{motivic integral identity}
    for the motivic Behrend function,
    first conjectured by
    \textcite[Conjecture~4]{kontsevich-soibelman-motivic-dt}
    in the linear case,
    proved by \textcite{le-2015-integral} in that case,
    and proved by \textcite[Theorem~4.2.2]{bu-integral}
    in general.

    Suppose that we are given a self-dual orientation data
    $(o_\mathcal{X}, o_{\smash{\mathcal{X}^\mathrm{sd}}})$ on~$\mathcal{X}$.
    The motivic integral identity states, in this case, that we have
    \begin{alignat}{2}
        \label{eq-integral-identity-linear}
        \nu_\mathcal{X}^\mathrm{mot} \boxtimes \nu_\mathcal{X}^\mathrm{mot}
        & =
        \mathbb{L}^{-{\vdim \mathcal{X}_{\alpha, \beta}^+} / 2} \cdot
        \mathrm{gr}_! \circ \mathrm{ev}_1^* (\nu_\mathcal{X}^\mathrm{mot})
        && \quad
        \text{ in }
        \hat{\mathbb{M}}^\mathrm{mon} (\mathcal{X}_{\alpha} \times \mathcal{X}_\beta) \ ,
        \\
        \label{eq-integral-identity-sd}
        \nu_\mathcal{X}^\mathrm{mot} \boxtimes \nu_{\subXsd}^\mathrm{mot}
        & =
        \mathbb{L}^{-{\vdim \mathcal{X}_{\alpha, \theta}^{\smash{\mathrm{sd}, +}}} / 2} \cdot
        \mathrm{gr}_! \circ \mathrm{ev}_1^* (\nu_{\subXsd}^\mathrm{mot})
        && \quad
        \text{ in }
        \hat{\mathbb{M}}^\mathrm{mon} (\mathcal{X}_{\alpha} \times \mathcal{X}^\mathrm{sd}_\theta) \ ,
    \end{alignat}
    where
    $\alpha, \beta \in \uppi_0 (\mathcal{X})$
    and
    $\theta \in \uppi_0 (\mathcal{X}^\mathrm{sd})$,
    and the compositions are through
    $\hat{\mathbb{M}}^\mathrm{mon} (\mathcal{X}_{\alpha, \beta}^+)$
    and
    $\hat{\mathbb{M}}^\mathrm{mon} (\mathcal{X}_{\alpha, \theta}^{\smash{\mathrm{sd}, +}})$,
    respectively.
    These identities imply the relations
    \begin{align}
        \label{eq-integral-homo-linear}
        \biggl(
            \int_{\mathcal{X}_\alpha} a \cdot \nu_\mathcal{X}^\mathrm{mot}
        \biggr) \cdot \biggl(
            \int_{\mathcal{X}_\beta} b \cdot \nu_\mathcal{X}^\mathrm{mot}
        \biggr)
        & =
        \mathbb{L}^{-{\vdim \mathcal{X}_{\alpha, \beta}^+} / 2} \cdot
        \int_{\mathcal{X}_{\alpha + \beta}} {}
        (a * b) \cdot \nu_\mathcal{X}^\mathrm{mot} \ ,
        \\
        \label{eq-integral-homo-sd}
        \biggl(
            \int_{\mathcal{X}_\alpha} a \cdot \nu_{\subXsd}^\mathrm{mot}
        \biggr) \cdot \biggl(
            \int_{\mathcal{X}^\mathrm{sd}_\theta} m \cdot \nu_{\subXsd}^\mathrm{mot}
        \biggr)
        & =
        \mathbb{L}^{-{\vdim \mathcal{X}_{\alpha, \theta}^{\smash{\mathrm{sd}, +}}} / 2} \cdot
        \int_{\mathcal{X}_{\alpha + \theta + \smash{\alpha^\vee}}^\mathrm{sd}} {}
        (a \diamond m) \cdot \nu_{\subXsd}^\mathrm{mot} \ ,
    \end{align}
    where $a \in \mathbb{M}_\mathrm{qc} (\mathcal{X}_\alpha; \mathbb{Q})$,
    $b \in \mathbb{M}_\mathrm{qc} (\mathcal{X}_\beta; \mathbb{Q})$,
    and $m \in \mathbb{M}_\mathrm{qc} (\mathcal{X}^\mathrm{sd}_\theta; \mathbb{Q})$.
    These follow from identifying both sides
    of each relation with the integrals
    \begin{equation*}
        \mathbb{L}^{-{\vdim \mathcal{X}_{\alpha, \beta}^+} / 2} \cdot
        \int_{\mathcal{X}_{\alpha, \beta}^+} {}
        \mathrm{gr}^* (a \boxtimes b) \cdot
        \mathrm{ev}_1^* (\nu_\mathcal{X}^\mathrm{mot}) \ ,
        \qquad
        \mathbb{L}^{-{\vdim \mathcal{X}_{\alpha, \theta}^{\smash{\mathrm{sd}, +}}} / 2} \cdot
        \int_{\mathcal{X}_{\alpha, \theta}^{\smash{\mathrm{sd}, +}}} {}
        \mathrm{gr}^* (a \boxtimes m) \cdot
        \mathrm{ev}_1^* (\nu_{\subXsd}^\mathrm{mot}) \ ,
    \end{equation*}
    respectively,
    using the projection formula~\cref{eq-motive-projection}.
    The relation~\cref{eq-integral-homo-linear}
    was first described by
    \textcite[Theorem~8]{kontsevich-soibelman-motivic-dt}.
\end{para}

\begin{para}[The numeric integral identity]
    Using the numeric version of the motivic integral identity,
    proved by \textcite[Theorem~4.3.3]{bu-integral},
    we can also obtain numerical versions of the integral relations
    \crefrange{eq-integral-homo-linear}{eq-integral-homo-sd},
    \begin{align}
        \label{eq-integral-homo-linear-num}
        & \int_{\mathcal{X}_{\alpha + \beta}} {}
        (1 - \mathbb{L}) \cdot [a, b] \cdot \nu_\mathcal{X} \, d \chi
        =
        (-1)^{1 + \vdim \mathcal{X}_{\alpha, \beta}^+} \cdot
        \vdim \mathcal{X}_{\alpha, \beta}^+ \cdot {}
        \notag \\[-1.5ex]
        & \hspace{6em}
        \biggl(
            \int_{\mathcal{X}_\alpha} {}
            (1 - \mathbb{L}) \cdot a \cdot \nu_\mathcal{X} \, d \chi
        \biggr) \cdot \biggl(
            \int_{\mathcal{X}_\beta} {}
            (1 - \mathbb{L}) \cdot b \cdot \nu_\mathcal{X} \, d \chi
        \biggr) \ ,
        \\
        \label{eq-integral-homo-sd-num}
        & \int_{\mathcal{X}_{\alpha + \theta + \smash{\alpha^\vee}}^\mathrm{sd}} {}
        (a \heart m) \cdot \nu_{\subXsd} \, d \chi
        =
        (-1)^{1 + \vdim \mathcal{X}_{\alpha, \theta}^{\smash{\mathrm{sd}, +}}} \cdot
        \vdim \mathcal{X}_{\alpha, \theta}^{\smash{\mathrm{sd}, +}} \cdot {}
        \notag \\[-1.5ex]
        & \hspace{6em}
        \biggl(
            \int_{\mathcal{X}_\alpha} {}
            (1 - \mathbb{L}) \cdot a \cdot \nu_\mathcal{X} \, d \chi
        \biggr) \cdot \biggl(
            \int_{\mathcal{X}^\mathrm{sd}_\theta} m \cdot \nu_{\subXsd} \, d \chi
        \biggr) \ ,
    \end{align}
    provided that the motives~$a, b, m$ are chosen so that
    the integrals on the right-hand sides are finite,
    that is, they lie in $\hat{\mathbb{M}}^{\mathrm{mon, reg}} (K; A)$
    as in \cref{para-monodromic-motives}
    before taking the Euler characteristics.
    These identities do not require orientations
    on~$\mathcal{X}$ or~$\mathcal{X}^\mathrm{sd}$.
    The identity~\cref{eq-integral-homo-linear-num}
    was proved by
    \textcite[Theorem~5.14]{joyce-song-2012}
    in the setting of Calabi--Yau threefolds.

    To prove them,
    we use a similar argument as in
    \cref{para-mot-integral-identity}.
    Namely, for \cref{eq-integral-homo-linear-num},
    we identify the left-hand side with
    \begin{align}
        &
        \chi \biggl(
            (1 - \mathbb{L})^2 \cdot \biggl(
                - \int_{\mathbb{P} (\mathcal{X}_{\alpha, \beta}^+)} {}
                \mathrm{gr}^* (a \boxtimes b) \cdot
                \mathrm{ev}_1^* (\nu_\mathcal{X})
                + \int_{\mathbb{P} (\mathcal{X}_{\beta, \alpha}^+)} {}
                \bar{\mathrm{gr}}^* (a \boxtimes b) \cdot
                \bar{\mathrm{ev}}_1^* (\nu_\mathcal{X})
            \biggr)
        \biggr)
        \notag \\[-1ex]
        & \hspace{2em}
        {} + \int_{\mathcal{X}_\alpha \times \mathcal{X}_\beta} {}
        (1 - \mathbb{L}) \cdot
        (a \boxtimes b) \cdot
        \bigl( \mathbb{L}^{-h^1 (\mathbb{L}_{\mathrm{gr}})}
        - \mathbb{L}^{-h^1 (\mathbb{L}_{\bar{\mathrm{gr}}})} \bigr) \cdot
        {\oplus}^* (\nu_\mathcal{X})
        \, d \chi \ ,
        \label{eq-integral-homo-linear-num-1}
    \end{align}
    where $\mathbb{P} (\mathcal{X}_{\alpha, \beta}^+)
    = (\mathcal{X}_{\alpha, \beta}^+ \setminus
    \mathrm{sf} (\mathcal{X}_\alpha \times \mathcal{X}_\beta))
    / \mathbb{G}_\mathrm{m}$,
    with the~$\mathbb{G}_\mathrm{m}$-action
    given by choosing an identification of
    $\mathcal{X}_{\alpha, \beta}^+$
    with a component of $\mathrm{Filt} (\mathcal{X})$,
    and $\mathbb{P} (\mathcal{X}_{\beta, \alpha}^+)$
    is defined similarly, using the opposite component.
    We denote by $\bar{\mathrm{gr}}$, $\bar{\mathrm{ev}}_1$
    the maps~$\mathrm{gr}$, $\mathrm{ev}_1$
    for~$\mathcal{X}_{\beta, \alpha}^+$,
    and by $\mathbb{L}_{\mathrm{gr}}$
    the relative cotangent complex of~$\mathcal{X}_{\alpha, \beta}^+$
    over $\mathcal{X}_\alpha \times \mathcal{X}_\beta$.
    We regard $h^1 (\mathbb{L}_{\mathrm{gr}}) =
    \dim \mathrm{H}^1 (\mathbb{L}_{\mathrm{gr}})$
    as a constructible function on~$\mathcal{X}_{\alpha, \beta}^+$,
    which can be pulled back to $\mathcal{X}_\alpha \times \mathcal{X}_\beta$.
    The factors
    $\mathbb{L}^{\smash{-h^1 (\mathbb{L}_{\mathrm{gr}})}}$
    and
    $\mathbb{L}^{\smash{-h^1 (\mathbb{L}_{\bar{\mathrm{gr}}})}}$
    are due to the difference of stabilizer groups
    in $\mathcal{X}_{\alpha, \beta}^+$ and $\mathcal{X}_\alpha \times \mathcal{X}_\beta$;
    see \cite[\S 4.3.4]{bu-integral} for details.
    Applying \cite[(4.3.3.2)]{bu-integral},
    \cref{eq-integral-homo-linear-num-1} becomes
    \begin{equation}
        \label{eq-integral-homo-linear-num-2}
        \int_{\mathcal{X}_\alpha \times \mathcal{X}_\beta} {}
        (1 - \mathbb{L})^2 \cdot
        (a \boxtimes b) \cdot
        \bigl(
            h^1 (\mathbb{L}_{\mathrm{gr}})
            - h^0 (\mathbb{L}_{\mathrm{gr}})
            + h^0 (\mathbb{L}_{\bar{\mathrm{gr}}})
            - h^1 (\mathbb{L}_{\bar{\mathrm{gr}}})
        \bigr) \cdot
        {\oplus}^* (\nu_\mathcal{X})
        \, d \chi \ ,
    \end{equation}
    where we also replaced
    $\mathbb{L}^{\smash{-h^1 (\mathbb{L}_{\mathrm{gr}})}}
    - \mathbb{L}^{\smash{-h^1 (\mathbb{L}_{\bar{\mathrm{gr}}})}}$
    by
    $(1 - \mathbb{L}) \cdot
    (h^1 (\mathbb{L}_{\mathrm{gr}}) - h^1 (\mathbb{L}_{\bar{\mathrm{gr}}}))$,
    as they are equal modulo~$(1 - \mathbb{L})^2$,
    so this will not affect the integral.
    By \cite[Lemma~3.1.7]{bu-integral},
    the alternating sum in~\cref{eq-integral-homo-linear-num-2}
    is equal to
    $-{\vdim \mathcal{X}_{\alpha, \beta}^+}$.
    Finally, by \cite[(4.3.3.1)]{bu-integral},
    we have $\oplus^* (\nu_\mathcal{X}) =
    (-1)^{\vdim \mathcal{X}_{\alpha, \beta}^+} \cdot
    (\nu_\mathcal{X} \boxtimes \nu_\mathcal{X})$,
    which identifies~\cref{eq-integral-homo-linear-num-2}
    with the right-hand side of~\cref{eq-integral-homo-linear-num}.

    The identity~\cref{eq-integral-homo-sd-num}
    can be proved analogously.
\end{para}

\begin{para}[Proof of \texorpdfstring{\cref{thm-wcf-dt}}{Theorem \ref*{thm-wcf-dt}}]
    \label{para-pf-wcf-dt}
    Consider the integration maps
    \begin{alignat*}{2}
        (\mathbb{L}^{1/2} - \mathbb{L}^{-1/2}) \cdot
        \int_\mathcal{X} {} (-) \cdot \nu_\mathcal{X}^\mathrm{mot}
        & \colon \
        & \mathbb{M}_\mathrm{qc} (\mathcal{X}; \mathbb{Q})
        & \longrightarrow \Lambda_\mathcal{X} \ ,
        \\
        \int_{\mathcal{X^\mathrm{sd}}} {} (-) \cdot \nu_{\subXsd}^\mathrm{mot}
        & \colon \
        & \mathbb{M}_\mathrm{qc} (\mathcal{X}^\mathrm{sd}; \mathbb{Q})
        & \longrightarrow \Lambda_\mathcal{X}^\mathrm{sd} \ ,
    \end{alignat*}
    where the generators~$\lambda_\alpha$
    and~$\lambda^\mathrm{sd}_\theta$
    record which components the motives are supported on.
    The relations
    \crefrange{eq-integral-homo-linear}{eq-integral-homo-sd}
    imply that these maps are algebra and module homomorphisms.

    Similarly, the relations
    \crefrange{eq-integral-homo-linear-num}{eq-integral-homo-sd-num}
    imply that the integration maps
    \begin{alignat*}{2}
        \int_\mathcal{X} {} (1 - \mathbb{L}) \cdot (-) \cdot \nu_\mathcal{X} \, d \chi
        & \colon \
        & \mathbb{M}_\mathrm{qc}^*
        (\mathcal{X}; \mathbb{Q})
        & \longrightarrow \bar{\Lambda}_\mathcal{X} \ ,
        \\
        \int_{\mathcal{X^\mathrm{sd}}} {} (-) \cdot \nu_{\subXsd}
        & \colon \
        & \mathbb{M}_\mathrm{qc}^*
        (\mathcal{X}^\mathrm{sd}; \mathbb{Q})
        & \longrightarrow \bar{\Lambda}_\mathcal{X}^\mathrm{sd} \ ,
    \end{alignat*}
    are Lie algebra and twisted module homomorphisms,
    where the superscripts~$*$ indicate subspaces of motives
    for which the integrals are finite.
    It follows from
    \crefrange{eq-integral-homo-linear-num}{eq-integral-homo-sd-num}
    that these subspaces are a Lie subalgebra
    and a sub-twisted module for this subalgebra, respectively.

    The theorem is now a direct consequence of
    \cref{thm-wcf-epsilon-antisym},
    by applying the above integration homomorphisms to the relations
    \crefrange{eq-wcf-epsilon-lie}{eq-wcf-epsilon-lie-sd}.
\end{para}

\section{Applications}

\subsection{Self-dual quivers}
\label{subsec-quiver-dt}

\begin{para}
    We apply our theory to study DT invariants counting
    \emph{self-dual representations} of a \emph{self-dual quiver}.
    These are an analogue of orthogonal and symplectic principal bundles on a variety,
    similar to how the usual quiver representations are analogous to
    vector bundles or coherent sheaves on varieties.

    Self-dual quivers were first introduced by
    \textcite{derksen-weyman-2002}
    as a special case of \emph{$G$-quivers}
    for $G = \mathrm{O} (n)$ or~$\mathrm{Sp} (2n)$,
    and studied by
    \textcite{young-2015-self-dual,young-2016-hall-module,young-2020-quiver}
    in the context of DT theory.

    Throughout, we fix an algebraically closed field~$K$
    of characteristic zero.
\end{para}

\begin{para}[Self-dual quivers]
    To fix notations, recall that a \emph{quiver}
    is a quadruple $Q = (Q_0, Q_1, s, t)$,
    where~$Q_0$ and~$Q_1$ are finite sets,
    thought of as the sets of vertices and edges,
    and $s, t \colon Q_1 \to Q_0$ are the source and target maps.

    For a quiver~$Q$, a \emph{self-dual structure} on~$Q$
    consists of the following data:
    \begin{enumerate}
        \item
            A \emph{contravariant involution}
            \begin{equation*}
                (-)^\vee \colon Q \longsimto Q^\mathrm{op} \ ,
            \end{equation*}
            where~$Q^\mathrm{op} = (Q_0, Q_1, t, s)$
            is the opposite quiver of~$Q$,
            such that $(-)^{\vee \vee} = \mathrm{id}$.

        \item
            Choices of signs
            \begin{equation*}
                u \colon Q_0 \longrightarrow \{ \pm 1 \} \ ,
                \qquad
                v \colon Q_1 \longrightarrow \{ \pm 1 \} \ ,
            \end{equation*}
            such that $u (i) = u (i^\vee)$ for all $i \in Q_0$,
            and $v (a) \, v (a^\vee) = u (s (a)) \, u (t (a))$
            for all $a \in Q_1$.
    \end{enumerate}
    In this case, the $K$-linear abelian category $\mathsf{Mod} (K Q)$
    of finite-dimensional representations of~$Q$
    admits a self-dual structure, defined as follows.
    For a representation~$E$ of~$Q$,
    write~$E_i$ for the vector space at $i \in Q_0$
    and $e_a \colon E_{s (a)} \to E_{t (a)}$ the linear map
    for the edge $a \in Q_1$.
    Define the \emph{dual representation}~$E^\vee$ by
    assigning the vector space~$(E_{i^\vee})^\vee$ to the vertex~$i$,
    and the linear map~$v (a) \cdot (e_{a^\vee})^\vee$ to the edge~$a$.
    Then, identify~$E^{\vee \vee}$ with~$E$
    using the sign~$u (i)$ at each vertex~$i$.
    As in \cref{para-sd-cat},
    we have the groupoid $\mathsf{Mod} (K Q)^\mathrm{sd}$
    of \emph{self-dual representations} of~$Q$.
\end{para}

\begin{para}[Moduli stacks]
    Let~$Q$ be a quiver, and let~$\mathcal{X}_Q$
    be the moduli stack of representations of~$Q$ over~$K$.
    Explicitly, we have
    \begin{equation}
        \label{eq-moduli-quiver}
        \mathcal{X}_Q = \coprod_{\alpha \in \mathbb{N}^{Q_0}}
        V_\alpha / G_\alpha \ ,
    \end{equation}
    where
    $V_\alpha = \bigoplus_{a \in Q_1}
    \mathrm{Hom} (K^{\alpha_{s (a)}}, K^{\alpha_{t (a)}})$,
    and $G_\alpha = \prod_{i \in Q_0} \mathrm{GL} (\alpha_i)$.

    If~$Q$ is equipped with a self-dual structure,
    the self-dual structure on~$\mathsf{Mod} (K Q)$
    extends to an involution of~$\mathcal{X}_Q$,
    establishing it as a self-dual linear stack.
    The homotopy fixed locus~$\mathcal{X}_Q^\mathrm{sd}$
    can be seen as the moduli stack of self-dual representations of~$Q$.
    Explicitly, we have
    \begin{equation}
        \label{eq-moduli-quiver-sd}
        \mathcal{X}_Q^\mathrm{sd} \simeq
        \coprod_{\theta \in (\mathbb{N}^{Q_0})^\mathrm{sd}}
        V^\mathrm{sd}_\theta / G^\mathrm{sd}_\theta \ ,
    \end{equation}
    where $(\mathbb{N}^{Q_0})^\mathrm{sd} \subset \mathbb{N}^{Q_0}$
    is the subset of dimension vectors~$\theta$ such that
    $\theta_i = \theta_{i^\vee}$ for all $i \in Q_0$,
    and~$\theta_i$ is even if $i = i^\vee$ and $u (i) = -1$.
    The vector space~$V^\mathrm{sd}_\theta$
    and the group~$G^\mathrm{sd}_\theta$ are given by
    \begin{align}
        \label{eq-quiver-vsd}
        V^\mathrm{sd}_\theta & =
        \prod_{a \in Q_1^{\smash{\circ}} / \mathbb{Z}_2} {}
        \mathrm{Hom} (K^{\theta_{s (a)}}, K^{\theta_{t (a)}}) \times
        \prod_{a \in Q_1^{\smash{+}}} {}
        \mathrm{Sym}^2 (K^{\theta_{t (a)}}) \times
        \prod_{a \in Q_1^{\smash{-}}} {}
        {\wedge}^2 (K^{\theta_{t (a)}}) \ , 
        \\
        \label{eq-quiver-gsd}
        G^\mathrm{sd}_\theta & =
        \prod_{i \in Q_0^{\smash{\circ}} / \mathbb{Z}_2} {}
        \mathrm{GL} (\theta_i) \times
        \prod_{i \in Q_0^{\smash{+}}} {}
        \mathrm{O} (\theta_i) \times
        \prod_{i \in Q_0^{\smash{-}}} {}
        \mathrm{Sp} (\theta_i) \ ,
    \end{align}
    where~$Q_0^\circ$ is the set of vertices~$i$ with $i \neq i^\vee$,
    and~$Q_0^\pm$ the sets of vertices~$i$ with $i = i^\vee$ and $u (i) = \pm 1$.
    Similarly, $Q_1^\circ$ is the set of edges~$a$ with $a \neq a^\vee$,
    and $Q_1^\pm$ the sets of edges~$a$ with $a = a^\vee$ and
    $v (a) \, u (t (a)) = \pm 1$.
\end{para}

\begin{para}[Potentials]
    \label{para-quiver-potentials}
    Recall that a \emph{potential} on a quiver~$Q$
    is an element $W \in K Q / [K Q, K Q]$,
    where~$K Q$ is the path algebra of~$Q$,
    and $[K Q, K Q] \subset K Q$ is the $K$-linear subspace spanned by commutators.
    Such an element can be seen as
    a formal linear combination of cyclic paths in~$Q$,
    and there is a trace function
    $\varphi_W = \mathrm{tr} (W) \colon \mathcal{X}_Q \to \mathbb{A}^1$
    defined by taking traces along cyclic paths in a representation.
    The derived critical locus
    \begin{equation*}
        \mathcal{X}_{Q, W} = \mathrm{Crit} (\varphi_W) \subset \mathcal{X}_Q
    \end{equation*}
    admits a natural $(-1)$-shifted symplectic structure,
    and is a $(-1)$-shifted symplectic linear stack,
    equipped with a canonical orientation data.

    When~$Q$ is equipped with a self-dual structure,
    the potential~$W$ is said to be \emph{self-dual}
    if it is invariant under the involution of~$K Q$
    sending a path to its dual path, multiplied by the product
    of the signs assigned to the edges in the path.
    In this case, the function~$\varphi_W$ is~$\mathbb{Z}_2$-invariant,
    so~$\mathcal{X}_{Q, W}$ is a self-dual linear stack,
    and the fixed locus~$\mathcal{X}_{Q, W}^\mathrm{sd}$
    admits a natural $(-1)$-shifted symplectic derived structure
    and a canonical self-dual orientation data.

    When the potential~$W$ is zero,
    $\mathcal{X}_{Q, 0} \simeq \mathrm{T}^* [-1] \, \mathcal{X}_Q$
    is the $(-1)$-shifted cotangent stack of the smooth stack~$\mathcal{X}_Q$,
    as in \cref{para-dt-smooth-stacks},
    and in particular,
    its classical truncation coincides with~$\mathcal{X}_Q$.
\end{para}

\begin{para}[Stability conditions]
    A \emph{slope function} on a quiver~$Q$ is a map
    $\mu \colon Q_0 \to \mathbb{Q}$.
    Given such a map,
    the \emph{slope} of a dimension vector
    $\alpha \in \mathbb{N}^{Q_0} \setminus \{ 0 \}$
    is the number
    \begin{equation*}
        \tau (\alpha) = \frac{\sum_{i \in Q_0} \alpha_i \, \mu (i)}{\sum_{i \in Q_0} \alpha_i} \ .
    \end{equation*}
    This defines a stability condition
    on the linear stack~$\mathcal{X}_Q$
    in the sense of \cref{para-linear-stack-stability},
    where the $\Theta$-stratification can be constructed from
    \textcite[Theorem~2.6.3]{ibanez-nunez-filtrations}.

    If~$Q$ is equipped with a self-dual structure,
    then a slope function~$\mu$ is said to be \emph{self-dual}
    if $\mu (i^\vee) = -\mu (i)$ for all $i \in Q_0$.
    In this case, the corresponding stability condition
    on~$\mathsf{Mod} (K Q)$ is self-dual,
    and the corresponding stability condition
    on~$\mathcal{X}_Q$ is also self-dual.

    The above also applies to quivers with potentials.
    For a potential~$W$ on a quiver~$Q$,
    any slope function~$\tau$ on~$Q$
    defines a stability condition on~$\mathcal{X}_{Q, W}$,
    where the existence of a $\Theta$-stratification
    follows from \cite[Theorem~2.6.3]{ibanez-nunez-filtrations}.
    For a self-dual potential~$W$ on a self-dual quiver~$Q$,
    a self-dual slope function~$\tau$ on~$Q$
    defines a self-dual stability condition on~$\mathcal{X}_{Q, W}$.
\end{para}

\begin{para}[DT invariants]
    \label{para-quiver-dt}
    For a quiver~$Q$, a potential~$W$,
    a slope function~$\tau$ on~$Q$,
    and a dimension vector~$\alpha \in \mathbb{N}^{Q_0} \setminus \{ 0 \}$,
    we have the DT invariants
    \begin{equation*}
        \mathrm{DT}_\alpha (\tau) \in \mathbb{Q} \ , \qquad
        \mathrm{DT}_\alpha^\mathrm{mot} (\tau)
        \in \hat{\mathbb{M}}^\mathrm{mon} (K; \mathbb{Q}) \ ,
    \end{equation*}
    defined as in \cref{para-dt-linear,para-motivic-dt},
    for the $(-1)$-shifted symplectic linear stack~$\mathcal{X}_{Q, W}$.
    These invariants were studied by
    \textcite{joyce-song-2012},
    \textcite{kontsevich-soibelman-motivic-dt},
    and others.

    When~$Q$ is equipped with a self-dual structure and $W, \tau$ are self-dual,
    we have the self-dual DT invariants
    \begin{equation*}
        \mathrm{DT}_\theta (\tau) \in \mathbb{Q} \ , \qquad
        \mathrm{DT}_\theta^\mathrm{mot} (\tau)
        \in \hat{\mathbb{M}}^\mathrm{mon} (K; \mathbb{Q}) \ ,
    \end{equation*}
    defined as in \cref{para-dt-sd,para-motivic-dt}
    for the self-dual $(-1)$-shifted symplectic linear stack~$\mathcal{X}_{Q, W}$.
    These are new constructions in this paper.

    When the potential~$W$ is zero,
    we have $\mathcal{X}_{Q, 0} \simeq \mathrm{T}^* [-1] \, \mathcal{X}_Q$
    as in \cref{para-quiver-potentials},
    and the discussions in
    \cref{para-dt-smooth-stacks,para-mot-dt-smooth-stacks}
    apply, which provide more straightforward formulae
    for the DT invariants.
\end{para}

\begin{para}[Wall-crossing formulae]
    For a self-dual quiver~$Q$
    with a self-dual potential~$W$,
    \cref{thm-wcf-dt}
    applies to the self-dual $(-1)$-shifted symplectic linear stack~$\mathcal{X}_{Q, W}$,
    proving wall-crossing formulae for the DT~invariants
    defined in \cref{para-quiver-dt}.
    We are allowed to take~$\tau_+, \tau_-$ in the theorem to be
    any two self-dual slope functions,
    since we can take~$\tau_0$ in the theorem
    to be the trivial stability condition, which is permissible.
\end{para}

\begin{para}[An algorithm for computing DT invariants]
    \label{para-quiver-algorithm}
    \allowdisplaybreaks
    For a self-dual quiver~$Q$,
    in the case when the potential~$W$ is zero,
    we describe an algorithm for computing
    all the invariants
    $\mathrm{DT}_\alpha (\tau)$,
    $\mathrm{DT}_\alpha^\mathrm{mot} (\tau)$,
    $\mathrm{DT}_\theta (\tau)$,
    and $\mathrm{DT}_\theta^\mathrm{mot} (\tau)$,
    for any self-dual slope function~$\tau$.

    First, we compute the motives of
    $\mathcal{X}_\alpha = V_\alpha / G_\alpha$
    and $\mathcal{X}^\mathrm{sd}_\theta = V^\mathrm{sd}_\theta / G^\mathrm{sd}_\theta$,
    as in \crefrange{eq-moduli-quiver}{eq-moduli-quiver-sd},
    in $\mathbb{M} (K)$.
    We use the relation \cref{eq-motive-vb}
    for the vector bundles $\mathcal{X}_\alpha \to {*} / G_\alpha$
    and $\mathcal{X}^\mathrm{sd}_\theta \to {*} / G^\mathrm{sd}_\theta$,
    and the motives
    \begin{align}
        \label{eq-motive-bgl}
        [* / \mathrm{GL} (n)]
        & = \prod_{i=0}^{n-1}
        \frac{1}{\mathbb{L}^n - \mathbb{L}^i} \ ,
        \\
        \label{eq-motive-bo-even}
        [* / \mathrm{O} (2n)]
        & = \mathbb{L}^n \cdot \prod_{i=0}^{n-1}
        \frac{1}{\mathbb{L}^{2n} - \mathbb{L}^{2i}} \ ,
        \\
        \label{eq-motive-bo-odd}
        \mathllap{[* / \mathrm{O} (2n+1)] = {}}
        [* / \mathrm{Sp} (2n)]
        & = \mathbb{L}^{-n} \cdot \prod_{i=0}^{n-1}
        \frac{1}{\mathbb{L}^{2n} - \mathbb{L}^{2i}} \ ,
    \end{align}
    where the linear and symplectic cases follow from
    \textcite[Theorem~4.10]{joyce-2007-stack-functions},
    as these are \emph{special groups} in the sense there,
    while the orthogonal cases are due to
    \textcite[Theorem~3.7]{dhillon-young-2016-motive}.
    We then have
    \begin{alignat}{2}
        \label{eq-quiver-motive}
        \int_{\mathcal{X}_\alpha} \nu_\mathcal{X}^\mathrm{mot}
        & = \mathbb{L}^{-(\dim V_\alpha - \dim G_\alpha) / 2} \cdot
        [\mathcal{X}_\alpha]
        && =
        \mathbb{L}^{(\dim V_\alpha + \dim G_\alpha) / 2} \cdot
        [* / G_\alpha] \ ,
        \\
        \label{eq-quiver-motive-sd}
        \int_{\mathcal{X}^\mathrm{sd}_\theta} \nu_{\subXsd}^\mathrm{mot}
        & = \mathbb{L}^{-(\dim V^\mathrm{sd}_\theta - \dim G^\mathrm{sd}_\theta) / 2} \cdot
        [\mathcal{X}^\mathrm{sd}_\theta]
        && =
        \mathbb{L}^{(\dim V^\mathrm{sd}_\theta + \dim G^\mathrm{sd}_\theta) / 2} \cdot
        [* / G^\mathrm{sd}_\theta] \ ,
    \end{alignat}
    where $[* / G_\alpha]$ and $[* / G^\mathrm{sd}_\theta]$
    are products of the rational functions in
    \cref{eq-motive-bgl,eq-motive-bo-even,eq-motive-bo-odd}.

    Next, we compute the invariants
    $\mathrm{DT}_\alpha^\mathrm{mot} (0)$
    and $\mathrm{DT}_\theta^{\smash{\mathrm{sd,mot}}} (0)$
    for the trivial slope function~$0$.
    These can be obtained from
    \crefrange{eq-dt-mot-smooth}{eq-dt-mot-smooth-sd}
    by substituting in
    \cref{eq-def-epsilon-linear,eq-def-epsilon-sd},
    then using the relations
    \crefrange{eq-integral-homo-linear}{eq-integral-homo-sd}
    to reduce to the known integrals
    \crefrange{eq-quiver-motive}{eq-quiver-motive-sd}.
    This process also shows that
    $\mathrm{DT}_\alpha^\mathrm{mot} (0)$
    and $\mathrm{DT}_\theta^{\smash{\mathrm{sd,mot}}} (0)$
    are rational functions in~$\mathbb{L}^{1/2}$,
    and evaluating them at $\mathbb{L}^{1/2} = -1$
    gives the numerical invariants
    $\mathrm{DT}_\alpha (0)$
    and $\mathrm{DT}_\theta^\mathrm{sd} (0)$.

    Finally, for a general self-dual slope function~$\tau$,
    we may apply the wall-crossing formulae
    \crefrange{eq-wcf-dt}{eq-wcf-dt-mot-sd}
    to compute the invariants
    $\mathrm{DT}_\alpha (\tau)$,
    $\mathrm{DT}_\theta^\mathrm{sd} (\tau)$,
    $\mathrm{DT}_\alpha^\mathrm{mot} (\tau)$,
    and $\mathrm{DT}_\theta^{\smash{\mathrm{sd,mot}}} (\tau)$
    from the case when $\tau = 0$,
    which is already known.

    As an alternative to the final step,
    we may first compute the integrals
    $\smash{\int_{\mathcal{X}^\mathrm{ss}_\alpha (\tau)} \nu_\mathcal{X}^\mathrm{mot}}$
    and
    $\smash{\int_{\mathcal{X}^{\smash{\mathrm{sd,ss}}}_\theta (\tau)}
    \nu_{\subXsd}^\mathrm{mot}}$
    using the relations
    \crefrange{eq-wcf-anti-dom-delta}{eq-wcf-anti-dom-delta-sd},
    together with
    \crefrange{eq-integral-homo-linear}{eq-integral-homo-sd}
    to reduce to the known integrals
    \crefrange{eq-quiver-motive}{eq-quiver-motive-sd},
    then repeat the process above to obtain the invariants
    $\mathrm{DT}_\alpha^\mathrm{mot} (\tau)$
    and $\mathrm{DT}_\theta^{\smash{\mathrm{sd,mot}}} (\tau)$,
    which are rational functions in~$\mathbb{L}^{1/2}$.
    We then evaluate them at~$\mathbb{L}^{1/2} = -1$
    to obtain the numerical invariants
    $\mathrm{DT}_\alpha (\tau)$
    and $\mathrm{DT}_\theta^\mathrm{sd} (\tau)$.

    The author has implemented the above algorithm
    using a computer program,
    and some numerical results are presented below.
\end{para}

\begin{example}[The point quiver]
    \label{eg-point-quiver}
    Consider the point quiver~$Q$ with a single vertex and no edges,
    with the trivial slope function $\tau = 0$.
    There are two self-dual structures on~$Q$,
    with the signs $+1$ and $-1$ assigned to the vertex, respectively.

    We have the moduli stacks
    $\mathcal{X}_Q = \coprod_{n \geq 0} * / \mathrm{GL} (n)$
    and~$\mathcal{X}_Q^\mathrm{sd} =
    \coprod_{n \geq 0} * / \mathrm{O} (n)$
    or~$\coprod_{n \geq 0} * / \mathrm{Sp} (2n)$,
    depending on the sign of the vertex.
    As in \textcite[Example~7.19]{joyce-song-2012},
    the usual DT invariants of~$Q$ are given by
    \begin{equation*}
        \mathrm{DT}_{\mathsf{A}_{n - 1}} = \frac{1}{n^2}
    \end{equation*}
    for all $n \geq 1$, where the subscript~$\mathsf{A}_{n - 1}$ refers to
    the Dynkin type of $\mathrm{GL} (n)$.

    Based on explicit computation following the algorithm
    in \cref{para-quiver-algorithm},
    we conjecture that
    \begin{equation*}
        \mathrm{DT}^\mathrm{sd}_{\mathsf{B}_n} =
        \mathrm{DT}^\mathrm{sd}_{\mathsf{C}_n} =
        (-1)^n \, \binom{-1/4}{n} \ ,
        \qquad
        \mathrm{DT}^\mathrm{sd}_{\mathsf{D}_n} =
        (-1)^n \, \binom{\, 1/4 \,}{n} \ ,
    \end{equation*}
    where the subscripts~$\mathsf{B}_n$, $\mathsf{C}_n$, and~$\mathsf{D}_n$
    refer to the Dynkin types of
    $\mathrm{O} (2n + 1)$, $\mathrm{Sp} (2n)$, and $\mathrm{O} (2n)$, respectively.
    Equivalently, we have the generating series
    \begin{equation*}
        \sum_{n \geq 0} q^n \cdot \mathrm{DT}^\mathrm{sd}_{\mathsf{B}_n} =
        \sum_{n \geq 0} q^n \cdot \mathrm{DT}^\mathrm{sd}_{\mathsf{C}_n} =
        (1 - q)^{-1/4} \ ,
        \qquad
        \sum_{n \geq 0} q^n \cdot \mathrm{DT}^\mathrm{sd}_{\mathsf{D}_n} =
        (1 - q)^{1/4} \ .
    \end{equation*}
    We expect to prove this conjecture in a future paper \cite{epsilon-iii},
    and we expect that the coincidence of
    the type~$\mathsf{B}$ and~$\mathsf{C}$ invariants here
    should be related to the fact that these groups
    are Langlands dual to each other.
\end{example}

\begin{example}[The \texorpdfstring{$\tilde{\mathsf{A}}_1$}{Ã₁} quiver]
    \label{eg-tilde-a1}
    Let $Q = (\bullet \rightrightarrows \bullet)$
    be the quiver with two vertices and
    two arrows pointing in the same direction,
    called the \emph{$\tilde{\mathsf{A}}_1$ quiver}.
    Consider the contravariant involution of $Q$
    that exchanges the two vertices but fixes the edges.
    We use the simplified notation $\tilde{\mathsf{A}}_1^{\smash{u,v}}$,
    where $u, v$ are the signs in the self-dual structure.
    For example, $\tilde{\mathsf{A}}_1^{\smash{+,++}}$
    means that we take the sign $+1$ on all vertices and edges.
    Note that both vertices must have the same sign.
    We use the slope function $\tau = (1, -1)$.

    Based on numerical evidence from applying
    the algorithm in \cref{para-quiver-algorithm},
    we conjecture that we have the generating series
    \begin{numcases}{
            \sum_{n = 0}^\infty
            q^{n/2} \cdot \mathrm{DT}^{\mathrm{sd,mot}}_{\smash{(n,n)}} (\tau)
        = }
        \label{eq-tilde-a1-dt-1}
        \frac{(1 - q)^{1/2}}{(1 - q^{1/2} \, \mathbb{L}^{-1/2}) (1 - q^{1/2} \, \mathbb{L}^{1/2})}
        \hspace{-0.5em}
        & \text{for $\tilde{\mathsf{A}}_1^{\smash{+,++}}$
        and $\tilde{\mathsf{A}}_1^{\smash{-,--}}$} \ ,
        \\[.5ex]
        \label{eq-tilde-a1-dt-2}
        \biggl( \frac{1+q^{1/2}}{1-q^{1/2}} \biggr)^{1/2}
        & \text{for $\tilde{\mathsf{A}}_1^{\smash{+,+-}}$
        and $\tilde{\mathsf{A}}_1^{\smash{-,+-}}$} \ ,
        \\[.5ex]
        \label{eq-tilde-a1-dt-3}
        (1-q)^{1/2}
        & \text{for $\tilde{\mathsf{A}}_1^{\smash{+,--}}$
        and $\tilde{\mathsf{A}}_1^{\smash{-,++}}$} \ .
        {\addtocounter{equation}{-1}\refstepcounter{equation}\label{eq-tilde-a1-dt-3-fix}}
    \end{numcases}

    This example is related to coherent sheaves on~$\mathbb{P}^1$,
    as we will discuss in \cref{eg-dt-p1}.
\end{example}

\subsection{Orthosymplectic complexes}
\label{subsec-coh}

\begin{para}
    In the following sections,
    \crefrange{subsec-curves}{subsec-vw},
    we will apply our theory to study orthosymplectic DT invariants
    for certain smooth projective varieties over~$\mathbb{C}$,
    in three variants which apply to
    curves, surfaces, and threefolds, respectively.
    This section will provide background material
    that is common to these settings.

    Our DT invariants will count
    \emph{orthogonal} or \emph{symplectic complexes} on a variety,
    which are perfect complexes equipped with
    isomorphisms to their dual complexes.

    We note that this approach of defining a coherent-sheaf-like version of
    principal bundles is different from the related construction of
    \textcite{gomez-fernandez-herrero-zamora-2024-principal}.
\end{para}

\begin{para}[The setting]
    \label{para-coh-setup}
    Throughout, we work over the complex number field~$\mathbb{C}$,
    and we fix a connected, smooth, projective
    $\mathbb{C}$-variety~$Y$ of dimension~$n$.

    Let~$\bar{\mathcal{X}}$ be the derived moduli stack
    of perfect complexes on~$Y$,
    as in \textcite{toen-vaquie-2007-moduli}.
    It is a derived algebraic stack
    locally of finite presentation over~$\mathbb{C}$.
    By \textcite{pantev-toen-vaquie-vezzosi-2013},
    if~$Y$ is a \emph{Calabi--Yau $n$-fold},
    meaning that its canonical bundle~$K_Y$ is trivial,
    then~$\bar{\mathcal{X}}$
    has a $(2 - n)$-shifted symplectic structure.

    We fix the data $(I, L, s, \varepsilon)$,
    where $I \colon Y \simto Y$ is an involution,
    $L \to Y$ is a line bundle,
    $s \in \mathbb{Z}$,
    and $\varepsilon \colon L \simto I^* (L)$ is an isomorphism
    such that $I^* (\varepsilon) \circ \varepsilon = \mathrm{id}_L$.
    Define a self-dual structure on~$\mathsf{Perf} (Y)$
    by the dual functor
    \begin{equation}
        \label{eq-coh-dual-functor}
        \mathbb{D} = \mathbb{R} \calHom (I^* (-), L) [s] \colon
        \mathsf{Perf} (Y) \longsimto \mathsf{Perf} (Y)^\mathrm{op} \ ,
    \end{equation}
    and identify~$\mathbb{D} (\mathbb{D} (E))$ with~$E$
    using the isomorphism~$\varepsilon$,
    for all objects $E \in \mathsf{Perf} (Y)$.

    This self-dual structure induces a
    $\mathbb{Z}_2$-action on~$\bar{\mathcal{X}}$,
    and the fixed locus~$\bar{\mathcal{X}}^\mathrm{sd}$
    is the moduli stack of self-dual perfect complexes on~$Y$.
    In particular, when $I = \mathrm{id}_Y$
    and $\varepsilon = \pm \mathrm{id}_L$,
    the stack~$\bar{\mathcal{X}}^\mathrm{sd}$ parametrizes
    \emph{$L [s]$-twisted orthogonal} or \emph{symplectic complexes} on~$Y$,
    respectively.
    When $L = \mathcal{O}_Y$ and $s = 0$,
    they are simply called
    \emph{orthogonal} or \emph{symplectic complexes}.
\end{para}

\begin{para}[Bridgeland stability conditions]
    \label{para-bridgeland}
    Consider the free abelian group
    \begin{equation}
        K (Y) = \{ \mathrm{ch} (E) \mid
        E \in \mathsf{Perf} (Y) \}
        \subset \mathrm{H}^{2 \bullet} (Y; \mathbb{Q}) \ .
    \end{equation}
    It has an involution~$(-)^\vee$ given by
    $\mathrm{ch} (E) \mapsto \mathrm{ch} (\mathbb{D} (E))$.
    Let
    $K^\mathrm{sd} (Y) \subset K (Y)$
    be the fixed locus.

    Define a \emph{Bridgeland stability condition} on~$Y$
    as in \textcite[Definition~5.1]{bridgeland-2007-stability},
    which is a pair $\tau = (Z, \mathcal{P})$,
    where $Z \colon K (Y) \to \mathbb{C}$
    is a group homomorphism,
    and $\mathcal{P}$ is a slicing of $\mathsf{Perf} (Y)$.

    Let $\mathrm{Stab} (Y)$ be the set of
    Bridgeland stability conditions on~$Y$,
    which has a topology given by a
    \emph{generalized metric}~$d$,
    that is, a metric allowing infinite distance,
    defined as in~\cite[\S8]{bridgeland-2007-stability} by
    \begin{equation}
        d (\tau, \tilde{\tau}) = \sup {} \Bigl\{
            | \phi^+ (E) - \tilde{\phi}^+ (E) |,
            | \phi^- (E) - \tilde{\phi}^- (E) |,
            | {\log m (E)} - {\log \tilde{m} (E)} |
            \Bigm| E \neq 0
        \Bigr\} \ ,
    \end{equation}
    where $E$ runs through all non-zero objects of $\mathsf{Perf} (Y)$,
    and $\phi^+ (E)$, $\phi^- (E)$, $m (E)$ are the maximal phase,
    the minimal phase, and the sum of lengths of central charges
    of the $\tau$-HN factors of~$E$,
    respectively, and similarly for $\tilde{\tau}$.
    The projection
    \begin{equation}
        \mathrm{Stab} (Y) \longrightarrow \mathrm{Hom} (K (Y), \mathbb{C})
    \end{equation}
    given by $(Z, \mathcal{P}) \mapsto Z$ is a local homeomorphism,
    and equips $\mathrm{Stab} (Y)$ with the structure of a complex manifold.

    The self-dual structure on $\mathsf{Perf} (Y)$ defined in \cref{para-coh-setup}
    gives an anti-holomorphic involution
    $(-)^\vee \colon \mathrm{Stab} (Y) \to \mathrm{Stab} (Y)$,
    given by $(Z, \mathcal{P}) \mapsto (Z^\vee, \mathcal{P}^\vee)$,
    where $Z^\vee (\alpha) = \smash{\overline{Z (\alpha^\vee)}}$
    and $\mathcal{P}^\vee (t) = \mathcal{P} (-t)^\vee$.
    The fixed locus
    $\mathrm{Stab}^\mathrm{sd} (Y) \subset \mathrm{Stab} (Y)$
    is the set of \emph{self-dual stability conditions},
    and we have a local homeomorphism
    \begin{equation}
        \mathrm{Stab}^\mathrm{sd} (Y) \longrightarrow \mathrm{Hom} (K (Y), \mathbb{C})^{\mathbb{Z}_2} \ ,
    \end{equation}
    given by $(Z, \mathcal{P}) \mapsto Z$, where $\mathbb{Z}_2$ acts on $\mathrm{Hom} (K (Y), \mathbb{C})$
    via the anti-holomorphic involution $Z \mapsto Z^\vee$.
    This equips $\mathrm{Stab}^\mathrm{sd} (Y)$
    with the structure of a real analytic manifold.
\end{para}

\begin{para}[Permissibility]
    \label{para-bridgeland-permissibility}
    We define subspaces of
    \emph{permissible Bridgeland stability conditions},
    \begin{equation*}
        \mathrm{Stab}^\circ (Y) \subset \mathrm{Stab} (Y) \ ,
        \qquad
        \mathrm{Stab}^{\circ, \mathrm{sd}} (Y) \subset \mathrm{Stab}^\mathrm{sd} (Y)
    \end{equation*}
    as maximal open subsets such that
    every element $\tau = (Z, \mathcal{P})$
    with $Z (K (Y)) \subset \mathbb{Q} + \mathrm{i} \mathbb{Q}$
    satisfies the following conditions:
    \begin{enumerate}
        \item
            \label{item-bridgeland-supp}
            \emph{Support property}.
            For any $r > 0$, there are only finitely many classes
            $\alpha \in K (Y)$
            admitting a semistable object,
            such that $|Z (\alpha)| \leq r$.

        \item
            \label{item-bridgeland-open}
            \emph{Generic flatness}.
            See \textcite[Problem~3.5.1]{abramovich-polishchuk-2006},
            \textcite[Definition~6.2.4]{halpern-leistner-instability},
            or \textcite[Definition~4.4]{piyaratne-toda-2019-bridgeland}
            for the formulation.

        \item
            \label{item-bridgeland-bounded}
            \emph{Boundedness}.
            For any $t \in \mathbb{R}$
            and $\alpha \in K (Y)$
            with $Z (\alpha) \in \mathbb{R}_{\geq 0} \cdot \mathrm{e}^{\uppi \mathrm{i} t}$,
            there is a quasi-compact open substack
            $\mathcal{X} (\tau; t)_\alpha \subset \bar{\mathcal{X}}$
            whose $\mathbb{C}$-points are
            the objects of~$\mathcal{P} (t)$ of class~$\alpha$.
    \end{enumerate}
    By \textcite[Proposition~4.12]{piyaratne-toda-2019-bridgeland},
    if a stability condition~$\tau$ satisfies these conditions
    and has rational central charge,
    then a neighbourhood of~$\tau$ lies in~$\mathrm{Stab}^\circ (Y)$.

    For $\tau \in \mathrm{Stab}^\circ (Y)$
    and an interval $J \subset \mathbb{R}$
    of length $|J| < 1$,
    there is an open substack
    \begin{equation*}
        \mathcal{X} (\tau; J) \subset \bar{\mathcal{X}}
    \end{equation*}
    whose $\mathbb{C}$-points are the objects of~$\mathcal{P} (J)$,
    which we construct in \cref{para-bridgeland-openness} below.
    It is a derived linear stack in the sense of
    \cref{para-derived-linear-stacks},
    and $\tau$ defines a permissible stability condition
    on its classical truncation in the sense of
    \cref{subsec-stack-stability},
    where the $\Theta$-stratification is constructed in
    \cref{para-bridgeland-openness} below.

    In particular, if $\tau \in \mathrm{Stab}^{\circ, \mathrm{sd}} (Y)$
    and $J = -J$,
    then $\mathcal{X} (\tau; J)$ is a self-dual derived linear stack,
    and the induced stability condition on $\mathcal{X} (\tau; J)$ is self-dual.
    The stack~$\mathcal{X} (\tau; 0)^\mathrm{sd}$
    is the moduli stack of
    \emph{$\tau$-semistable self-dual complexes},
    which our orthosymplectic DT invariants will count.
\end{para}

\begin{example}
    \label{eg-bridgeland}
    Let~$Y$ be either a curve, a surface,
    or a threefold satisfying the conjectural
    Bogomolov--Gieseker inequality of
    \textcite[Conjecture~3.2.7]{bayer-macri-toda-2014-bridgeland},
    and fix the data $(I, L, s, \varepsilon)$
    as in \cref{para-orientation-data}.
    Suppose we are given an ample class
    $\omega \in \mathrm{H}^{1,1} (Y; \mathbb{Q})$
    with $I^* (\omega) = \omega$.

    In this case, we give an example of a
    permissible self-dual Bridgeland stability condition,
    with central charge valued in $\mathbb{Q} + \mathrm{i} \mathbb{Q}$.

    Let $\beta = c_1 (L) / 2 \in \mathrm{H}^2 (Y; \mathbb{Q})$.
    Consider the group homomorphism
    $Z_\omega \colon K (Y) \to \mathbb{C}$
    given by
    \begin{equation}
        Z_\omega (\alpha) =
        \mathrm{i}^{n - s} \cdot
        \int_Y \exp (-\beta - \mathrm{i} \omega) \cdot \alpha
    \end{equation}
    for $\alpha \in K (Y)$, where $n = \dim Y$.
    This is compatible with the self-dual structure,
    in the sense that we have
    $Z_\omega (\mathbb{D} (\alpha)) = \smash{\overline{Z_\omega (\alpha)}}$
    for all $\alpha \in K (Y)$.
    Here, our coefficient~$\mathrm{i}^{n - s}$
    is only inserted to make~$Z_\omega$ self-dual,
    and does not essentially affect the stability condition.

    There is a Bridgeland stability condition
    $\tau_\omega = (Z_\omega, \mathcal{P}_\omega) \in \mathrm{Stab}^\circ (Y)$
    with central charge~$Z_\omega$,
    by the works of \textcite{toda-2008-k3}
    and \textcite{piyaratne-toda-2019-bridgeland}.
    See also the earlier works of
    \textcite{bridgeland-2008-k3}
    and \textcite{arcara-bertram-2012-surfaces}
    in the case of surfaces.

    In fact, we can also choose~$\mathcal{P}_\omega$
    so that~$\tau_\omega$ is self-dual, or equivalently,
    the slicing~$\mathcal{P}_\omega$ coincides with its
    dual slicing~$\mathcal{P}_\omega^\vee$ given by
    \begin{equation*}
        \mathcal{P}_\omega^\vee (t) = \mathbb{D} (\mathcal{P}_\omega (-t)) \ .
    \end{equation*}
    This follows from
    \textcite[Remark~4.4.3]{bayer-macri-toda-2014-bridgeland},
    which is essentially the same statement in the case when $s = 1$,
    and the general case is constructed from this case
    by simply shifting the phase by $(1 - s) / 2$.
\end{example}

\begin{para}
    \label{para-bridgeland-openness}
    In the situation of \cref{para-bridgeland-permissibility},
    for $\tau \in \mathrm{Stab}^\circ (Y)$
    and an interval $J \subset \mathbb{R}$ with $|J| < 1$,
    we construct the open substack
    $\mathcal{X} (\tau; J) \subset \bar{\mathcal{X}}$
    and its $\Theta$-stratification by $\tau$-HN types as follows.

    Applying \textcite[Proposition~4.12]{piyaratne-toda-2019-bridgeland},
    we may apply a phase shift and assume that
    $J \subset \mathopen{]} \varepsilon, 1 - \varepsilon \mathclose{[}$
    for some $\varepsilon > 0$.
    Fix $\alpha \in K (Y)$ of slope within~$J$,
    and then choose a perturbation~$\tau' = (Z', \mathcal{P}')$ of~$\tau$
    satisfying the above properties,
    with $d (\tau', \tau) < \varepsilon$
    and $Z' (K (Y)) \subset \mathbb{Q} + \mathrm{i} \mathbb{Q}$.
    Then if~$\beta \in K (Y)$ is the class of a $\tau$-HN factor
    of an object of $\mathcal{P} (J)$ of class~$\alpha$,
    then $Z' (\beta)$ must lie in the bounded region
    \begin{equation*}
        \bigl\{ r \mathrm{e}^{\uppi \mathrm{i} t} \mid
        r \geq 0, \ t \in J_\varepsilon \bigr\}
        \cap
        \bigl\{ Z' (\alpha) - r \mathrm{e}^{\uppi \mathrm{i} t} \mid
        r \geq 0, \ t \in J_\varepsilon \bigr\}
        \subset \mathbb{C} \ ,
    \end{equation*}
    where $J_\varepsilon$ is the $\varepsilon$-neighbourhood of~$J$,
    so the set~$B$ of such classes~$\beta$ is finite.
    We then choose~$\varepsilon$ small enough,
    possibly changing~$\tau'$,
    so that for any $\beta, \beta' \in B$,
    $\arg Z (\beta) < \arg Z (\beta')$
    implies $\arg Z' (\beta) < \arg Z' (\beta')$,
    where we take phases within $J_\varepsilon$.
    Now, \textcite[Theorem~6.5.3]{halpern-leistner-instability}
    gives the open substack
    $\mathcal{X} (\tau'; \mathopen{]} 0, 1 \mathclose{[})$
    with a $\Theta$-stratification by $\tau'$-HN types.
    The part of~$\mathcal{X} (\tau; J)$ lying in~$\mathcal{X}_\alpha$
    can be defined as a finite open union of strata.

    To construct the $\Theta$-stratification on~$\mathcal{X} (\tau; J)$,
    we follow the proof of
    \cite[Theorem~6.5.3]{halpern-leistner-instability},
    with the following modifications.
    Instead of using rational weights for HN~filtrations,
    we use \emph{real-weighted filtrations}
    in the sense of \cite[\S\S 7.2--7.3]{epsilon-i}.
    As a result, we obtain real-weighted $\Theta$-stratifications,
    which non-canonically give usual $\Theta$-stratifications
    by \cite[Proposition~7.2.12]{epsilon-i}.
    The key ingredients of the proof in \cite{halpern-leistner-instability}
    are the conditions (R), (S), and (B) there.
    The rationality condition~(R)
    is no longer needed as we use real weights.
    The condition (S) needs to be modified to incorporate real weights,
    but the argument still works to prove it.
    The condition~(B) follows from the quasi-compactness of
    $\mathcal{X} (\tau; J)$.

    This also shows that any $\tau \in \mathrm{Stab}^\circ (Y)$
    satisfies the support property and the boundedness property
    in \cref{para-bridgeland-permissibility},
    where for the support property,
    fixing $r > 0$ and choosing $\tau'$ rational
    with $d (\tau', \tau) < \varepsilon$ with $\varepsilon < 1/2$,
    for any class $\alpha$ with $|Z (\alpha)| \leq r$
    admitting a $\tau$-semistable object~$E$,
    by considering the $\tau'$-HN filtration of~$E$,
    we see that~$\alpha$ is a finite sum of classes~$\beta$
    with $|Z' (\beta)| < r \mathrm{e}^\varepsilon$
    admitting $\tau'$-semistable objects,
    and these classes lie on the same side of a line
    in~$\mathbb{C}$, so there are only finitely many choices.
\end{para}

\subsection{DT invariants for curves}
\label{subsec-curves}

\begin{para}
    We define DT invariants
    counting orthogonal and symplectic bundles on a curve.
    These are orthosymplectic versions of
    Joyce's motivic invariants counting vector bundles on a curve,
    as in \cite[\S6.3]{joyce-2008-configurations-iv}.
\end{para}

\begin{para}
    Let~$C$ be a connected, smooth, projective curve over~$\mathbb{C}$,
    and fix the data $(I, L, s, \varepsilon)$
    as in \cref{para-coh-setup}.
    This defines a self-dual structure on~$\mathsf{Perf} (C)$.

    Let $\tau = (Z, \mathcal{P})$ be the Bridgeland stability condition
    defined in \cref{eg-bridgeland},
    where we choose the unique element
    $\omega \in \mathrm{H}^2 (C; \mathbb{Q})$
    with $\int_C \omega = 1$.
    Explicitly, we have
    \begin{equation}
        Z (E) = \mathrm{i}^{-s} \cdot \Bigl(
            \Bigl( 1 - \mathrm{i} \, \frac{\deg L}{2} \Bigr) \, r + d
        \Bigr)
    \end{equation}
    for $E \in \mathsf{Perf} (C)$ with rank~$r$ and degree~$d$,
    so that $\operatorname{ch} (E) = r + d \omega$.
    Note that the choices of~$L$ and~$s$
    do not affect which objects are semistable,
    although they affect which objects are self-dual.
    The subcategory $\mathsf{Vect} (C) \subset \mathsf{Perf} (C)$
    of vector bundles on~$C$ satisfies
    $\mathsf{Vect} (C) = \mathcal{P} \bigl(
        \mathopen{]} (-1-s)/2, (1-s)/2 \mathclose{[}
    \bigr)$.
\end{para}

\begin{para}[The even case]
    \label{para-curves-even}
    When $s$ is even, the abelian category $\mathcal{P} (0)$
    consists of objects $E [s/2]$
    for semistable vector bundles~$E$ on~$C$
    in the usual sense,
    whose rank~$r$ and degree~$d$ satisfy $d = r \deg L / 2$.
    The self-dual objects are such~$E$ with isomorphisms
    $\phi \colon E \simto \calHom (I^* (E), L)$
    with $I^* (\phi)^\vee \circ \phi = (-1)^{s/2} \cdot \varepsilon$.

    In~particular, when $L = \mathcal{O}_C$,
    semistable self-dual complexes can be identified,
    up to a shift,
    with orthogonal or symplectic
    bundles on~$C$, depending on whether
    $(-1)^{s/2} \cdot \varepsilon = 1$ or $-1$,
    whose underlying vector bundles are semistable in the usual sense.

    For each rank $r > 0$,
    we have the self-dual DT invariants
    \begin{equation*}
        \mathrm{DT}^\mathrm{sd}_{r} \in \mathbb{Q} \ , \qquad
        \mathrm{DT}^\mathrm{sd, mot}_{r}
        \in \hat{\mathbb{M}}^\mathrm{mon} (K; \mathbb{Q}) \ ,
    \end{equation*}
    counting semistable self-dual vector bundles of rank~$r$ as above,
    defined as in \cref{para-dt-smooth-stacks,para-mot-dt-smooth-stacks}
    using the self-dual linear stack $\mathcal{X} (\tau; 0)$ defined in
    \cref{para-bridgeland-permissibility}
    with the trivial stability condition.
\end{para}

\begin{para}[The odd case]
    \label{para-curves-odd}
    When $s$ is odd, $\mathcal{P} (0)$
    consists of objects $E [(s-1)/2]$
    for torsion sheaves~$E$ on~$C$,
    and the semistable self-dual objects are such~$E$
    with isomorphisms
    $\phi \colon E \simto \mathbb{R} \calHom (I^* (E), \allowbreak L [1])$
    with $I^* (\phi)^\vee \circ \phi = (-1)^{(s-1)/2} \cdot \varepsilon$.
    For each degree $d > 0$, we have the self-dual DT invariants
    \begin{equation*}
        \mathrm{DT}^\mathrm{sd}_{0, d} \in \mathbb{Q} \ , \qquad
        \mathrm{DT}^{\smash{\mathrm{sd, mot}}}_{0, d}
        \in \hat{\mathbb{M}}^\mathrm{mon} (K; \mathbb{Q}) \ ,
    \end{equation*}
    counting these self-dual torsion sheaves,
    defined similarly as above.

    In fact, these invariants do not depend on the choice of~$L$,
    since choosing a suitable $I$-invariant open cover of~$C$ trivializing~$L$,
    torsion sheaves supported on the open sets
    give an open cover of the moduli stacks,
    where pieces and intersections do not depend on~$L$.
    It then follows from \cite[Theorem~5.2.10~(i)]{epsilon-ii}
    that the invariants do not depend on~$L$.
\end{para}

\begin{example}[Invariants for \texorpdfstring{$\mathbb{P}^1$}{P¹}]
    \label{eg-dt-p1}
    Consider the case when $C = \mathbb{P}^1$
    and $I = \mathrm{id}_{\mathbb{P}^1}$.
    We describe the invariants in two situations.

    When $s = 0$ and $L = \mathcal{O}_{\mathbb{P}^1}$,
    since every vector bundle on~$\mathbb{P}^1$
    splits as a direct sum of line bundles,
    all semistable vector bundles of slope~$0$ are trivial bundles.
    The self-dual abelian category~$\mathcal{P} (0)$
    is thus equivalent to the category of
    finite-dimensional $\mathbb{C}$-vector spaces,
    with one of the two self-dual structures
    described in \cref{eg-point-quiver},
    depending on the sign~$\varepsilon$.
    The DT~invariants agree with the ones given there.

    When $s = 1$, invariants for~$\mathbb{P}^1$
    are related to DT invariants for self-dual quivers.
    Indeed, as a special case of \textcite[Theorem~6.2]{bondal-1990},
    we have an equivalence
    \begin{equation}
        \label{eq-dbcoh-equiv-dbmod}
        \Phi \colon \mathsf{Perf} (\mathbb{P}^1) \longsimto
        \mathsf{D}^\mathrm{b} \mathsf{Mod} (\mathbb{C} Q) \ ,
    \end{equation}
    where $Q$ is the~$\tilde{\mathsf{A}}^1$ quiver in
    \cref{eg-tilde-a1},
    and $\Phi (E) = \bigl(
        \mathbb{R} \Gamma (\mathbb{P}^1, E (-1))
        \rightrightarrows
        \mathbb{R} \Gamma (\mathbb{P}^1, E)
    \bigr)$,
    with the two maps given by multiplying with the coordinate functions
    $x_0, x_1 \in \Gamma (\mathbb{P}^1, \mathcal{O}_{\mathbb{P}^1} (1))$.

    \begin{figure}[t]
        \centering
        \scalebox{.9}{%
        \begin{tikzpicture}
            \begin{scope}[shift={(-4.5, 0)}]
                \fill [black!10] (-2.7, 1.35) -- (2.7, -1.35) -- (2.7, 2.7) -- (-2.7, 2.7);
                \fill [black!30] (-2.7, 1.325) -- (2.7, -1.375) -- (2.7, -1.325) -- (-2.7, 1.375);
                \draw [-{Stealth[length=6]}]
                    (0, -3) -- (0, 3) node [anchor=south] {$d$};
                \draw [-{Stealth[length=6]}]
                    (-3, 0) -- (3, 0) node [anchor=west] {$r$};
                \foreach \i in {-2, ..., 2} {
                    \fill (1, \i) circle (.05)
                        node [anchor=north west] {$\mathcal{O} (\i)$};
                }
                \foreach \i in {2, 1, 0, -1, -2} {
                    \fill (-1, -\i) circle (.05)
                        node [anchor=south east] {$\mathcal{O} (\i) [1]$};
                }
                \foreach \i in {-2, ..., 2} {
                    \fill (0, \i) circle (.05);
                }
                \foreach \i in {-1, ..., 2} {
                    \draw [thick, Stealth-Stealth, shorten > = 5, shorten < = 5] (-1, \i) -- (1, \i-1);
                }
                \node at (0, -3.5) {$\mathsf{Perf} (\mathbb{P}^1)$};
            \end{scope}
            \begin{scope}
                \node [anchor=south] at (0, 0) {$\Phi$};
                \node [anchor=north] at (0, 0) {$\simeq$};
                \draw [thick, -Stealth] (-.4, 0) -- (.4, 0);
            \end{scope}
            \begin{scope}[shift={(4.2, 0)}]
                \fill [black!10] (0, -2.7) -- (2.7, -2.7) -- (2.7, 2.7) -- (0, 2.7);
                \draw [-{Stealth[length=6]}]
                    (0, -3) -- (0, 3) node [anchor=south] {$d_0 - d_1$};
                \draw [-{Stealth[length=6]}]
                    (-3, 0) -- (3, 0) node [anchor=west] {$d_0 + d_1$};
                \foreach \i in {-2, ..., 1} {
                    \fill (1.5 * \i + .75, .75) circle (.05);
                    \fill (1.5 * \i + .75, -.75) circle (.05);
                }
                \foreach \i in {-1, 0, 1} {
                    \fill (1.5 * \i, 0) circle (.05);
                }
                \foreach \i in {-2, ..., 1} {
                    \draw [thick, Stealth-Stealth, shorten > = 5, shorten < = 5] (1.5 * \i + .75, -.75) -- (1.5 * \i + .75, .75);
                }
                \node [anchor=south] at (-2.25, .75) {$(-1, -2)$};
                \node [anchor=south] at (-.75, .75) {$(0, -1)$};
                \node [anchor=south] at (.75, .75) {$(1, 0)$};
                \node [anchor=south] at (2.25, .75) {$(2, 1)$};
                \node [anchor=north] at (-2.25, -.75) {$(-2, -1)$};
                \node [anchor=north] at (-.75, -.75) {$(-1, 0)$};
                \node [anchor=north] at (.75, -.75) {$(0, 1)$};
                \node [anchor=north] at (2.25, -.75) {$(1, 2)$};
                \node [anchor=south, scale=.8] at (-1.5, 0) {$(-1, -1)$};
                \node [anchor=south] at (1.5, 0) {$(1, 1)$};
                \node at (0, -3.5) {$\mathsf{D}^\mathrm{b} \mathsf{Mod} (\mathbb{C} Q)$};
            \end{scope}
        \end{tikzpicture}
        }
        \caption{An equivalence of categories}
        \label{fig-coh-p1-duality}
    \end{figure}
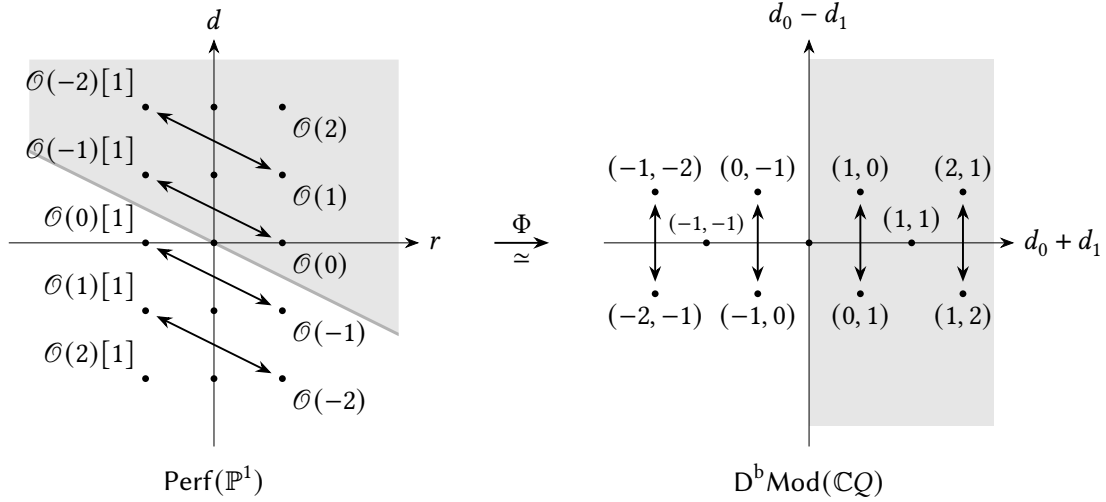

    In fact, under the isomorphism~$\Phi$,
    the self-dual structure on $\mathsf{Perf} (\mathbb{P}^1)$ given by
    $(I = \mathrm{id}_{\mathbb{P}^1}, \allowbreak
    L = \mathcal{O}_{\mathbb{P}^1} (-1), \allowbreak
    s = 1, \varepsilon)$
    corresponds to the self-dual structure on
    $\mathsf{D}^\mathrm{b} \mathsf{Mod} (\mathbb{C} Q)$
    given by the signs $(\varepsilon, ++)$
    in the notation of \cref{eg-tilde-a1},
    as shown in \cref{fig-coh-p1-duality}.
    Here, $r$ and~$d$ denote the rank and degree of a complex on~$\mathbb{P}^1$,
    and $(d_0, d_1)$ is the dimension vector of a representation of~$Q$.
    The two-way arrows indicate the dual operation,
    and the self-dual objects lie on the vertical axis on the left-hand side,
    or the horizontal axis on the right-hand side.
    The shaded regions indicate the usual heart of
    $\mathsf{D}^\mathrm{b} \mathsf{Mod} (\mathbb{C} Q)$
    and the corresponding heart of $\mathsf{Perf} (\mathbb{P}^1)$.
    The right-hand side can also be viewed
    either as the central charge of~$\tau$,
    or that of the stability condition on~$Q$
    given by the slope function $(1, -1)$.

    In particular, the DT invariants in this case
    coincide with those in \cref{eg-tilde-a1},
    and should be given by the conjectural formulae
    \cref{eq-tilde-a1-dt-1,eq-tilde-a1-dt-3-fix}.
\end{example}

\subsection{DT invariants for threefolds}
\label{subsec-threefolds}

\begin{para}
    We define DT invariants
    counting orthogonal or symplectic complexes on a Calabi--Yau threefold.
    These invariants are one of the main applications of our theory,
    and are an extension of the usual
    DT invariants studied by
    \textcite{thomas-2000-dt},
    \textcite{joyce-song-2012},
    \textcite{kontsevich-soibelman-motivic-dt}, and many others.
    We expect our invariants to be related to
    counting D-branes on Calabi--Yau $3$-orientifolds,
    as discussed in
    \textcite[\S5.2]{witten-1998-d-branes},
    \textcite{diaconescu-2007-orientifolds},
    and \textcite{hori-walcher-2008-orientifolds}.

    We also prove wall-crossing formulae for these invariants
    in \cref{thm-threefold-wcf},
    which relate the invariants for different
    Bridgeland stability conditions.
\end{para}

\begin{para}[Invariants]
    \label{para-threefold-inv}
    Let~$Y$ be a Calabi--Yau threefold over~$\mathbb{C}$,
    and fix the data $(I, L, s, \varepsilon)$
    as in \cref{para-coh-setup}.

    In this case, the derived stack $\bar{\mathcal{X}}$
    and the fixed locus~$\bar{\mathcal{X}}^\mathrm{sd}$
    are $(-1)$-shifted symplectic stacks.
    The stack~$\bar{\mathcal{X}}$ has an orientation data
    in the sense of \cref{para-orientation-data},
    by \textcite[Theorem~3.6]{joyce-upmeier-2021-orientation-data}.
    However, we do not know if the stack~$\bar{\mathcal{X}}^\mathrm{sd}$
    has an orientation in general.

    Let $\tau = (Z, \mathcal{P}) \in \mathrm{Stab}^{\circ, \mathrm{sd}} (Y)$
    be a self-dual Bridgeland stability condition on~$Y$,
    which is guaranteed to exist in the situation of \cref{eg-bridgeland}.
    For each phase $t \in \mathbb{R}$,
    let $\mathcal{X} (\tau; t) \subset \bar{\mathcal{X}}$
    be the open substack as in \cref{para-bridgeland-permissibility},
    which is a $(-1)$-shifted symplectic linear stack,
    and has a self-dual structure when $t = 0$.

    Given a class $\alpha \in K (Y)$ with
    $Z_\omega (\alpha) \in \mathbb{R}_{> 0} \cdot \mathrm{e}^{\uppi \mathrm{i} t}$,
    define the numerical and motivic DT invariants
    \begin{equation*}
        \mathrm{DT}_\alpha (\tau) \in \mathbb{Q} \ , \qquad
        \mathrm{DT}_\alpha^\mathrm{mot} (\tau)
        \in \hat{\mathbb{M}}^\mathrm{mon} (\mathbb{C}; \mathbb{Q}) \ ,
    \end{equation*}
    as in \cref{para-dt-linear,para-motivic-dt}
    for the stack~$\mathcal{X} (\tau; t)$
    with the trivial stability condition,
    where we take the sum of DT invariants
    of connected components of the open and closed substack
    $\mathcal{X} (\tau; t)_\alpha \subset \mathcal{X} (\tau; t)$,
    and we use the orientation of
    \textcite{joyce-upmeier-2021-orientation-data}
    for the motivic version.
    These invariants are not new,
    and can be constructed from the formalisms of
    \textcite{joyce-song-2012} and
    \textcite{kontsevich-soibelman-motivic-dt}.

    When $t = 0$, for each $\theta \in K^\mathrm{sd} (Y)$
    with $Z_\omega (\theta) \in \mathbb{R}_{> 0}$,
    we have the numerical self-dual DT invariant
    \begin{equation*}
        \mathrm{DT}^\mathrm{sd}_\theta (\tau) \in \mathbb{Q} \ ,
    \end{equation*}
    defined as in \cref{para-dt-sd}
    for the self-dual linear stack $\mathcal{X} (\tau; 0)$
    with the trivial stability condition,
    where we sum over connected components of
    $\mathcal{X} (\tau; 0)^\mathrm{sd}_\theta$.
    These are new invariants for Calabi--Yau threefolds,
    and are one of the main constructions of this paper.

    If one can construct a self-dual orientation data on
    $\mathcal{X} (\tau; 0)$
    in the sense of \cref{para-self-dual-orientation-data},
    then the motivic self-dual DT invariant
    $\mathrm{DT}^{\smash{\mathrm{mot,sd}}}_\theta (\tau)$
    will also be defined,
    as in \cref{para-motivic-dt}.
\end{para}

\begin{theorem}
    \label{thm-threefold-wcf}
    Let $Y$ be a Calabi--Yau threefold over $\mathbb{C}$.
    Choose the data $(I, L, s, \varepsilon)$
    as in \cref{para-coh-setup}.
    Let $\tau = (Z, \mathcal{P}),$
    $\tilde{\tau} = (\tilde{Z}, \tilde{\mathcal{P}})
    \in \mathrm{Stab}^\circ (Y)$
    be Bridgeland stability conditions.

    \begin{enumerate}
        \item
            \label{item-threefold-linear}
            If $\tau, \tilde{\tau}$
            can be connected by a path of length~$< 1/4$
            in $\mathrm{Stab}^\circ (Y)$,
            then for any class $\alpha \in K (Y)$ with
            $Z (\alpha) \neq 0$,
            the wall-crossing formula
            \cref{eq-wcf-dt} holds.

        \item
            \label{item-threefold-sd}
            If $\tau, \tilde{\tau} \in \mathrm{Stab}^{\circ, \mathrm{sd}} (Y)$,
            and they can be connected by a path of length~$< 1/4$
            in $\mathrm{Stab}^{\circ, \mathrm{sd}} (Y)$,
            then for any class $\theta \in K^\mathrm{sd} (Y)$ with
            $Z (\theta) \in \mathbb{R}_{> 0}$,
            the wall-crossing formula
            \cref{eq-wcf-dt-sd} holds.
    \end{enumerate}
    Here, the length of a path is defined as the supremum
    of sums of distances over all subdivisions.
    In \crefrange{eq-wcf-dt}{eq-wcf-dt-sd},
    we use $\tau, \tilde{\tau}$ in place of $\tau_+, \tau_-$.
    The sets $\uppi_0 (\mathcal{X}),$ $\uppi_0 (\mathcal{X}^\mathrm{sd})$
    in the formulae are defined using
    $\mathcal{X} = \mathcal{X} (\tau; \mathopen{]} t-1/4, t+1/4 \mathclose{[})$,
    where~$t$ is a phase of\/~$Z (\alpha)$ in~\cref{item-threefold-linear}
    or $t = 0$ in~\cref{item-threefold-sd}.
    The coefficients $\tilde{U} (\ldots)$,
    $\tilde{U}^\mathrm{sd} (\ldots)$
    are defined using the total order on phases in
    $\mathopen{]} t-1/2, t+1/2 \mathclose{[}$.

    Moreover, if we are given
    an orientation data on
    $\mathcal{X} (\tau; \mathopen{]} t-1/2, t+1/2 \mathclose{[})$,
    or a self-dual orientation data on
    $\mathcal{X} (\tau; \mathopen{]} -1/2, 1/2 \mathclose{[})$,
    respectively,
    then~\crefrange{item-threefold-linear}{item-threefold-sd}
    also hold for the motivic versions
    \crefrange{eq-wcf-dt-mot}{eq-wcf-dt-mot-sd},
    where~$\alpha$ has phase~$t$.
\end{theorem}

\begin{proof}
    \allowdisplaybreaks
    To avoid repetition,
    we prove \crefrange{item-threefold-linear}{item-threefold-sd}
    using a common argument, where we write~$\alpha$ for~$\theta$ for~\cref{item-threefold-sd}.

    We first prove the following claim:
    For fixed~$\tau$ and a fixed class~$\alpha$ or~$\theta$,
    there exists $\delta > 0$ such that
    the wall-crossing formulae hold whenever
    $d (\tau, \tilde{\tau}) < \delta$,
    with the sets $\uppi_0 (\mathcal{X})$, $\uppi_0 (\mathcal{X}^\mathrm{sd})$
    defined using $\mathcal{X} = \mathcal{X} (\tau; t)$,
    and we may take $\tau_+, \tau_-$ in the formulae to be
    either $\tau, \tilde{\tau}$
    or $\tilde{\tau}, \tau$.

    Write $\mathcal{A} = \mathcal{P} (t)$.
    Let $K \subset K (Y)$ be the set of Chern characters
    of $\tau$-semistable objects in $\mathsf{Perf} (Y)$,
    and $C \subset K$ the set of classes
    realized by objects in~$\mathcal{A}$.

    We choose $0 < \delta < 1/8$ such that
    $K \cap Z^{-1} (V_{4 \delta} (\mathrm{e}^{2 \delta} \cdot Z (\alpha))) \subset C$, where
    \[
        V_u (z) =
        \{ r \mathrm{e}^{\uppi \mathrm{i} \phi} \mid 0 \leq r \leq |z| \ , \ |\phi| \leq u \}
        \subset \mathbb{C} \ .
    \]
    If $\beta \in K (Y)$ is the class of a $\tau$-HN factor
    of a $\tilde{\tau}$-semistable object of class~$\alpha$,
    then $Z (\beta)$ must lie in $V_{2 \delta} (Z (\alpha))$.
    By the choice of~$\delta$,
    all such classes~$\beta$ have phase~$t$,
    and are hence equal to~$\alpha$.
    This implies that all $\tilde{\tau}$-semistable objects of class~$\alpha$
    are $\tau$-semistable and are in~$\mathcal{A}$.

    Similarly, we may assume that
    all $\tilde{\tau}$-semistable objects with phase in
    $[t - \delta, t + \delta]$
    and norm $\leq \mathrm{e}^\delta \cdot |Z (\alpha)|$
    are in~$\mathcal{A}$.
    Indeed, such objects have $\tau$-phase in
    $[t - 2 \delta, t + 2 \delta]$
    and $\tau$-norm $\leq \mathrm{e}^{2 \delta} \cdot |Z (\alpha)|$,
    and this property holds by the choice of~$\delta$.

    It follows that for any object in~$\mathcal{A}$ of class~$\alpha$,
    its $\tilde{\tau}$-HN factors
    also belong to~$\mathcal{A}$.
    In other words,
    $\tilde{\tau}$ almost defines a stability condition on
    $\mathcal{X} (\tau; t)$
    in the sense of \cref{para-linear-stack-stability},
    except that the $\Theta$-stratification is only defined on
    $\mathcal{X} (\tau; t)_\beta$
    for classes $\beta \in C$ with $|Z (\beta)| \leq |Z (\alpha)|$.
    However, this is enough to prove wall-crossing
    for~$\alpha$, as the other classes are irrelevant in the argument.
    The claim thus follows from \cref{thm-wcf-dt},
    where~$\tau$ corresponds to trivial stability on~$\mathcal{X} (\tau; t)$.

    We now turn to the original statement of the theorem.
    Choose a path $(\tau_s = (Z_s, \mathcal{P}_s))_{s \in [0, 1]}$
    of length $\ell < 1/4$,
    with $\tau_0 = \tau$ and $\tau_1 = \tilde{\tau}$.
    By the compactness of $[0, 1]$, our claim implies that
    we can choose $0 = s_0 < \cdots < s_n = 1$
    such that there are wall-crossing formulae
    between each $\tau_{s_i}$ and $\tau_{s_{i + 1}}$.
    We may thus apply \cref{eq-wcf-dt}, etc.,
    to express $\mathrm{DT}_\alpha (\tau_{s_0})$, etc.,
    in terms of invariants for~$\tau_{s_1}$,
    and so on,
    finally in terms of invariants for
    $\tau_{s_n} = \tilde{\tau}$.
    In each step, the involved invariants
    $\mathrm{DT}_{\beta} (\tau_{s_i})$
    must satisfy that~$Z (\beta)$ lies in the bounded region
    \begin{equation*}
        \{ r \mathrm{e}^{\uppi \mathrm{i} \phi} \mid
        r \geq 0 \ , \ |\phi - t| \leq \ell \}
        \cap
        \{ Z (\alpha) - r \mathrm{e}^{\uppi \mathrm{i} \phi} \mid
        r \geq 0 \ , \ |\phi - t| \leq \ell \}
        \subset \mathbb{C} \ ,
    \end{equation*}
    so that the sums \cref{eq-wcf-dt}, etc.,
    can not only be written using some
    $\uppi_0 (\mathcal{X} (\tau_{s_i}; t_i))$
    and its self-dual version,
    as in the argument above,
    but also using the larger set
    $\uppi_0 (\mathcal{X} (\tau; \mathopen{]} t-1/2, t+1/2 \mathclose{[}))$
    and its self-dual version,
    where the coefficients
    $\tilde{U} (\ldots)$, $\tilde{U}^{\mathrm{sd}} (\ldots)$
    are zero for the newly introduced terms.
    The support property of~$\tau$
    ensures that only finitely many non-zero terms
    appear in each step.

    It remains to prove that the coefficients
    $\tilde{U} (\ldots)$, $\tilde{U}^{\mathrm{sd}} (\ldots)$
    respect composition of wall-crossing formulae,
    so that the wall-crossing formulae
    obtained from the above process are equivalent to
    \cref{eq-wcf-dt}, etc.,
    from~$\tau$ directly to~$\tilde{\tau}$.
    This follows from the fact that the coefficients
    $S (\ldots)$, $S^\mathrm{sd} (\ldots)$
    respect composition, which was proved in
    \cref{lemma-s-comp}.
\end{proof}

\begin{para}[Generic stability conditions]
    \label{para-generic-stability}
    Following \textcite[Conjecture~6.12]{joyce-song-2012},
    we say that a stability condition~$\tau$ as above is \emph{generic},
    if for any $\alpha, \beta \in K (Y)$
    with $Z (\alpha) = \lambda Z (\beta) \neq 0$
    for some $\lambda \in \mathbb{R}_{> 0}$,
    we have the numerical condition
    $\vdim \bar{\mathcal{X}}_{\alpha, \beta}^+ = 0$.

    Similarly, when~$\tau$ is self-dual,
    we say that it is \emph{generic} as a self-dual stability condition,
    if it is generic as above,
    and for any $\alpha \in K (Y)$ of phase~$0$
    and $\theta \in K^\mathrm{sd} (Y)$,
    we have
    $\vdim \bar{\mathcal{X}}_{\alpha, \theta}^{\smash{\mathrm{sd}, +}} = 0$.

    By the first part in the proof of \cref{thm-threefold-wcf},
    combined with \cref{cor-wcf-symmetric},
    we see that if~$\tau \in \mathrm{Stab}^{\circ, \mathrm{sd}} (Y)$
    is generic, then for each class~$\alpha$ or~$\theta$ as in
    \cref{thm-threefold-wcf},
    there exists $\delta > 0$,
    such that the invariant
    $\mathrm{DT}_{\alpha} (\tau)$ or $\mathrm{DT}^\mathrm{sd}_{\theta} (\tau)$
    does not change if we move~$\tau$
    inside its $\delta$-neighbourhood.
    Moreover, this also holds for the motivic versions
    $\mathrm{DT}_{\alpha}^\mathrm{mot} (\tau)$ or
    $\mathrm{DT}^{\smash{\mathrm{mot,sd}}}_{\theta} (\tau)$,
    where the self-dual version requires a self-dual orientation data.
\end{para}

\begin{para}[Expectations on deformation invariance]
    We expect that the numeric version of
    the orthosymplectic DT invariants,
    $\mathrm{DT}_\theta^\mathrm{sd} (\tau)$,
    should satisfy \emph{deformation invariance},
    analogously to
    \textcite[Corollary~5.28]{joyce-song-2012} in the linear case,
    that is, they should stay constant under
    deformations of the complex structure of the threefold~$Y$.
    However, we have not yet been able to prove this,
    as it does not seem straightforward
    to adapt the strategy of \cite{joyce-song-2012}
    using Joyce--Song pairs to our case,
    and further work is needed.

    We do not expect the motivic version,
    $\mathrm{DT}_\theta^{\smash{\mathrm{sd,mot}}} (\tau)$,
    to satisfy deformation invariance.
\end{para}

\subsection{Vafa--Witten type invariants for surfaces}
\label{subsec-vw}

\begin{para}
    We construct a motivic version of orthosymplectic analogues
    of \emph{Vafa--Witten invariants} for algebraic surfaces,
    studied by
    \textcite{tanaka-thomas-2020-vw-i,tanaka-thomas-2018-vw-ii},
    \textcite{maulik-thomas-2018-vw},
    and \textcite{thomas-2020-vw}.
    We define our invariants for surfaces~$S$ with $K_S \leq 0$.

    Our invariants count \emph{self-dual Higgs complexes} on a surface,
    which are orthosymplectic complexes defined in \cref{subsec-coh},
    equipped with Higgs fields.
    They are a generalization of \emph{$G$-Higgs bundles},
    introduced by \textcite{hitchin-1987-higgs},
    for $G = \mathrm{O} (n)$ or $\mathrm{Sp} (2n)$.

    Via the spectral construction,
    these invariants can be seen as a version of
    orthosymplectic DT invariants in
    \cref{subsec-threefolds},
    for the non-compact Calabi--Yau threefold~$K_S$,
    the total space of the canonical bundle of the surface~$S$,
    with an involution that reverses the fibre direction.
\end{para}

\begin{para}[Higgs complexes]
    \label{para-vw-setup}
    Let~$S$ be a connected, smooth, projective algebraic surface over~$\mathbb{C}$,
    and fix the data $(I, L, s, \varepsilon)$ as in \cref{para-coh-setup}
    defining a self-dual structure~$\mathbb{D}$ on~$\mathsf{Perf} (S)$.

    For an object $E \in \mathsf{Perf} (S)$,
    a \emph{Higgs field} on~$E$ is a morphism
    \begin{equation*}
        \psi \colon E \longrightarrow E \otimes K_S
    \end{equation*}
    in~$\mathsf{Perf} (S)$.
    We call such a pair $(E, \psi)$ a
    \emph{Higgs complex} on~$S$.

    A \emph{self-dual Higgs complex} is then defined as
    a fixed point of the involution
    \begin{equation*}
        (E, \psi) \longmapsto (\mathbb{D} (E), -\mathbb{D} (\psi) \otimes K_S)
    \end{equation*}
    on the $\infty$-groupoid of Higgs complexes, where
    $\mathbb{D} (\psi) \colon \mathbb{D} (E) \otimes K_S^{-1} \to \mathbb{D} (E)$.

    More concretely,
    for a self-dual object $(E, \phi) \in \mathsf{Perf} (S)^\mathrm{sd}$
    with $\mathrm{Ext}^i (E, E \otimes K_S) = 0$ for all $i < 0$,
    where $\phi \colon E \simto \mathbb{D} (E)$,
    a self-dual Higgs field on $(E, \phi)$
    is the same data as a Higgs field
    $\psi \colon E \to E \otimes K_S$
    such that
    $(\phi \otimes K_S) \circ \psi = -(\mathbb{D} (\psi) \otimes K_S) \circ \phi$
    as morphisms $E \to \mathbb{D} (E) \otimes K_S$.
\end{para}

\begin{para}[Moduli stacks]
    Let $\bar{\mathcal{Y}}$ be the derived moduli stack
    of perfect complexes on~$S$.
    Let $\bar{\mathcal{X}} = \mathrm{T}^* [-1] \, \bar{\mathcal{Y}}$
    be the $(-1)$-shifted cotangent stack of~$\bar{\mathcal{Y}}$,
    equipped with the canonical $(-1)$-shifted symplectic structure.

    Then~$\bar{\mathcal{X}}$
    is a moduli stack of Higgs complexes on~$S$,
    since at a $\mathbb{C}$-point $E \in \bar{\mathcal{Y}} (\mathbb{C})$,
    we have
    \begin{equation*}
        \mathbb{L}_{\bar{\mathcal{Y}}} [-1] |_E \simeq
        \mathbb{R} \mathrm{Hom}_S (E, E)^\vee [-2] \simeq
        \mathbb{R} \mathrm{Hom}_S (E, E \otimes K_S) \ ,
    \end{equation*}
    parametrizing Higgs fields on~$E$.

    The self-dual structure on~$\mathrm{Perf} (S)$
    determines a $\mathbb{Z}_2$-action on~$\bar{\mathcal{Y}}$,
    which induces a $\mathbb{Z}_2$-action on~$\bar{\mathcal{X}}$.
    We have $\bar{\mathcal{X}}^\mathrm{sd} \simeq
    \mathrm{T}^* [-1] \, \bar{\mathcal{Y}}^\mathrm{sd}$,
    giving $\bar{\mathcal{X}}^\mathrm{sd}$
    a canonical $(-1)$-shifted symplectic structure.

    We regard~$\bar{\mathcal{X}}^\mathrm{sd}$
    as a moduli stack of self-dual Higgs complexes on~$S$.
    This description agrees with
    the definition of a self-dual Higgs field,
    as the $(-1)$-shifted
    tangent map of the involution~$\mathbb{D}$,
    as a map $\mathbb{R} \mathrm{Hom} (E, E) \simto
    \mathbb{R} \mathrm{Hom} (\mathbb{D} (E), \mathbb{D} (E))$,
    is given by $\psi \mapsto -\mathbb{D} (\psi)$.
\end{para}

\begin{para}[Stability conditions]
    \label{para-vw-stab}
    We now restrict to the case when
    the anti-canonical bundle~$K_S^{-1}$ of~$S$
    is ample or trivial,
    so that~$S$ is either a del~Pezzo surface,
    a K3 surface, or an abelian surface.
    We abbreviate this condition as $K_S \leq 0$.

    In this case,
    for any $\tau \in \mathrm{Stab}^\circ (S)$
    and any $E \in \mathsf{Perf} (S)$,
    every Higgs field $\psi \colon E \to E \otimes K_S$
    respects the $\tau$-HN filtration of~$E$,
    since choosing a non-zero map
    $\xi \colon K_S \to \mathcal{O}_S$,
    the composition
    $\xi \circ \psi \colon E \to
    E \otimes K_S \to E$
    must preserve the HN filtration.
    Therefore, heuristically,
    a Higgs complex $(E, \psi)$ is $\tau$-semistable
    if and only if $E$ is $\tau$-semistable.
    This justifies the following series of definitions:

    For each interval $J \subset \mathbb{R}$
    of length $|J| < 1$,
    let $\mathcal{Y} (\tau; J) \subset \bar{\mathcal{Y}}$
    be the open substack of objects in~$\mathcal{P} (J)$,
    and let
    $\mathcal{X} (\tau; J)
    = \mathrm{T}^* [-1] \, \mathcal{Y} (\tau; J)
    \subset \bar{\mathcal{X}}$
    be the corresponding open substack.
    The stacks $\mathcal{X} (\tau; J)$ and $\mathcal{Y} (\tau; J)$
    are derived linear stacks.
    When $J = -J$, they are also self-dual derived linear stacks,
    and we have
    $\mathcal{X} (\tau; J)^\mathrm{sd} \simeq
    \mathrm{T}^* [-1] \, \mathcal{Y} (\tau; J)^\mathrm{sd}$.

    Moreover, $\tau$ defines permissible stability conditions
    on $\mathcal{X} (\tau; J)$ and $\mathcal{Y} (\tau; J)$,
    in the sense of \cref{para-sd-linear-stack-stability}.
    Here, the $\Theta$-stratification on~$\mathcal{X} (\tau; J)$
    can be obtained by following the proof of
    \textcite[Theorem~6.5.3]{halpern-leistner-instability},
    similarly to
    \cref{para-bridgeland-openness},
    where the conditions~(S) and~(B)
    follow from the respective properties of~$\mathcal{Y} (\tau; J)$.
\end{para}

\begin{para}[Invariants]
    Suppose $K_S \leq 0$ and $\tau \in \mathrm{Stab}^{\circ, \mathrm{sd}} (S)$.
    For a class $\alpha \in K (S)$ with $Z (\alpha) \neq 0$
    or $\theta \in K^\mathrm{sd} (S)$ of phase~$0$,
    define the \emph{Vafa--Witten type invariants}
    \begin{equation*}
        \mathrm{vw}_\alpha (\tau) \in \mathbb{Q} \ ,
        \qquad
        \mathrm{vw}^\mathrm{sd}_\theta (\tau) \in \mathbb{Q}
    \end{equation*}
    counting semistable Higgs complexes of class~$\alpha$
    or semistable self-dual Higgs complexes of class~$\theta$,
    as the DT invariants in \cref{subsec-dt}
    for the $(-1)$-shifted symplectic linear stack~$\mathcal{X} (\tau; t)$
    with the trivial stability condition,
    where $t \in \mathbb{R}$
    is a phase of~$Z (\alpha)$ or $t = 0$ for~$\theta$.

    Moreover, since~$\mathcal{X} (\tau; t)$
    and~$\mathcal{X} (\tau; 0)^\mathrm{sd}$
    are $(-1)$-shifted cotangent stacks,
    they come with canonical orientations,
    which define an orientation data on~$\mathcal{X} (\tau; t)$
    and a self-dual orientation data on~$\mathcal{X} (\tau; 0)$.
    We use them to define \emph{motivic Vafa--Witten type invariants}
    \begin{equation*}
        \mathrm{vw}^\mathrm{mot}_\alpha (\tau) \ ,
        \qquad
        \mathrm{vw}^{\smash{\mathrm{mot,sd}}}_\theta (\tau)
        \in \hat{\mathbb{M}}^\mathrm{mon} (\mathbb{C}; \mathbb{Q}) \ .
    \end{equation*}
\end{para}

\begin{para}[Wall-crossing]
    \label[theorem]{thm-vw-wcf}
    We have the following theorem
    stating the wall-crossing formulae
    for our Vafa--Witten invariants,
    which can be proved analogously to
    \cref{thm-threefold-wcf}.
\end{para}

\begin{theorem*}
    Let $S$ be a surface with $K_S \leq 0$,
    and choose the data $(I, L, s, \varepsilon)$
    as in \cref{para-vw-setup}.
    Let $\tau, \tilde{\tau} \in \mathrm{Stab}^\circ (S)$
    be Bridgeland stability conditions.

    \begin{enumerate}
        \item
            \label{item-vw-linear}
            If $\tau, \tilde{\tau}$
            can be connected by a path of length~$< 1/4$
            in $\mathrm{Stab}^\circ (S)$,
            then for any class $\alpha \in K (S)$ with
            $Z (\alpha) \neq 0$,
            the wall-crossing formulae
            \cref{eq-wcf-dt,eq-wcf-dt-mot}
            hold for the invariants
            $\mathrm{vw}_\alpha (-)$,
            $\mathrm{vw}^\mathrm{mot}_\alpha (-)$
            when changing between $\tau$ and~$\tilde{\tau}$.

        \item
            \label{item-vw-sd}
            If $\tau, \tilde{\tau} \in \mathrm{Stab}^{\circ, \mathrm{sd}} (S)$,
            and they can be connected by a path of length~$< 1/4$
            in $\mathrm{Stab}^{\circ, \mathrm{sd}} (S)$,
            then for any class $\theta \in K^\mathrm{sd} (S)$ with
            $Z (\theta) \in \mathbb{R}_{> 0}$,
            the wall-crossing formulae
            \cref{eq-wcf-dt-sd,eq-wcf-dt-mot-sd}
            hold for the invariants
            $\mathrm{vw}^\mathrm{sd}_\theta (-)$,
            $\mathrm{vw}^{\smash{\mathrm{mot,sd}}}_\theta (-)$
            when changing between $\tau$ and~$\tilde{\tau}$.
    \end{enumerate}
    Here, the precise formulations of the wall-crossing formulae
    are as in \cref{thm-threefold-wcf}.
\end{theorem*}

\begin{para}[The case of K3 surfaces]
    Suppose that~$S$ is a K3 surface or an abelian surface.
    Then for any $E, F \in \mathsf{Perf} (S)$,
    we have the numerical relations
    \begin{equation*}
        \operatorname{rk} \mathrm{Ext}_S^\bullet (E, F)
        = \operatorname{rk} \mathrm{Ext}_S^\bullet (F, E) \ ,
        \qquad
        \operatorname{rk} \mathrm{Ext}_S^\bullet (E, \mathbb{D} (E))^{\mathbb{Z}_2}
        = \operatorname{rk} \mathrm{Ext}_S^\bullet (\mathbb{D} (E), E)^{\mathbb{Z}_2} \ ,
    \end{equation*}
    where `$\operatorname{rk}$'
    denotes the alternating sum of dimensions,
    and $(-)^{\mathbb{Z}_2}$ denotes the fixed part of the involution
    $\phi \mapsto \mathbb{D} (\phi)$.
    These relations imply that
    $\bar{\mathcal{X}}$ and $\bar{\mathcal{X}}^\mathrm{sd}$
    are numerically symmetric in the sense of \cref{para-symmetric-stacks}.
    By \cref{cor-wcf-symmetric}
    and \cref{thm-vw-wcf}, the invariants
    $\mathrm{vw}_\alpha (-)$,
    $\mathrm{vw}^\mathrm{sd}_\theta (-)$,
    $\mathrm{vw}^\mathrm{mot}_\alpha (-)$,
    and $\mathrm{vw}^{\smash{\mathrm{mot,sd}}}_\theta (-)$
    are locally constant functions on
    $\mathrm{Stab}^\circ (S)$ or
    $\mathrm{Stab}^{\circ, \mathrm{sd}} (S)$.
\end{para}

\phantomsection
\addcontentsline{toc}{section}{References}
\sloppy
\setstretch{1.1}
\renewcommand*{\bibfont}{\normalfont\small}
\printbibliography

\par\noindent\rule{0.38\textwidth}{0.4pt}
{\par\noindent\small
\hspace*{2em}Chenjing Bu\qquad
\texttt{bu@maths.ox.ac.uk}
\\[-2pt]
\hspace*{2em}Mathematical Institute, University of Oxford, Oxford OX2 6GG, United Kingdom.}

\end{document}